\documentclass[final,11pt]{article}
\usepackage[utf8]{inputenc}
\usepackage{anysize}
\marginsize{3cm}{3cm}{2cm}{2cm} 

\usepackage{fancyhdr} 
\usepackage{lipsum} 
\usepackage{color} 
\usepackage{bm} 
\usepackage{amsfonts,amsthm,amsmath,amssymb} 
\usepackage{latexsym,graphicx,epstopdf} 
\usepackage{authblk} 
\usepackage{hyperref}

\usepackage[left,pagewise,mathlines]{lineno}

\renewcommand{\a}{{\bf a}}
\renewcommand{\b}{{\bf b}}

\newcommand{\f}{{\bf f}}
\newcommand{\g}{{\bf g}}

\newcommand{\n}{{\bf n}}

\renewcommand{\t}{{\bf t}}
\renewcommand{\u}{{\bf u}}
\renewcommand{\v}{{\bf v}}
\newcommand{\w}{{\bf w}}
\newcommand{\x}{{\bf x}}

\newcommand{\zero}{{\bf 0}}

\newcommand{\bb}{{\bf B}}
\newcommand{\cc}{{\bf C}}

\newcommand{\hh}{{\bf H}}
\newcommand{\ii}{{\bf I}}

\renewcommand{\ll}{{\bf L}}
\newcommand{\mm}{{\bf M}}

\newcommand{\pp}{{\bf P}}

\newcommand{\rr}{{\bf R}}

\renewcommand{\tt}{{\bf T}}

\newcommand{\zz}{{\bf Z}}

\newcommand{\nnn}{\mathbb{N}}

\newcommand{\rrr}{\mathbb{R}}

\newcommand{\ttt}{\mathbb{T}}

\newcommand{\btau}{{\bm\tau}}

\newcommand{\bsigma}{{\bm\sigma}}

\newcommand{\bvarepsilon}{{\bm\varepsilon}}

\renewcommand{\div}{{\rm div}}
\newcommand{\bfdiv}{{\bf div}}
\newcommand{\curl}{{\rm curl}}

\newcommand{\tr}{{\rm tr}}
\newcommand{\transpose}{\texttt{t}}
\newcommand{\deviator}{\texttt{d}}
\newcommand{\gen}{{\rm span}}

\newtheorem{theorem}{Theorem}

\title{Revisiting the Jones eigenvalue problem in fluid-structure interaction\thanks{This work was partially supported by CONICYT-Chile, through Becas 
Chile, and NSERC through the Discovery program of Canada.}}

\author[1]{Sebasti\'an Dom\'inguez\thanks{Corresponding author: domingue@sfu.ca}}
\author[1]{Nilima Nigam}
\author[2]{Jiguang Sun}

\affil[1]{Department of Mathematics, Simon Fraser University, Canada}
\affil[2]{Department of Mathematical Sciences, Michigan Technological University, USA}

\begin{document}

\maketitle

\begin{abstract}
The Jones eigenvalue problem, first described in \cite{ref:jones1983}, concerns unusual vibrational modes in bounded elastic 
bodies: 
time-harmonic displacements whose tractions and normal components are both identically zero on the boundary. This 
problem is usually associated with a lack of unique solvability for certain models of fluid-structure interaction. The 
boundary conditions in this problem appear, at first glance, to rule out {\it any} non-trivial modes unless the domain 
possesses significant geometric symmetries. Indeed, Jones modes were shown to not be possible in most $C^\infty$ 
domains 
in \cite{ref:harge1990}. However, we {show} in this paper that while the existence of Jones modes sensitively 
depends on 
the domain geometry, such modes {\it do} exist in a broad class of domains.   This paper presents the first detailed 
theoretical and computational investigation of this eigenvalue problem in Lipschitz domains. We also analytically 
demonstrate Jones modes on some simple geometries.
\end{abstract}


{\bf Keywords}: fluid-structure interaction, Jones eigenvalue problem, finite element method
\vspace{.25cm}

{\bf AMS subject classifications}: 65N25, 65N30, 74B05

\section{Introduction}\label{sec:Introduction}
In this paper we investigate the {\it Jones eigenvalue problem}, which is to locate non-trivial {vector fields 
$\u$ in 
$\rrr^n $} ($n\in\{2,3\}$), and {scalars} $w\in \rrr$ so that
\begin{subequations}\label{Introjones}
\begin{align}
\mathcal{L}\u:= - \left(\mu\Delta\u + (\lambda + \mu)\nabla({\rm div}\u)\right) &= {\rho w^2\u} & \text{in 
$\Omega$,}\label{eq:introjones1}\\
{\big(\mu \nabla\u  {+} (\lambda+\mu)(\div\,\u)\ii\big)\n}& = {\bf 0} & \text{on 
$\partial\Omega$,}\label{eq:introjones2}\\
\u\cdot\n &=0&
\text{on 
$\partial\Omega$.}\label{eq:introjones3}
\end{align}
\end{subequations}

Here $\Omega \subseteq \mathbb{R}^n$ is a bounded domain with Lipschitz boundary $\partial \Omega$, $\n$ is the unit 
outer normal on $\partial \Omega$, $\rho>0$ is a {\it density} and $\mu>0$, $\lambda \in \rrr$  such that $ \lambda + 
\left(\frac{2}{n}\right)\mu>0$ are the so-called {\it Lam\'e parameters}.

The Jones eigenvalue problem arises  while studying  time-harmonic solutions of a fluid-solid interaction 
problem in $\mathbb{R}^n$. Precisely, suppose an isotropic elastic bounded body occupying {the} region $\Omega \subset \mathbb{R}^n$ is 
immersed in an inviscid compressible fluid 
occupying the 
rest of the space. The {solutions of the} Jones eigenvalue problem coincides exactly with the determination of non-trivial solutions of the corresponding homogeneous equations governing the displacements of the elastic body. The occurrence of these 
eigenpairs was first noticed in \cite{ref:jones1983}, where the author introduced the 
{fluid-structure interaction problem of interest and 
pointed out its lack of uniqueness}. Many other authors have {also} noticed the 
non-uniqueness issue in this model \cite{ref:barucq2014, MR3200273, MR2566767, ref:hsiao2000, ref:hsiao2017, luke1995, ref:meddahi2013, MR2169165}. In these papers the main interest was in studying the full fluid-structure problem, and the Jones eigenmodes were of interest only within the context of well-posedness, which was only guaranteed away from such modes. We note this is not the only possible model for fluid-structure interaction in the frequency domain; other models which ameliorate the breakdown of uniqueness at exceptional frequencies have been proposed. We discuss this later {in} \autoref{sec:problem}. {Nonetheless, in  \cite{ref:estecahandy2013}, it was found that Jones eigenpairs may pollute numerical approximations in both the solid and the fluid. Such phenomenon was later confirmed in \cite{ref:azpir2018}. This shows the importance of 
identifying eigenpairs of \autoref{Introjones} in order to obtain suitable numerical methods to handle the 
corresponding fluid-structure problem.}

Our focus in this paper is the eigenvalue problem \autoref{Introjones}. We notice that \autoref{eq:introjones1} and \autoref{eq:introjones2} together define a standard eigenvalue 
problem (we call this the {\it traction eigenvalue problem}) for the Lam\'e operator $\mathcal{L}$ on $\Omega$, 
analogous to the Neumann eigenvalue problem for the Laplacian. The traction eigenvalue problem has been extensively studied and has numerous applications in mechanical engineering;  the existence of a countable discrete spectrum for Lipschitz domains is well-established (see, e.g. \cite{ref:knops1971}). However, the problem under consideration in the present article 
asks: 
do there exist traction eigenmodes which additionally satisfy \autoref{eq:introjones3}? This constraint intimately 
couples the geometry of the domain with the Jones eigenmode; in essence, the only traction modes which are also Jones 
modes are those which are purely tangential to the boundary.

Not much is known about the Jones eigenvalue problem itself. As mentioned, the most intriguing feature of this 
problem is its dependence on the boundary of the domain. An influential paper \cite{ref:harge1990} showed that for 
almost any 3D domain with $C^\infty$ boundary, there can be no modes with free traction and zero normal component on 
the 
boundary that solve the Jones eigenvalue problem. The central claim in this paper was established in a fairly narrow 
setting - for instance, the analysis cannot extend to domains with corners - yet perhaps the main theorem served to 
deter further investigations. Likewise, \cite{ref:natroshvili2005} showed that smooth 3D domains having two flat 
non-parallel manifolds of the boundary cannot support a non-trivial divergence-free  Jones mode. Even though 
the authors claim these kind of deformations are Jones eigenvectors, we note that the full eigenproblem 
\autoref{eq:introjones1} does not impose the condition on the divergence. 

The rest of this paper is organized as follows: in \autoref{sec:problem} we introduce the eigenvalue problem. We first 
describe the fluid-solid interaction where the Jones modes appear. We provide exact Jones eigenmodes on rectangles. We 
next provide a detailed description of the point spectrum of this problem and identify important properties relating 
the 
eigenpairs with the domain. In \autoref{sec:formulation} we derive a primal formulation to approximate Jones eigenpairs 
where the extra constraint on the displacement in the normal direction on the boundary has been introduced as an 
essential condition in the search and test spaces. The continuity of the normal trace will  ensure this space is 
closed. 
A careful treatment of this formulation is then provided as it is known that the spectrum of this problem depends 
heavily on the geometry of the domain \cite{ref:harge1990,ref:natroshvili2005}. In fact, the proof of the usual 
ellipticity of one of the bilinear forms depends entirely on a Korn's inequality shown in 
\cite{ref:bauer2016} for 
Lipschitz domains in $\rrr^n$, with $n\in\{2,3\}$. A weaker version of this result for domains with $C^1$ boundary is 
given in \cite{ref:desvillettes2002}. In addition, in \cite{ref:bauer2016} the authors showed that the gradient of a 
vector with mixed tangential and/or normal components vanishing on the boundary can be bounded (up to a constant) above 
by the deviatoric part of its strain tensor in concave or polyhedral domains in $\rrr^3$ with piecewise $C^2$ 
boundaries 
(as defined in the same reference). In \autoref{sec:discrete} we {use} a conforming discretization of the continuous 
eigenvalue problem via Lagrange finite elements. The sensitivity of the spectrum to the shape of the domain suggests that the classical FEM using triangular meshes may 
{not be} the best method to use to approximate the spectrum of this problem for curved domains. Numerical examples presented in \autoref{sec:numerics} show the performance of this classical scheme and exhibit the different regularity of the eigenfunctions of this problem in different domains.

\section{The fluid-structure interaction problem}\label{sec:problem}
Solutions of the Jones eigenvalue problem appear as non-trivial elements in the kernel of a model of fluid-solid 
interaction where 
an isotropic elastic body is immersed in a compressible inviscid fluid occupying the whole space $\rrr^n$, $n\in\{2,3\}$. In this 
section we introduce such problem and establish its connection with the Jones eigenpairs.

\subsection{Some notation}
We begin by fixing the notation for the remainder of this paper. For vectors in $\rrr^n$, the operation $\a\cdot\b$ is the standard dot product with induced norm $\|\cdot\|$. 
For second-order tensors $\bsigma,\,\btau$ in $\rrr^{n\times n}$, the double dot product is the usual Frobenius inner product for matrices 
\begin{align*}
 \bsigma:\btau:=\sum_{i,j=1}^n \btau_{ij}\bsigma_{ij} =  \tr(\btau^\transpose\bsigma).
\end{align*}
This inner product induces 
the usual Frobenius norm.
 For differential operators, 
$\nabla$ denotes the usual gradient operator {acting on} either a scalar field or a vector field. The divergence 
operator 
``\div'' of a vector field reduces to the trace of its gradient. The operator ``\bfdiv'' acting on tensors stands for 
the usual divergence operator applied to each row of the tensors. The curl operator $\curl$ of a vector field is defined as usual in the 3D case. For vector fields $\u$ in 2D, $\curl\,\u$ reduces to a vector that only points out of the plane. The deviatoric part of a tensor of $\btau$ is 
$\btau^{\deviator} := 
\btau - \frac{1}{n}\tr(\btau)\ii$, where $\ii$ is the identity matrix of $n\times n$ entries. If $\btau, \bsigma$ are 
 second-order tensors whose entries are $L^2(\Omega)$ functions on a bounded domain $\Omega$, we define 
\begin{align*}
 (\btau,\bsigma)_0:=\int_{\Omega} \btau:\bsigma \, d\Omega.
\end{align*} We observe that
\begin{align*}
	\|\btau^\deviator\|_0^2=\|\btau\|_0^2-\frac{2}{n}\left(\tr(\btau)\ii,\btau\right)_0+ 
\frac{1}{n^2}(\tr(\btau)\ii,\tr(\btau)\ii)_0=\|\btau\|^2_0-\frac{1}{n}\|\tr(\btau)\|_0^2.
\end{align*}
If {$\u$ is a differentiable vector field in $\rrr^n$}, the strain {tensor} is a symmetric second-order tensor
\begin{align*}
 \bvarepsilon(\u):=\frac{1}{2}\left(\nabla \u + (\nabla \u)^\transpose\right).
\end{align*}


In what follows we need to identify domains which  are {\it axisymmetric}. We employ the definition given in 
\cite{ref:bauer2016}: The domain $\Omega$ is axisymmetric if it is invariant under rotations about an axis of symmetry. 
With this definition one can see that in the 2D case the {disk} and its complement are the only axisymmetric domains. For the 3D case,  the number of axisymmetric domains becomes a lot larger. Any solid 
of rotation is axisymmetric, and has circular cross-section transverse to the axis of rotation.

\subsection{A model of fluid-solid interaction}
As discussed in \autoref{sec:Introduction}, the Jones eigenproblem was originally described within the context of a 
fluid-structure interaction problem. 
Consider a bounded, simply connected domain 
$\Omega_s\subseteq\mathbb{R}^n$  representing an isotropic linearly elastic body in $\mathbb{R}^n$. This body is assumed to be
immersed in a compressible inviscid fluid occupying the region $\Omega_f := \mathbb{R}^n\backslash\bar{\Omega}_s$. See 
\autoref{fig:schematic} for a schematic of this situation.
\graphicspath{{./images/}}
\begin{figure}[!ht]
\centering
\includegraphics[width = .75\textwidth, height=0.25\textheight]{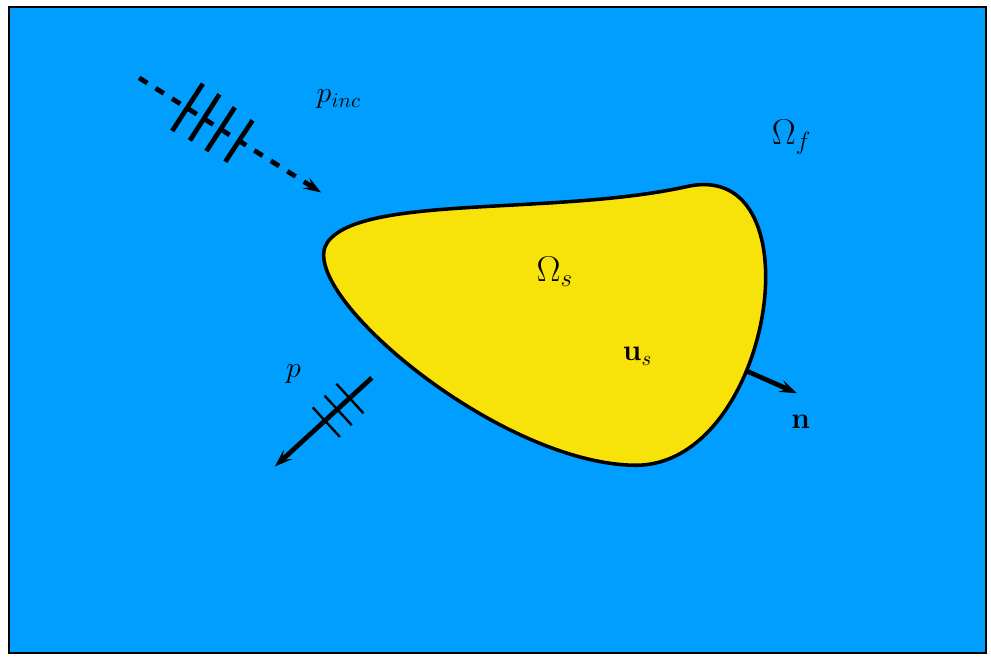}
\caption{Schematic of {the} fluid-structure interaction problem.}
\label{fig:schematic}
\end{figure}
Note that $\partial\Omega_s$, the bounded {component} of the boundary of $\Omega_f$ coincides with the boundary of the {elastic body} $\Omega_s$. {We  denote} 
$\Gamma := \partial\Omega_f = \partial\Omega_s$.

The parameters describing the elastic properties of $\Omega_s$ are the so-called Lam\'e constants $\mu>0$ and 
$\lambda\in\mathbb{R}$, satisfying the condition
\begin{align}
 \lambda+\left(\frac{2}{n}\right)\mu>0.\label{eq:lamecondition}
\end{align}
One fluid-structure interaction problem of interest concerns the situation when the fields are time-harmonic, allowing 
us to factor out the time-dependence and consider the problem in the frequency domain. Using standard  interface 
conditions coupling the pressure in the fluid $p$ and the elastic displacement in the solid $\u$, the fluid-solid 
interaction problem in the frequency domain reads: given volumetric forces $\f$ and $\g$, and an 
incident pressure $p_{\rm inc}$, find a pressure field $p$ in $\Omega_f$ and elastic deformations $\u$ of 
$\Omega_s$ 
satisfying
\begin{subequations}\label{eq:coupled-problem}
 \begin{align}
 &\Delta p + \left(\frac{w^2}{c^2}\right)p = \div\,\f,\,\,\text{on $\Omega_f$},\quad-\rho w^2\u 
- \bfdiv\,\bm\sigma(\u) = \g,\,\,\text{in $\Omega_s$},\label{eq:fluid-solid-problem1}\\
 &-(p + p_{\rm inc})\n = \bm\sigma(\u)\n,\quad \frac{\partial}{\partial\n}( p + p_{\rm inc}) = 
\rho w^2\u\cdot\n,\,\,\text{on $\Gamma$},\label{eq:fluid-solid-problem2}\\
 &\frac{\partial p}{\partial r} - i\left(\frac{w}{c}\right) p = o(1/r),\,\,\text{as 
$r:=\|\x\|\to\infty$}.\label{eq:fluid-solid-problem3}
\end{align}
\end{subequations}
The parameter {$c$} is the constant speed of the sound in the fluid, and the Cauchy tensor $\bm\sigma$ 
depends on the Lam\'e constants $\mu>0$ 
and $\lambda\in\rrr$ and is defined in terms of the strain tensor $\bm \varepsilon(\u)$ as
\begin{align*}
	\bm\sigma(\u) := 2\mu\bm\varepsilon(\u) + \lambda\,\tr(\bm\varepsilon(\u))\ii\quad\text{in $\Omega_s$.}
\end{align*}
Using the vector Laplacian operator, we 
see that
\begin{align*}
	{\rm\bf div}\,\bm\sigma(\u) = \mu\Delta\u + (\lambda + \mu)\nabla({\rm div}\,\u)\quad\text{on 
$\Omega_s$}.
\end{align*}

This is a commonly accepted formulation for time-harmonic fluid-solid interaction problems involving inviscid flow, 
see, for example, \cite{ref:hsiao2000, ref:hsiao2017, ref:huttunen2008}. The system in 
\autoref{eq:coupled-problem} is known to possess a non-trivial kernel under certain situations. As discussed in 
\cite{ref:jones1983}, this problem lacks a unique solution whenever $\u$ is a non-trivial solution of the homogeneous 
problem:
\begin{align}
 -\bfdiv\,\bm\sigma(\u) = \rho w^2\u,\,\,\text{in $\Omega_s$},\quad\bm\sigma(\u)\n = {\bf 
0},\quad\u\cdot\n = 0,\,\,\text{on $\Gamma$}.\label{eq:jones-modes}
\end{align}
The pair $(w^2,\u)$ solving this eigenvalue problem is a {\it Jones eigenpair} \cite{ref:jones1983}. The homogeneous problem 
for the 
displacements can be viewed as the usual 
eigenvalue problem for linear elasticity with traction free boundary condition, plus the extra constraint on the normal 
trace of $\u$ along the boundary. Therefore, we may consider this as an overdetermined problem. We know that there is a 
countable 
number of eigenmodes for linear 
elasticity 
with free traction given reasonable assumptions on $\Gamma$  (see \cite{ref:babuskaosborn1991} and references therein). 
The extra 
condition  $\u\cdot \n=0$ on the boundary  plays an important role in the existence of the zero eigenvalue of 
\autoref{eq:jones-modes}. All of these 
properties are {discussed} in detailed in the next sections.

{We note that other authors have addressed the lack of uniqueness in \autoref{eq:coupled-problem}.} A slightly different model for fluid-structure interaction in the frequency domain can be derived by considering the 
problem with non-zero fluid viscosity and then  taking {the viscosity} to zero. As pointed out in \cite{ref:hsiao2000},{ in this setting it is reasonable to 
adding another condition on the shear of $\u$ on the interface as a fix for the non-uniqueness of }
\autoref{eq:coupled-problem}. The condition
\begin{align*}
 \rho w^2\u\cdot\t = \frac{\partial}{\partial\t}(\bsigma(\u)\n\cdot\n),
\end{align*}
removes the non-zero solutions of \autoref{eq:jones-modes}. Here $\t$ is the unit tangent vector on $\Gamma$.  In 
\cite{ref:gatica2009}, the authors add a Robin boundary condition for the fluid pressure on a far enough 
``artificial'' boundary containing the solid. They then consider the fluid to be bounded between the solid and this 
interface. As shown in the same reference, the modified problem has a unique solution.

Since our interest in the present paper concerns the eigenvalue problem 
\autoref{Introjones}, we do not delve any further into the properties of the interaction problem (cf. 
\autoref{eq:coupled-problem}).

\subsection{Lipschitz domains can support Jones modes}\label{subsec:harge}
The paper by Harg\'e in 1990 \cite{ref:harge1990} examined the existence of non-trivial solutions of 
\autoref{Introjones}. The results of this paper have been cited extensively in subsequent works focusing on the 
well-posedness of the fluid-structure model presented in the previous section. As an instance, {\it ``Fortunately, 
these traction-free oscillations occur only in highly specific situations..."} \cite{hsiao2003};
{\it ``Note that Harg\'e [...] has established that Jones modes do not exist for arbitrarily-shaped bodies.}"
\cite{ref:barucq2014}; {\it ''However, intuitively, we do not expect Jones frequencies to
exist for an "arbitrary" body; this has been proved recently by Harg\'e [...]".} \cite{martin14} These papers also 
note that domains which are axisymmetric may indeed have such modes. As a historical aside, Horace Lamb \cite{ref:lamb} 
documented such 
modes in the sphere in 1881.

Revisiting \cite{ref:harge1990}, we note that the setting of the paper is as follows: 
\begin{quote}
{\it Pour $\Omega$ {ouvert} born\'e connexe de $\mathbb{R}^3$ \'a bord $C^\infty$  ... fix\'e et soit $
E = \Big\{\phi\in C^{\infty}(\overline{\Omega};\mathbb{R}^3);\, \phi \mbox{ diff\'eomorphisme de $\overline{\Omega}$ 
sur } \phi(\overline{\Omega})\Big\}$; on munit $E$ de la topologie $C^\infty$ ...}\cite{ref:harge1990}
\end{quote}
and the main theorem is then
\begin{quote}
{\it There is {an open}, countable dense intersection $G$ of open sets of $E$ such that for any $\phi$ in $G$, 
there 
	is no exceptional eigenvalue ...}
\end{quote}
This theorem and its technique of proof cannot be directly applied to the situation of polygonal domains in 
$\mathbb{R}^2$, nor to polyhedral domains. Intuitively one may believe the result should hold in polygonal or 
polyhedral 
domains; indeed, our initial starting point for the current study was to try to extend the result of Harg\'e to general Lipschitz domains. The critical observation was the following example.
It is easy to verify by inspection that $(w_s^2,\u_s), 
(w_p^2,\u_p)$ defined below are Jones eigenpairs on the rectangle $[0,a]\times[0,b]$:
\begin{subequations}\label{eq:smode}
\begin{align}
&\u_s := \left((a\ell)\sin\Big(\frac{m\pi x}{a}\Big)\cos\Big(\frac{\ell\pi y}{b}\Big), -(bm)
\cos\Big(\frac{m\pi x}{a}\Big)\sin\Big(\frac{\ell\pi y}{b}\Big)\right)^\transpose,\\ 
&w_s^2 := \frac{\mu\pi^2}{\rho}\left(\frac{m^2}{a^2}+\frac{\ell^2}{b^2}\right),\quad m,\ell = 
1,2,\ldots,
\end{align}
\end{subequations}
and 
\begin{subequations}\label{eq:pmode}
\begin{align}
&\u_p := \left((bm) \sin\Big(\frac{m\pi x}{a}\Big)\cos\Big(\frac{\ell\pi y}{b}\Big), (a\ell)\cos\Big(\frac{m\pi 
x}{a}\Big)\sin\Big(\frac{\ell\pi y}{b}\Big)\right)^\transpose,\\
&w_p^2 := \frac{(\lambda + 2\mu)\pi^2}{\rho}\left(\frac{m^2}{a^2}+\frac{\ell^2}{b^2}\right),\quad m,\ell = 
0,1,\ldots,\quad m+\ell >0.
\end{align}
\end{subequations}
It can also be readily seen that ${\div\,\u_s} = 0$ and $\curl\,\u_p = 
0$; 
eigenmodes of this form are termed 
$s$- and $p$- modes respectively. Further, some eigenvalues may have geometric multiplicity depending on 
{$a$} and $b$. 
In case $\lambda,\mu, \frac{a^2}{b^2}$ are integers, the value $w^2$ can be a higher-multiplicity Jones eigenvalue with an associated eigenspace which includes both $s$- 
and $p$- modes, 
provided we can find integer pairs $(m,\ell)$ and $(n,k)$ 
satisfying
$$ \frac{\mu \pi^2}{\rho}\left(\frac{m^2}{a^2}+\frac{\ell^2}{b^2}\right) = \frac{(\lambda+2\mu) 
\pi^2}{\rho}\left(\frac{n^2}{a^2}+\frac{k^2}{b^2}\right)=: w^2.$$

Studying this simple example it became clear that the situation for polygonal domains required a different approach, 
and would yield different results {to those described} in \cite{ref:harge1990}. One could think, for example, that under certain conditions it could be possible for domains comprising a finite union of rectangles could possess Jones modes; this is shown to be the case for the L-shaped domain in a subsequent section.

\section{Weak formulation}\label{sec:formulation}
In the presence of corners or edges, it is no longer reasonable to ask for the boundary condition in the problem \autoref{eq:jones-modes} to be 
 imposed pointwise, since the eigenfunctions may not be sufficiently regular. {In fact, the zero normal trace condition on the displacement holds almost everywhere on $\partial\Omega$; it is clear that this condition cannot be imposed, for example over vertices of the boundary.} A weak formulation is needed and later in this paper we shall compute Jones modes using a finite  element discretization.

Let $\Omega$ be a Lipschitz domain of $\mathbb{R}^n$ (we drop the subscript referring to the solid domain). Recall the 
eigenvalue problem in \autoref{eq:jones-modes}: find the Jones pairs 
$(w^2,\u)$, 
$\u$ non-zero, such that
\begin{subequations}\label{eq:jones}
 \begin{align}
 -{\rm\bf div}\,\bm\sigma(\u) &= \rho w^2\u,\quad\text{in $\Omega$},\label{eq:jones1}\\
 \bm\sigma(\u)\n = {\bf 0},&\quad \u\cdot\n = 0.\quad\text{on $\partial\Omega$}\label{eq:jones2}
\end{align}
\end{subequations}
with $\rho>0$ a fixed constant. 
Using the definition of the Cauchy stress tensor, we can write \autoref{eq:jones} as
\begin{subequations}\label{eq:equivjones}
 \begin{align}
 \mu\Delta\u + (\lambda + \mu)\nabla({\rm div}\u) +\rho w^2\u &= {\bf 0}\quad\text{in 
$\Omega$,}\label{eq:jones3}\\
\big(\mu \nabla\u + (\lambda+\mu)(\div\,\u)\ii\big)\n = {\bf 0},\,\u\cdot\n &= 0,\quad\text{on 
$\partial\Omega$.}\label{eq:jones4}
\end{align}
\end{subequations}
In order to introduce a weak formulation of \autoref{eq:jones} (equivalently \autoref{eq:equivjones}), we 
define the spaces
\begin{align*}
\hh^1(\Omega):= \Big\{\v=(v_1,\ldots,v_n):\, v_i\in H^1(\Omega)\Big\}, \quad \hh:= \Big\{\u\in \hh^1(\Omega):\,\gamma_\n\u = 
0\,\,\text{on $\partial\Omega$}\Big\}.
\end{align*} 
{Here $H^1(\Omega)$ denotes the usual Hilbert space of scalar functions in $L^2(\Omega)$ whose partial 
derivatives (in all directions) belong to $L^2(\Omega)$. The operator $\gamma_\n:\hh^1(\Omega)\to L^2(\partial\Omega)$ 
is the normal trace operator, which is linear and bounded in $\hh^1(\Omega)$. The space $\hh$ is endowed with the usual 
$\hh^1$-inner product, denoted by $(\cdot,\cdot)_1$. This implies that $\hh$ is a closed subspace of $\hh^1(\Omega)$}. 
{We shall also need the space $\hh^{t+1}(\Omega)$  of vector fields 
whose entries
belong to the Sobolev space $H^{t+1}(\Omega), t\geq 0$, and the semi-norm $|\cdot|_{t+1}$} in 
{$\hh^{t+1}(\Omega)$.}

We consider the 
following primal formulation of \autoref{eq:equivjones}: find 
$\u\in\hh$ and $\kappa \in\rrr$ such that

\begin{align}
 a(\u,\v) = \kappa\, (\u,\v)_0,\quad\forall\, \v\in \hh,\label{eq:primal1}
\end{align}
where $\kappa := \rho w^2$, and the bilinear form $a:\hh\times \hh\to \rrr$ is given by
\begin{align*}
 a(\u,\v) := \mu(\nabla\u,\nabla\v)_0 + (\lambda +\mu)({\rm div}\u,{\rm 
div}\v)_0,\quad\forall\,\u,\,\v\in \hh,
\end{align*}
Since ${\|{\rm div}\u\|_0 \leq \|\nabla\u\|_0} \leq \|\u\|_1$, for all $\u\in 
\hh^1(\Omega)$, the bilinear form 
$a(\cdot,\cdot)$ is 
bounded. In addition, $a(\cdot,\cdot)$ is symmetric and positive semi-definite. {The Rayleigh quotient gives}
\begin{align}
\frac{a(\v,\v)}{\|\v\|_0^2} \geq 0,\quad\forall\,\v\in\hh,\,\v\neq{\bf 0}.\label{eq:rayleigh}
\end{align}
We see that all possible eigenvalues of \autoref{eq:primal1} are real and non-negative. Using the Cauchy tensor 
$\bsigma$ we can write \autoref{eq:primal1} in the equivalent form
	\begin{align}
 \tilde a(\u,\v) = \kappa\, (\u,\v)_0,\quad\forall\, \v\in \hh,\label{eq:primal2}
\end{align}
where $\tilde{a}(\u,\v) := (\bm\sigma(\u),\nabla\v)_0 = (\bm\sigma(\u),\bm\varepsilon(\v))_0$ for all 
$\u,\v\in\hh.$ In terms of the strain tensor only, $\tilde a(\cdot,\cdot)$ becomes
\begin{align*}
\tilde{a}(\u,\v) = 2\mu\big(\bvarepsilon(\u),\bvarepsilon(\v)\big)_0 + \lambda 
\big(\tr(\bvarepsilon(\u)),\tr(\bvarepsilon(\v))\big)_0.
\end{align*}
Using the deviatoric part of the strain tensor we can write
\begin{align}
\tilde{a}(\u,\v) = 2\mu\big(\bvarepsilon(\u)^\deviator,\bvarepsilon(\v)^\deviator\big)_0 + \left(\lambda + 
\frac{2\mu}{n}\right)\big(\tr(\bvarepsilon(\u)),\tr(\bvarepsilon(\v))\big)_0,\quad\forall\, 
\u,\,\v\in\hh.\label{eq:aformdeviator}
\end{align}
Obviously, $a(\u,\v) = \tilde a(\u,\v)$ for all $\u,\,\v\in\hh$. Furthermore, the bilinear forms $a(\cdot,\cdot)$ and 
$\tilde a(\cdot,\cdot)$ are 
bounded in $\hh^1(\Omega)\times \hh^1(\Omega)$ and hence in $\hh\times\hh$. 
We can then define the solution operator $\tilde T:\ll^2(\Omega)\to\hh$ by $\tilde T(\f) = \u$ such that
\begin{align}
 \tilde{a}(\u,\v) = (\f,\v)_0,\quad\forall\,\v\in\hh.\label{eq:source}
\end{align}
{We immediately have that $\tilde T$ is a linear operator. However, to show all necessary properties of $\tilde T$, we need to show that $\tilde a$ (equivalent $a$) is coercive in $\hh$. In particular, this will give us the compactness of $\tilde T$ restricted to $\hh$, and therefore} we are guaranteed that $\tilde T$ has a countable and positive point spectrum with eigenfunctions lying in $\hh$. These properties will rely on the coercivity properties of $\tilde{a}(\cdot,\cdot)$, which will depend crucially on the domain shape as we shall see.
%

Using the definition of $\tilde a(\cdot,\cdot)$, for any $\u\in\hh$ we have
\begin{align*}
	\tilde a(\u,\u) =&\, 2\mu\|\bvarepsilon(\u)^\deviator\|_0^2 + \left(\lambda + 
	\frac{2\mu}{n}\right)\|\tr(\bvarepsilon(\u))\|_0^2\\
	=&\, n\Big(\frac{2\mu}{n}\|\bvarepsilon(\u)^\deviator\|_0^2 + \left(\lambda + 
	\frac{2\mu}{n}\right)\frac{1}{n}\|\tr(\bvarepsilon(\u))\|_0^2\Big)\\
	\geq&\,\min\Big\{2\mu,n\left(\lambda + \frac{2\mu}{n}\right)\Big\}\Big(\|\bvarepsilon(\u)^\deviator\|_0^2 + 
	\frac{1}{n}\|\tr(\bvarepsilon(\u))\|_0^2\Big)\\
	=&\, \min\Big\{2\mu,n\left(\lambda + 
	\frac{2\mu}{n}\right)\Big\}\|\bvarepsilon(\u)\|_0^2,
\end{align*}
where we have used  $\|\btau\|_0^2 = \|\btau^\deviator\|_0^2+\frac{1}{n}\|\tr(\btau)\|_0^2$.
This establishes the inequality
\begin{align}
	\tilde a(\u,\u)\geq&\,\min\Big\{2\mu,n\left(\lambda + 
	\frac{2\mu}{n}\right)\Big\}\|\bvarepsilon(\u)\|_0^2,\quad\forall\,\u\in\hh.\label{eq:strain-bound}
\end{align}

{Now, to show that $\tilde{a}(\cdot,\cdot)$ is coercive on $\hh$, we need to bound from below the term 
$\|\bvarepsilon(\u)\|_0$ by $\|\u\|_1$ (up to a constant). However, as we will see in the forthcoming sections, the 
intersection $\rr\mm(\Omega)\cap\hh$ depends intimately on $\Omega$ and therefore the positiveness of 
$\tilde{a}(\cdot,\cdot)$ cannot be established in case this intersection is not the trivial space.}

It turns out that for {\it non-axisymmetric Lipschitz 
domains}, $\kappa=0$ is not in the point spectrum of \autoref{eq:primal2}. We establish this in the following 
subsection.


\subsection{Existence of Jones modes on non-axisymmetric domains}
Let $\Omega$ be a non-axisymmetric domain in $\rrr^n,\,n\in\{2,3\}$. 
In  \cite{ref:bauer2016} the following version of {Korn's inequality for vector fields in $\hh$ defined on} 
non-axisymmetric 
Lipschitz 
domains was established: there is a positive constant $C$ which depends only on $\Omega$ so that
 \begin{align}
 	\|\bvarepsilon(\u)\|_0\geq C \|\u\|_1,\quad\forall\,\u\,\in\hh.\label{eq:korns}
 \end{align}
  {This is a type of Korn's inequality; we recall there are several variants of this inequality, \cite{ref:horgan1995}.}
  
 Combining the inequality above and the derived inequality for $\tilde a(\cdot,\cdot)$ in \autoref{eq:strain-bound} we obtain 
the 
coercivity of $\tilde a(\cdot,\cdot)$ in $\hh$ for 
non-axisymmetric domains:
 \begin{align}
 	\tilde a(\u,\u) \geq&\, C^2\min\Big\{2\mu,d\left(\lambda + 
 	\frac{2\mu}{n}\right)\Big\}\|\u\|_1^2,\quad\forall\,\u\in\hh.\label{eq:tildeacoercive}
 \end{align}
  Provided $\Omega$ is a non-axisymmetric Lipschitz domain, the coercivity of $\tilde a(\cdot,\cdot)$ means that the 
solution 
operator $\tilde{T}|_\hh:\hh \rightarrow \hh$ is well-defined, and satisfies
 $$ C^2\min\Big\{2\mu,d\left(\lambda +  \frac{2\mu}{n}\right)\Big\}\|\tilde T \f\|_1^2 \leq {\tilde{a}(\tilde T\f,\tilde T\f) = (\f,\tilde T\f)_0} 
\leq \|\tilde 
T\f\|_0\|\f\|_0,$$
 {for all $\f\in\hh$. This establishes the boundedness of $\tilde{T}$:}
 \begin{align}
 	\|\tilde T\f\|_1\leq \frac{C^{-2}}{\min\Big\{2\mu,d\left(\lambda + 
 		\frac{2\mu}{n}\right)\Big\}}\|\f\|_l,\quad l\in\{0,1\},\,\forall\,\f\in\hh.\label{eq:tildet-bound}
 \end{align}
 The compactness of the inclusion 
 $\hh^1(\Omega)\hookrightarrow\ll^2(\Omega)$, and the fact that $\hh$ is closed in $\hh^1(\Omega)$, imply that $\hh$ is continuously embedded in $\hh^1(\Omega)$, and hence
that the inclusion $\hh \hookrightarrow \ll^2(\Omega)$ is compact. Therefore, the previous bound with $l = 0$ imply the compactness of $\tilde T|_\hh$. 
The Spectral Theorem for bounded self-adjoint linear and compact operators says that $\tilde T$ has a countable real 
point spectrum  $\{{\alpha_n}\}\subseteq(0,{\|\tilde T\|})$ and eigenfunctions $\{\u_n\}$ such that $\tilde T\u_n = 
{\alpha_n}\u_n$ for 
all 
$n${, with $\|\tilde T\| := C^{-2}\min\Big\{2\mu,d\left(\lambda + 
 \frac{2\mu}{n}\right)\Big\}^{-1}$}. Note that {the eigenpair $({\alpha_n},\u_n)$ of $\tilde T$ solves \autoref{eq:primal2} with 
$\kappa_n =  \frac{1}{\alpha_n}$ and $\u_n$ as the corresponding eigenfunction.}

We remark that the results in this section also hold for the bilinear form $a(\cdot,\cdot)$ and therefore, since $\tilde a(\u,\v) = a(\u,\v)$ for all $\u,\,\v\in\hh$, this establishes the existence of the Jones spectrum for bounded Lipschitz domains $\Omega$ which are not axisymmetric.


{To finish up this section, we summarize the main properties in the following theorem.}
{\begin{theorem}\label{result:positiveevs}
 Let us assume $\Omega$ is a non-axisymmetric and Lipschitz domain of $\rrr^n$, $n\in\{2,3\}$. The spectrum of $\tilde 
T$ is $\{{\alpha_n}\}_{n\in\nnn}\subseteq (0,\|\tilde T\|)$, with eigenfunctions belonging to $\hh$. In addition, 
eigenfunctions corresponding to different eigenvalues are orthogonal in the usual $\ll^2$-inner product.
\end{theorem}}

\subsection{The case of zero eigenvalues and rigid motions}\label{sec:rigidmotions}
When studying problems involving the Lam\'e operator $\mathcal{L}$ (cf. \autoref{sec:Introduction}), we need to be aware of rigid motions. Depending on 
the boundary conditions imposed, rigid motions may be part of the eigenspace of certain eigenvalues. Rigid motions 
satisfy $\mathcal{L}\w=\zero$, and it is possible that they may satisfy both \autoref{eq:introjones2} and 
\autoref{eq:introjones3}. We now want to 
characterize domains having these eigenfunctions. Consider the space
\begin{align*}
 \rr\mm(\Omega) := \{\u\in\hh^1(\Omega):\, \u = \b+\bb\x,\,\,\b\in\rrr^n,\,\bb\in\rrr^{n\times 
n}_{skew},\,\x\in\Omega\},
\end{align*}
where $\rrr^{n\times n}_{skew}$ is the space of all skew-symmetric matrices in $\rrr^{n\times n}$. The space 
$\rr\mm(\Omega)$ consists of translations, rotations and combinations of these. It is known that (see, e.g. 
\cite{ref:babuskaosborn1991} and references therein) the linear elasticity problem with traction free boundary conditions
\begin{align}
 -\bfdiv\,\bsigma(\bar\u) = \delta \bar\u,\quad\text{in $\Omega$},\quad \bsigma(\bar\u)\n = {\bf 0},\quad\text{in 
$\partial\Omega$},\label{eq:neumann}
\end{align}
has eigenmodes in $\rr\mm(\Omega)$ with eigenvalue $\delta = 0$. In fact, if $n=2$, the eigenspace of 
$\delta = 0$ is exactly $\rr\mm(\Omega)$ with dimension 3. Define now the space 
\begin{align*}
 \zz := \Big\{\u\in\hh^1(\Omega):\,\tilde a(\u,\v) = 0,\,\,\forall\,\v\in\hh^1(\Omega)\Big\}.
\end{align*} 
Examining the weak formulation of \autoref{eq:neumann}, it is clear that $\rr\mm{(\Omega)} \subseteq \zz$.
We show that these spaces actually coincide.
\begin{theorem}\label{result:rigid}
 There holds $\rr\mm(\Omega) = \zz$.
\end{theorem}
\begin{proof}
 First, let $\u\in \rr\mm(\Omega)$. By definition, $\u = \b+\bb\x$, $\x\in\Omega$, $\bb$ skew-symmetric. Then $\nabla\u 
= \bb$ so that $\bvarepsilon(\u) = {\bf 0}$ and clearly $\tilde a(\u,\v) = 0$ for any $\v\in\hh^1(\Omega)$. Conversely, 
assume $\u\in\zz$. Then $\tilde{a}(\u,\u)=0$, and using the definition of $\tilde a(\cdot,\cdot)$ in 
\autoref{eq:aformdeviator} we 
get 
\begin{align*}
 0 = 2\mu\tr(\bvarepsilon(\u)) + n\lambda\tr(\bvarepsilon(\u)) = n\left(\lambda + 
\frac{2}{n}\mu\right)\tr(\bvarepsilon(\u)).
\end{align*}
Since $\lambda+\frac{2}{n}\mu>0$, we get that both $\tr(\bvarepsilon(\u)) = 0$  and $\bvarepsilon(\u) = {\bf 0}$. This 
implies that  $\u\in \rr\mm(\Omega)$.  
\end{proof}
Now that we know the eigenvalue problem in \autoref{eq:neumann} has $\rr\mm(\Omega)$ as the eigenspace of $\delta=0$, 
the question is if there is any non-zero elements in $\rr\mm(\Omega)\cap\hh$, i.e., those which satisfy the additional 
constraint $\u\cdot \n=0$ on the boundary. Such elements would be Jones modes corresponding to a zero Jones eigenvalue.

 The next {result} states the cases in which we 
have 2D rigid motions which additionally satisfy $\u \cdot \n =0$ on the boundary. We note parts (i) and (ii) of the 
result can be combined for a more succinct statement involving arbitrary half-spaces, but we provide the version below for 
clarity of exposition. 

\begin{theorem}\label{result:jonesrigid2d}
 Assume $n = 2$. If $\Omega$ is not bounded, then
 \begin{itemize}
  \item[(i)] $\rr\mm(\Omega)\cap\hh = {\rm span}\{(0,1)^\transpose\}$ if and only if $\Omega = 
\{(x,y)^\transpose\in\rrr^2:\, x>a, 
y\in\rrr\}$, for $a\in\rrr$.
\item[(ii)] $\rr\mm(\Omega)\cap\hh = {\rm span}\{(1,0)^\transpose\}$ if and only if $\Omega = 
\{(x,y)^\transpose\in\rrr^2:\, x\in\rrr,y>b\}$, for $b\in\rrr$.
\end{itemize}
In case $\Omega$ is bounded,
\begin{itemize}
\item[(iii)] $\rr\mm(\Omega)\cap\hh = {\rm span}\{(y,-x)^\transpose\}$ if and only if $\Omega = B(0,R)$.
 \end{itemize}
\end{theorem}
\begin{proof}
 For (i), suppose $\Omega = \{(x,y)^\transpose\in\rrr^2:\, x>a, 
y\in\rrr\}$, for some $a\in\rrr$.  Let 
$\u\in\rr\mm(\Omega)\cap\hh$. We write $\u = \b + \bb\x$, $\b\in\rrr^2$, $\bb$ skew-symmetric, and $\x\in\Omega$.
Assume
\begin{align*}
 \b = (b_1,b_2)^\transpose,\quad \bb = \left(\begin{array}{cc}
                                           0 & b\\
                                           -b & 0
                                          \end{array}
\right).
 \end{align*}
The normal on $\partial\Omega:=\{(a,y):\,y\in\rrr\}$ is $\n=(-1,0)^\transpose$. We have
\begin{align*}
 0 = \u\cdot\n = \b\cdot\n + \bb\x\cdot\n = b_1 + by,\quad\forall\,y\in\rrr.
\end{align*}
We must have $b_1 = 0$ and $b = 0$, which gives $\bb = {\bf 0}$ and $\b = (0,b_2)$, showing that $\u\in{\rm 
span}\{(0,1)^\transpose\}$. Part (ii) can be easily proved by following the same steps showed before. For (iii), assume 
$\Omega$ is a circle of radius $R$ {centered} at the origin. Let $\u\in\rr\mm(\Omega)\cap\hh$. As before, $\u = 
\b + 
\bb\x$, 
$\x\in\Omega$, and $\u\cdot\n = 0$ on $\partial\Omega = \{(x,y)\in\rrr^2:\, x^2+y^2=R^2\}$. Then
\begin{align*}
 0 = \u\cdot\n = \b\cdot\n + \bb\x\cdot\n = b_1n_1 + b_2n_2 + b(n_1y-n_2x),\quad\forall\,(x,y)\in\partial\Omega.
\end{align*}
The normal vector on $\partial\Omega$ is $\n = \frac{1}{R}(x,y)$. Putting this into the previous equation we 
obtain
\begin{align*}
 b_1x  + b_2y = 0,\quad\forall\,(x,y)\in\partial\Omega.
\end{align*}
Since $x$ and $y$ cannot be zero simultaneously, we conclude that $b_1 = b_2 = 0$ and $\u = \bb\x$, proving that 
$\u\in{\rm span}\{(y,-x)^\transpose\}$.

Note that the {converses} of all three parts (i), (ii) and (iii) are trivial since the basis of 
$\rr\mm(\Omega)\cap\hh$ is always orthogonal (in the Euclidean inner product) to the normal vector on the boundary of 
the corresponding domain.
\end{proof}

In the {result} above one could also {have the same conclusions for the} complement of each domain 
considered. Indeed, the 
complement of $\Omega$ in parts (i), (ii) and (iii) {obviously has the same boundary $\partial\Omega$, meaning 
that} the normal vector only changes its sign. This {indicates} that the vanishing 
condition of the normal trace of the displacement would be readily satisfied in this case as well.

\autoref{result:rigid} suggests that the traction eigenvalue problem given by \autoref{eq:neumann} has zero as 
eigenvalue with 
{rigid motions as eigenfunctions}, 
independent of the domain $\Omega$.  However, the extra constraint on the normal trace of the 
displacement $\u$ (cf. \autoref{eq:jones-modes}) may preclude {$\kappa = 0$} as a Jones eigenvalue on some 
domains. The elements in 
$\rr\mm(\Omega)\cap\hh$ 
depend on 
the boundary of $\Omega$ as shown in \autoref{result:jonesrigid2d}. Combining \autoref{result:rigid} and 
\autoref{result:jonesrigid2d}, we have the following theorem.
\begin{theorem}\label{thm:zero2d}
{Assume $n = 2$ and suppose} $\kappa = 0$ in \autoref{eq:primal1} (equivalently \autoref{eq:primal2}). If $\Omega$ is not bounded, then
\begin{itemize}
 \item[(i)] $\u_0 = (0,1)^\transpose$ is a Jones mode on $\Omega := \{(x,y)^\transpose\in\rrr^2:\, x>a, 
y\in\rrr\}$, for some $a\in\rrr$.
\item[(ii)] $\u_0 = (1,0)^\transpose$ is a Jones mode on $\Omega := \{(x,y)^\transpose\in\rrr^2:\, 
x\in\rrr,y>b\}$, for some $b\in\rrr$.
\end{itemize}
In case $\Omega$ is bounded, then
\begin{itemize}
\item[(iii)] $\u_0 = (y,-x)^\transpose$ is a Jones mode on $\Omega := B(0,R)$, for any fixed $0<R<\infty$. 
\end{itemize}
\end{theorem}
In the 3D case, a rigid motion can be decomposed as
\begin{align*}
\u =&\, c_1(1,0,0)^\transpose + c_2(y,-x,0)^\transpose + c_3(z,0,-x)^\transpose + c_4(0,1,0)^\transpose\\ 
&+c_5(0,z,-y)^\transpose + c_6(0,0,1)^\transpose,
\end{align*}
for constants $c_1,\ldots,c_6\in\mathbb{R}$. In this case, we see that we have three possible rotations and three 
possible translations. This implies 
that we may have more eigenvectors associated to the zero eigenvalue in \autoref{eq:primal1} (equivalently 
\autoref{eq:primal2}).

The spaces $\tt(\Omega)$ and $\rr(\Omega)$ are defined as the spaces of  pure translations and pure 
rotations of $\Omega$ respectively. These allow the following decomposition of $\rr\mm(\Omega)$:
\begin{align*}
	\rr\mm(\Omega) = \tt(\Omega)\oplus\rr(\Omega),
\end{align*}
with trivial intersection. To characterize the elements of $\rr\mm(\Omega)\cap \hh$, it was shown in 
\cite{ref:bauer2016} that axisymmetric domains always support rotational displacements in $\rr(\Omega)$ which are 
tangential to the boundary.

3D versions of \autoref{result:jonesrigid2d} and \autoref{thm:zero2d} can be now stated.
\begin{theorem}\label{result:jonesrigid3d}
 Assume $n = 3$. If $\Omega$ is not bounded, then
 \begin{itemize}
  \item[(i)] $\rr\mm(\Omega)\cap\hh = {\rm span}\{(1,0,0)^\transpose\}$ if $\Omega := 
\{(x,y,z)^\transpose\in\rrr^3:\, by + cz < a, x\in\rrr\}$, for some $a,b,c\in\rrr$ such that $1 = b^2 
+ c^2$.
\item[(ii)] $\rr\mm(\Omega)\cap\hh = {\rm span}\{(0,1,0)^\transpose\}$ if $\Omega := 
\{(x,y,z)^\transpose\in\rrr^3:\, ax + cz < b, 
y\in\rrr\}$, for some $a,b,c\in\rrr$ such that $1 = a^2 + c^2$.
\item[(iii)] $\rr\mm(\Omega)\cap\hh = {\rm span}\{(0,0,1)^\transpose\}$ if $\Omega := 
\{(x,y,z)^\transpose\in\rrr^3:\, ax + by < c, z\in\rrr\}$, for some $a,b,c\in\rrr$ such that $1 = a^2 + 
b^2$.
\end{itemize}
In case $\Omega$ is bounded,
\begin{itemize}
\item[(iv)] $\rr\mm(\Omega)\cap\hh \subseteq \rr(\Omega)$ if 
and only if $\Omega$ is axisymmetric. 
 \end{itemize}
\end{theorem}
\begin{proof}
 (i), (ii) and (iii) readily follow by applying the same steps as in the proof of parts (i) and (ii) of 
\autoref{result:jonesrigid2d}. For part (iv), let $\u\in\rr\mm(\Omega)\cap\hh\subseteq \rr(\Omega)$, and assume 
$\Omega$ 
is a non-axisymmetric domain. Then, Korn's inequality (cf. inequality \autoref{eq:korns}) holds for $\u$, that is, there 
is 
a 
constant $c>0$ such that $c\|\u\|_1\leq\|\bvarepsilon(\u)\|_0$. However, $\u\in\rr(\Omega)$ is a rotation so 
$\bvarepsilon(\u) = \zero$ and $\|\u\|_1 \neq 0$. This means that $c\leq 0$, which is a contradiction. For the converse 
of part(iv), assume that $\Omega$ is axisymmetric, and $\rr(\Omega)\subseteq\rr\mm(\Omega)\cap\hh$ with strict 
inclusion. This implies there is an element of $\rr\mm(\Omega)\cap\hh$ which is a non-zero translation motion; 
however, the boundary condition on the normal trace prohibits such modes.
\end{proof}

From here, \autoref{result:rigid} and \autoref{result:jonesrigid3d} give the next result.
\begin{theorem}\label{thm:zero3d}
 {Assume $n = 3$ and suppose} $\kappa = 0$ is an eigenvalue of \autoref{eq:jones}. If $\Omega$ is not bounded, then its corresponding 
eigenvector is:
\begin{itemize}
 \item[(i)] $\u_0 = (1,0,0)^\transpose$ on $\Omega := 
\{(x,y,z)^\transpose\in\rrr^3:\, by + cz < a, x\in\rrr\}$, for some $a,b,c\in\rrr$ such that $1 = b^2 
+ c^2$.
\item[(ii)] $\u_0 = (0,1,0)^\transpose$ on $\Omega := \{(x,y,z)^\transpose\in\rrr^3:\, ax + cz < b, 
y\in\rrr\}$, for some $a,b,c\in\rrr$ such that $1 = a^2 + c^2$.
\item[(iii)] $\u_0 = (0,0,1)^\transpose$ on the domain $\Omega := 
\{(x,y,z)^\transpose\in\rrr^3:\, ax + by < c, z\in\rrr\}$, for some $a,b,c\in\rrr$ such that $1 = a^2 + 
b^2$.
\end{itemize}
In case $\Omega$ is bounded, its corresponding eigenvector is
\begin{itemize}
\item[(iv)] $\u_0 \in\rr(\Omega)$ whenever $\Omega$ is axisymmetric.
\end{itemize}
\end{theorem}

In the case of the circle in 2D, the zero eigenvalue would lead to a bilinear form $a(\cdot,\cdot)$ that is not 
$\hh$-elliptic. For 
the 3D case, axisymmetric domains would lead to a loss of  $\hh$-ellipticity for the bilinear form 
$a(\cdot,\cdot)$. To overcome this issue, we 
add a shift to the formulation in \autoref{eq:primal1} to get the equivalent formulation: find 
$(\u,\kappa)\in\hh\times \rrr$ such 
that
\begin{align}
 \bar a(\u,\v) = (\kappa+1)(\u,\v)_0,\quad\forall\,\v\in\hh,\label{eq:shifted-form}
\end{align}
with $\bar a(\u,\v) := a(\u,\v) + (\u,\v)_0,$ for all $\u,\,\v\in\hh$. This new formulation is obviously 
$\hh$-elliptic 
since for any $\u\in\hh$ we have
\begin{align*}
 \bar a(\u,\v) = \mu\|\nabla\u\|_0^2 + (\lambda + \mu)\|{\rm div}\,\u\|_0^2 + \|\u\|_0^2 \geq 
\min\{\mu,1\}\|\u\|_1^2.
\end{align*}
The symmetry of $a(\cdot,\cdot)$ and the inner product $(\cdot,\cdot)_0$ along with the Rayleigh quotient (cf. 
\autoref{eq:rayleigh}) 
show that the eigenvalues $\kappa+1$ of \autoref{eq:shifted-form} are real and positive. 

We define the solution operator $\bar T:\ll^2(\Omega)\to\hh$ by $\bar T(\f) = \u$ such that
\begin{align}
 \bar a(\u,\v) = (\f,\v)_0,\quad\forall\,\v\in\hh,\label{eq:shifted-source}
\end{align}
Since $\bar a(\cdot,\cdot)$ is $\hh$-elliptic, the Lax-Milgram lemma shows that the restriction of $\bar T$ to $\hh$, $\bar T|_\hh$, is a well-defined linear 
operator and also gives the boundedness of $\bar T|_\hh$ in the $\ll^2$- and $\hh^1$-norms:
\begin{align}
 \|\bar T\f\|_1\leq \frac{1}{\min\{\mu,1\}}\|\f\|_l,\quad l\in\{0,1\},\,\,\forall\,\f\in\hh.\label{eq:bart-bound} 
\end{align}
{We again use the compact inclusion $\hh^1(\Omega)\hookrightarrow \ll^2(\Omega)$ to obtain the compactness of the inclusion $\hh\hookrightarrow\ll^2(\Omega)$.}
This compact inclusion and the bound of $\bar T|_\hh$ in 
\autoref{eq:bart-bound} with $l = 0$ imply that $\bar T$ is a compact operator (see \cite{ref:babuskaosborn1991}). 
Also, the symmetry of $\bar a(\cdot,\cdot)$ implies that $\bar T$ is a self-adjoint with respect to $\bar 
a(\cdot,\cdot)$. 
The 
Spectral Theorem for compact and self-adjoint bounded linear operators now implies the existence of positive eigenvalues 
$\{{\beta_n}\}_{n\in\mathbb{N}}$ 
and eigenfunctions $\{\u_n\}_{n\in\nnn}$ such that $\bar T\u_n = 
{\beta_n}\u_n$ and ${\beta_n}\to 0$. Note that $\bar T\u_n = {\beta_n}\u_n$ is a solution of 
\autoref{eq:shifted-source} 
if 
and only if $(\kappa_n,\u_n)$ solves \autoref{eq:primal1} with ${\beta_n} := \frac{1}{\kappa_n + 1}$. Since $\kappa_n\geq0$ for any $n\in\nnn$, we see that $\{{\beta_n}\}_{n\in\mathbb{N}}\subseteq (0,1]$. We 
summarize these 
properties in the following main result.
\begin{theorem}\label{result:nonnegativeevs}
 The point spectrum of $\bar T|_\hh$ is decomposed as follows: $\{{\beta_n}\}_{n\in\nnn}\cup\{1\}$, where
 \begin{enumerate}
  \item the associated eigenspace of the eigenvalue 1 is given by \autoref{result:jonesrigid2d} in 2D and 
\autoref{result:jonesrigid3d} in 3D;
  \item $\{{\beta_n}\}_{n\in\nnn}\subseteq (0,1)$ is a sequence of eigenvalues of $\bar T$ with 
finite multiplicity that converges to 0 and their corresponding eigenfunctions lie in $\hh$. 
 \end{enumerate}
 {In addition, 
eigenfunctions corresponding to different eigenvalues are orthogonal in the usual $\ll^2$-inner product.}
\end{theorem}

{\subsection{Elastic bodies with variable density}\label{subsection:vardensity}
In many realistic applications the density $\rho>0$ varies. In this section we discuss the existence of Jones 
eigenpairs for variable density. Under the same assumptions on $\Omega$ as given at the beginning of 
\autoref{sec:formulation}, we consider a variable density $\rho>0$ belonging to $L^\infty(\Omega)$. The weak formulation 
in \autoref{eq:primal1} would then be: find Jones eigenpairs $\u\in\hh$ and $w^2\in\rrr$ such that
\begin{align}
 a(\u,\v) = w^2\,b(\u,\v),\quad \forall\,\v\in\hh,
\end{align}
where the bilinear form $b(\cdot,\cdot)$ is defined as
\begin{align*}
 b(\u,\v) := \int_\Omega \rho\, \u\cdot\v,\quad\forall\,\u,\,\v\in\hh.
\end{align*}
Since $\rho\in L^\infty(\Omega)$, we have that $|b(\u,\v)|\leq \|\rho\|_\infty \|\u\|_0\|\v\|_0,$ for all 
$\u,\,\v\in\hh$. If $L^2_\rho(\Omega)$ denotes the space of functions in $L^2(\Omega)$ whose weighted inner product is 
finite
(with the variable density as the corresponding weight), we note that the inclusion $\hh\hookrightarrow 
L^2_\rho(\Omega)$ is compact.}

{The fact that the density varies does not change the positiveness of $a(\cdot,\cdot)$ (equivalently 
$\tilde a(\cdot,\cdot)$) in $\hh$. We then have that the $\hh$-ellipticity in \autoref{eq:tildeacoercive} holds true in 
this case as well.}

{We can conclude that the existence of Jones eigenpairs is guaranteed by \autoref{result:positiveevs} and 
\autoref{result:rigid} for the non-axisymmetric and the axisymmetric case, respectively. However, eigenfunctions 
corresponding to different eigenvalues would be orthogonal in the weighted inner product as defined by the bilinear 
form 
$b(\cdot,\cdot)$.}

{We end} this section by summarizing our main results. For axisymmetric domains, Jones modes exist and include 0 
as 
an eigenvalue with certain rigid motions as permissible eigenmodes. For non-axisymmetric Lipschitz domains which are 
bounded, there are countably many positive Jones eigenvalues whose only accumulation point is at infinity. 
{Finally, if a variable density of the elastic body is assumed, then the existence of eigenpairs falls into the setup of \autoref{result:positiveevs} and \autoref{result:nonnegativeevs}, depending on the nature of the shape of 
the 
domain.}

{In the forthcoming section we use a standard conforming finite element method to approximate Jones eigenmodes. 
As per usual of this approach, one has that at a given level of refinement, curved boundaries are approximated by 
straight edges and/or faces. Numerically speaking, this means that, in the case of an axisymmetric domain, the 
numerical method 
may not compute the zero eigenvalue which should be a Jones eigenvalue as shown in \autoref{result:nonnegativeevs}. As 
we 
will see in \autoref{sec:numerics}, even in the most simple case the standard conforming finite element method used for 
non-axisymmetric domain might not be a suitable choice to approximate Jones eigenpairs on axisymmetric domains. 
}

\section{Discrete weak formulations}\label{sec:discrete} 
{The preceding discussion shows that the behaviour of Jones modes for axisymmetric and non-axisymmetric domains is different. We now present discretization strategies for both. Even though two discrete formulations are given for the Jones eigenproblem, we only provide a priori error estimates on polygons or polyhedra since, as the results presented in \autoref{subsection:resultsdisk} below suggest that the discrete formulation in \autoref{eq:discrete1} does not appear to approximate the Jones eigenvalues correctly on curvilinear domains.}

We first let $\Omega$ be a {polygonal/polyhedral domain in $\rrr^n,$ with} $n\in\{2,3\}$, { and let} $\ttt_h$ be a regular triangulation by triangles ({tetrahedra in 3D}) of 
$\overline{\Omega}$ with meshsize $h$. For a given integer {$k\geq1$}, we consider the 
space $\pp_k(T)$ as the set of all vector polynomials of degree at most $k$ defined on 
$T\in\ttt_h$. Define the space
\begin{align*}
 \hh_h:=\Big\{\v_h\in {\cc}(\overline{\Omega}):\, \v_h|_T\in\pp_k(T),\,\,\forall\, 
T\in\mathcal{T}_h\Big\}\cap\hh,
\end{align*}
{where $\cc(\overline{\Omega})$ is the space of continuous vector fields defined over $\overline{\Omega}$. 
Consider} the {discrete weak formulation}: find $\u_h\in\hh_h$ and $\kappa_h\in\rrr$ such that
\begin{align}
 a(\u_h,\v_h) = \kappa_h\, (\u_h,\v_h)_0,\quad\forall\, \v_h\in \hh_h.\label{eq:discrete1}
\end{align}
{From the discussion in Sections 2 and 3},  
for non-axisymmetric domains the bilinear forms $\tilde a(\cdot,\cdot)$ and $a(\cdot,\cdot)$ 
{coincide} in 
$\hh$. We 
only 
provide approximation results for the eigenvalue problem \autoref{eq:primal1} as they readily apply to the formulations 
\autoref{eq:primal2} and \autoref{eq:shifted-form}.

Since $\hh_h$ is a subspace of $\hh$, the coercivity of $a(\cdot,\cdot)$ (cf. inequality \autoref{eq:tildeacoercive}) in 
$\hh$ implies 
its
$\hh_h$-ellipticity. We can define a discrete solution operator $T_h$ as follows:
\begin{align*}
 T_h:\,&\ll^2(\Omega)\to \hh_h\\
     &\f \to T_h(\f) :=\u_h,
\end{align*}
where $\u_h\in \hh_h$ is the solution of the problem
\begin{align*}
 a(\u_h,\v_h) = (\f,\v_h){_0}\quad\forall\, \v_h\in \hh_h.
\end{align*}
{Analogous to the situation of the continuous weak formulation}, the pair $(\kappa_h,\u_h)$ solves \autoref{eq:discrete1} if and only if $T_h\u_h = 
\sigma_h\u_h$ and $\sigma_h := \frac{1}{\kappa_h}$. Also, the restriction operator 
$T_h|_{\hh_h}:\,\hh_h\to \hh_h$ is self-adjoint with respect to $a(\cdot,\cdot)$ and 
{$(\cdot,\cdot)_0$}. We thus have the following result concerning the 
spectrum of $T_h|_{\hh_h}$ on {non-axisymmetric Lipschitz domains}.
\begin{theorem}\label{thm:discrete-spectrum}
 The spectrum of $T_h|_{\hh_h}$ consists of $M_h := \dim (\hh_h)$ 
eigenvalues. Moreover,
\begin{enumerate}
 \item the point spectrum consists of {positive} eigenvalues $\{\sigma_{h,k}\}_{k=1}^{M_h}$ 
counted with their multiplicities;
 \item $\sigma_h = 0$ is not an eigenvalue of $T_h$.
\end{enumerate}
\end{theorem}

Concerning the approximation properties of this scheme, as described in \cite{ref:babuskaosborn1991}, we have the 
following error bounds for the eigenvalues and eigenfunctions of \autoref{eq:discrete1}:
\begin{align}
 \frac{|\kappa-\kappa_h|}{|\kappa|} \leq\,C\epsilon_h(\kappa)^2,\quad \|\u-\u_h\|_1 
\leq\,C\epsilon_h(\kappa),\label{eq:evs-error}
\end{align}
where the term $\epsilon_h$ is defined as
\begin{align*}
 \epsilon_h(\kappa) := \sup_{\u\in\hh(\kappa)}\inf_{\u_h\in\hh_h}\|\u-\u_h\|_1.
\end{align*}

For a given $\kappa\in\rrr$, {the subset $\hh(\kappa)$ of $\hh$ is defined} as
\begin{align*} 
\hh(\kappa):= \Big\{\u\in\hh:\, \u\,\,\text{solves \autoref{eq:primal1} with eigenvalue $\kappa$}, \|\u\|_0 
= 1\Big\}.
\end{align*}
That is, $\hh(\kappa)$ is the eigenspace of the eigenvalue $\kappa$, containing normalized eigenvectors (in the 
$\ll^2$-norm).

The upper bounds for the errors in \autoref{eq:evs-error} hold for eigenvalues with multiplicity greater than 1. In 
fact, 
if $\kappa$ is an eigenvalue of \autoref{eq:primal1} of multiplicity $M\in\nnn$ with $\u_m\in\hh(\kappa)$, for all $m = 
1,\ldots,M$, then there is a {unit vector (in the $\ll^2$-norm)} $\w$ in $\gen\{\u_1,\ldots,\u_M\}$ and a 
vector field 
$\w_h$ in the span of $\{\u_{1,h},\ldots,\u_{M,h}\}${, with $\|\w_h\|_0 = 1$,} such that
\begin{align*}
 \|\w-\w_h\|_1 
\leq\,C\epsilon_h(\kappa),
\end{align*}
where the vectors $\u_{1,h},\ldots,\u_{M,h}$ solve \autoref{eq:discrete1} with $\kappa_h$.

Regarding the approximation estimates of the finite element discretization, for a regular triangulation, the 
interpolation error estimate for the Lagrange finite elements is
\begin{align*}
 \|\u-\ii_h\u\|_s\leq C{h^{\min\{k,t\}+1-s}|\u|_{t+1},\quad\forall\,\u\in\hh^{t+1}(\Omega),}
\end{align*}
where $\ii_h$ is the vector version of the usual Lagrange interpolant (componentwise), and $t,s\geq 0$. Using this interpolation error estimate in the error bound for 
$\kappa$ in \autoref{eq:evs-error}, we have
\begin{align}
 \frac{|\kappa-\kappa_h|}{|\kappa|} \leq\,c\,{h^{2\min\{k,t\}}|\u|_{t+1}^2,\quad t>0}.\label{eq:decayevs}
\end{align}
Note that the rate of convergence of the eigenvalues $\kappa_h$ depends on the regularity of its corresponding 
eigenvector $\u\in{\hh^{t+1}(\Omega)}$, and the error in the computed eigenvectors would decay as 
{$h^{\min\{k,t\}}$}.

{Clearly, Jones eigenpairs form a subset of the eigenpairs of the Lam\'e operator with traction conditions (see eigenproblem given in \autoref{eq:neumann}). It is known that the eigenvectors of the latter problem posses extra regularity, depending on the vertices/edges of the domain. Concretely, it is shown in  \cite{ref:grisvard1989} (see also \cite{ref:nicaise1992,ref:costabel1995,ref:rossle2000}) that the tractions eigenfunctions belong to $\hh^{1+\epsilon}(\Omega)$, for some $\epsilon\in(0,1]$.} { For a polygonal/polyhedral domain, we have therefore that the Jones eigenvectors belong to $\hh^{1+\epsilon}(\Omega)$, where $\epsilon\in(0,1]$ is the first positive root $r$ solving the following nonlinear equation \cite{ref:grisvard1989,ref:nicaise1992}:
	\begin{align}
		r^2\sin^2(\theta) = \sin^2(r\theta),\,\,r\in\rrr,\label{eq:regularity-eqs}
	\end{align}
	where $\theta$ represents the largest of the interior angles of $\Omega$. We notice that in the case of Neumann boundary conditions, the regularity of the eigenvectors does not seem to be affected by the Lam\'e parameters. Moreover, we note that $r = 1$ is always a solution to \autoref{eq:regularity-eqs}; this means that the best regularity one can obtain is for the Jones eigenvectors to belong to $\hh^2(\Omega)$. }

{In the case of domains with a smooth boundary, the problem formulation must be modified. The approximation of the Jones eigenvalues on this domain is not guaranteed when using the discrete formulation given by} \autoref{eq:discrete1}; {numerical results demonstrating this are presented in} \autoref{subsection:resultsdisk} {below. Instead, a mixed formulation may be more appropriate. }
	
{If $\Omega$ is axisymmetric, the analysis in \autoref{sec:rigidmotions} implies that a shift needs to be added in 
	\autoref{eq:discrete1}. 
	The essential condition on the normal trace of the displacement is added to the formulation in \autoref{eq:primal1}. 
	The 
	equivalent mixed formulation would then be: find $(\u,p)\in \hh^1(\Omega) \times H^1(\Omega)$ such that}
	\begin{align}\label{eq:mixedform}
		\begin{array}{rcr}
			\tilde a(\u,\v) + \langle \v\cdot\n,p\rangle_{1/2} =& \,\kappa\,(\u,\v)_0,&\forall\,\v\in\hh^1(\Omega),\\
			\langle \u\cdot\n,q\rangle_{1/2} - \eta\,(p,q)_0 =&\, 0,&\forall\,q\in H^1(\Omega),
		\end{array}
	\end{align}
	{where $\eta\geq0$ is a stabilization constant. For $\eta = 0$, this formulation is obviously equivalent to 
	\autoref{eq:primal1}. The stabilization 
	term $\eta\,(p,q)_0$ is added only for implementation purposes as the Lagrange multiplier $p$ is being defined on 
	the 
	whole domain $\Omega$. 
	Note that this implementation does not require the use of a penalty method to introduce Dirichlet boundary data as we 
	needed for the original formulation \autoref{eq:discrete1}.  A full 
	error analysis of this formulation on curvilinear domains will be presented in a future work.}

\section{Numerical results}\label{sec:numerics}
We now present some numerical results that support the theoretical {results established} in the 
previous 
sections.

\subsection{Convergence studies for {polygonal and} polyhedral domains}
We {first present three} numerical examples {on non-smooth domains. We consider three domains:} the square 
$\Omega_1 := (-1,1)^2$, the L-shape $\Omega_2 := \Omega_1\backslash [0,1)^2$, and the unit cube $\Omega_3 := 
(0,1)^3$. We recall that the rate of convergence of the discretization depends entirely on the regularity of the eigenvectors in \autoref{eq:jones-modes}. Using the discussion of the previous section and \autoref{eq:regularity-eqs} we see that on the square and on the cube, the Jones eigenmodes belong to $\hh^2(\Omega)$. Combined with  \autoref{eq:decayevs}, the error in the discrete eigenvalues should decay as $h^{2}$ if we use piecewise linear elements. 
For the L-shape domain $\Omega_2$, the decay rate will be slower due to the lower regularity. From \autoref{eq:regularity-eqs} with $\theta = \frac{3\pi}{2}$, we see that its first non-zero root is approximately $r_1:=0.5445$. This means that the eigenfunctions on the L-shape belong to $\hh^{1+r_1}(\Omega_2)$.

In all the experiments we have used $\pp_1$-conforming elements {to approximate the eigenpairs}
on a sequence of regular (not necessarily uniform) meshes. The true eigenvalues were used as exact solutions on the square (cf. \autoref{subsec:harge}) and cube, and a $\pp_2$-conforming approximation  on a very fine grid was used to obtain reference solutions on the L-shape. These experiments were implemented in FreeFem++ \cite{ref:freefem}. {For completeness, we remark that Dirichlet boundary conditions are added to the system as a penalty term in the standard manner. We shall examine the numerical convergence rates in terms of the degrees of freedom (DOFs), which for a $P1$ element scales as the number of vertices in the mesh. Recall that in this case the meshsize  $h$ scales as DOFs$^{-1/n}$, $n\in\{2,3\}$, and therefore the predicted rates of convergence for the eigenvalues on $\Omega_1$ and $\Omega_3$ are DOFs$^{-1}$ and DOFs$^{-2/3}$ respectively.}

\graphicspath{{./images/cv-study/}}
\begin{figure}[!ht]
\centering\includegraphics[width = .5\textwidth, 
height=0.2\textheight]{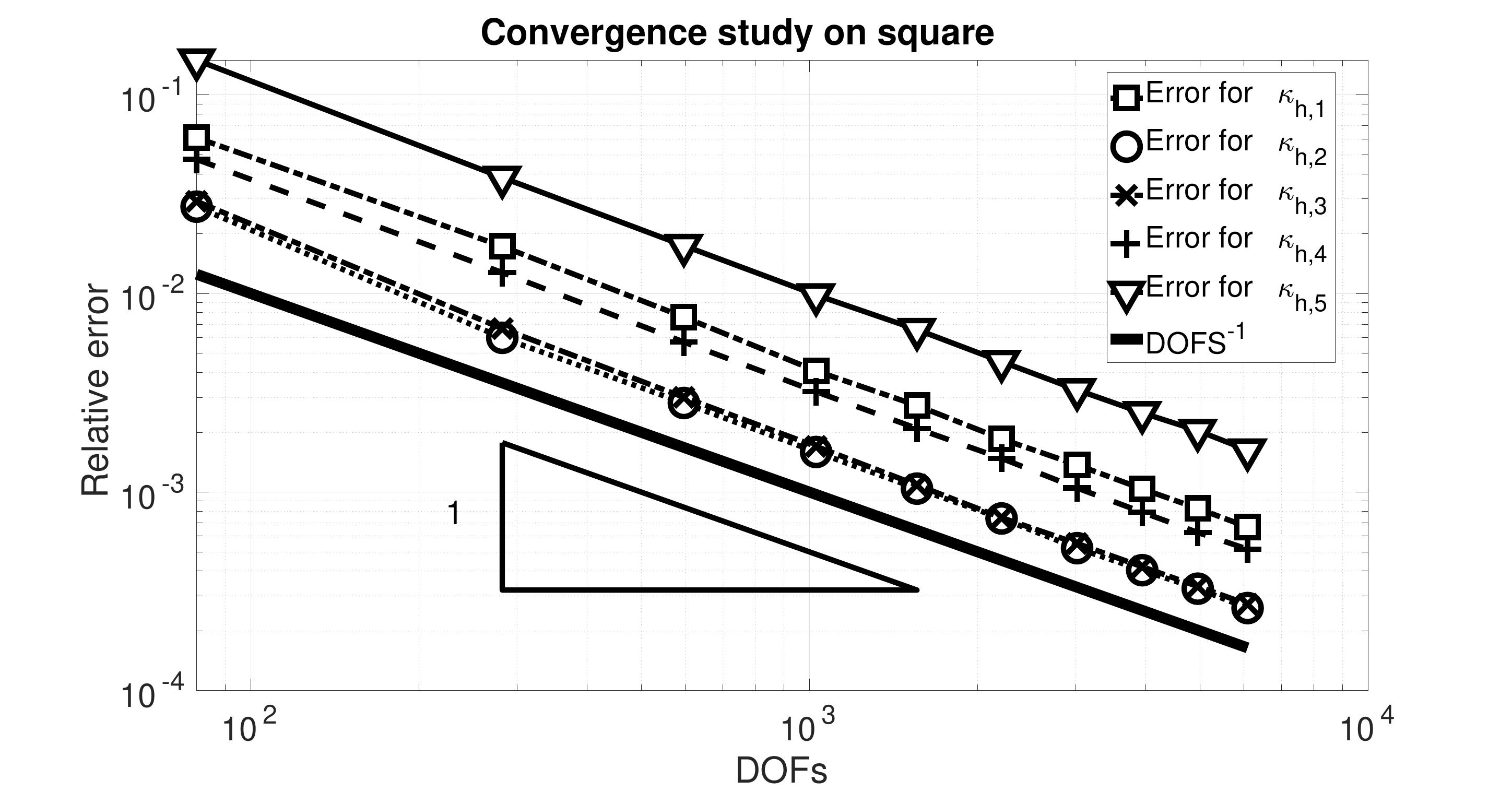}\includegraphics[width = .5\textwidth, 
height=0.2\textheight]{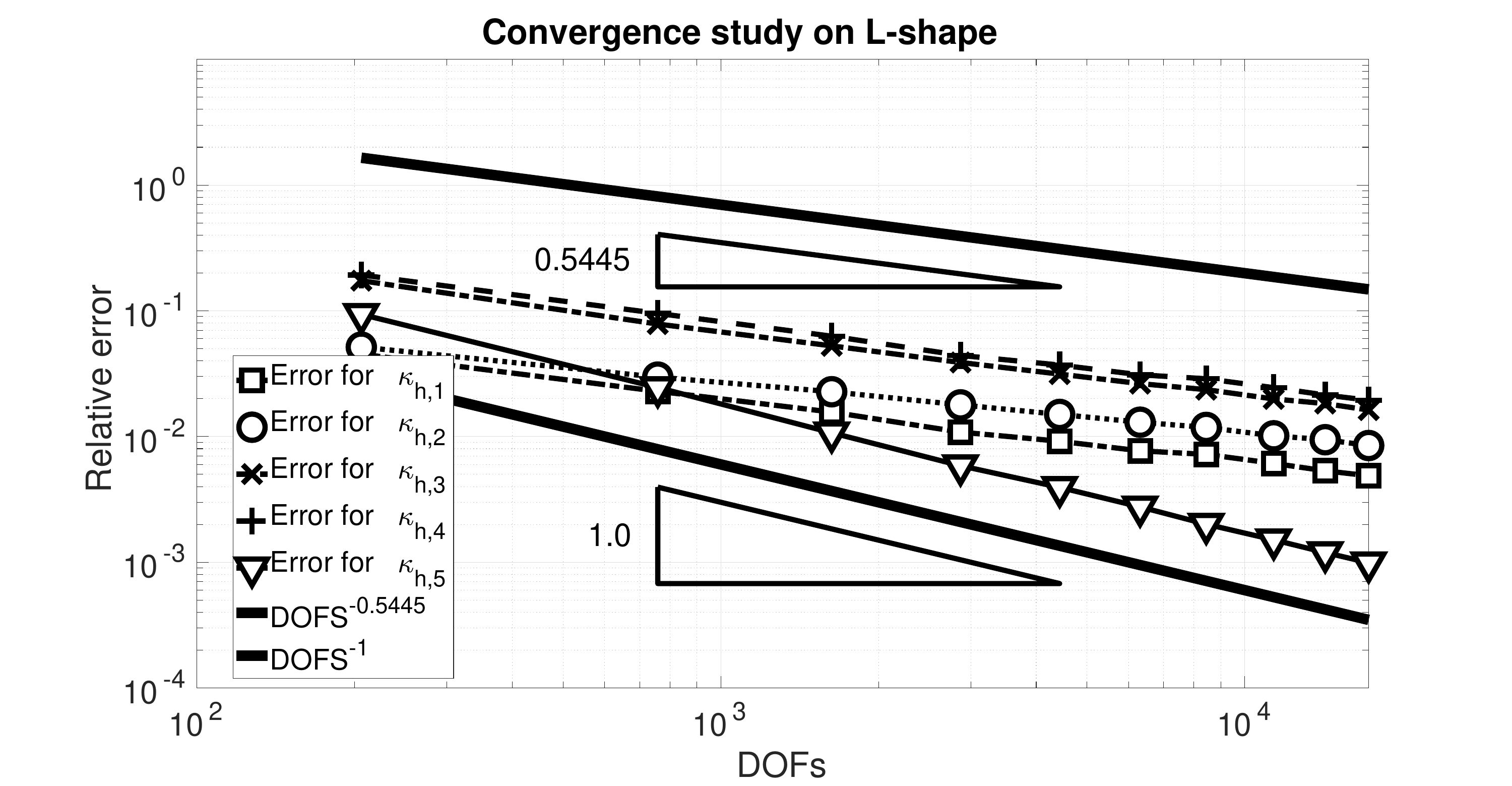}

\includegraphics[width = .5\textwidth, 
height=0.2\textheight]{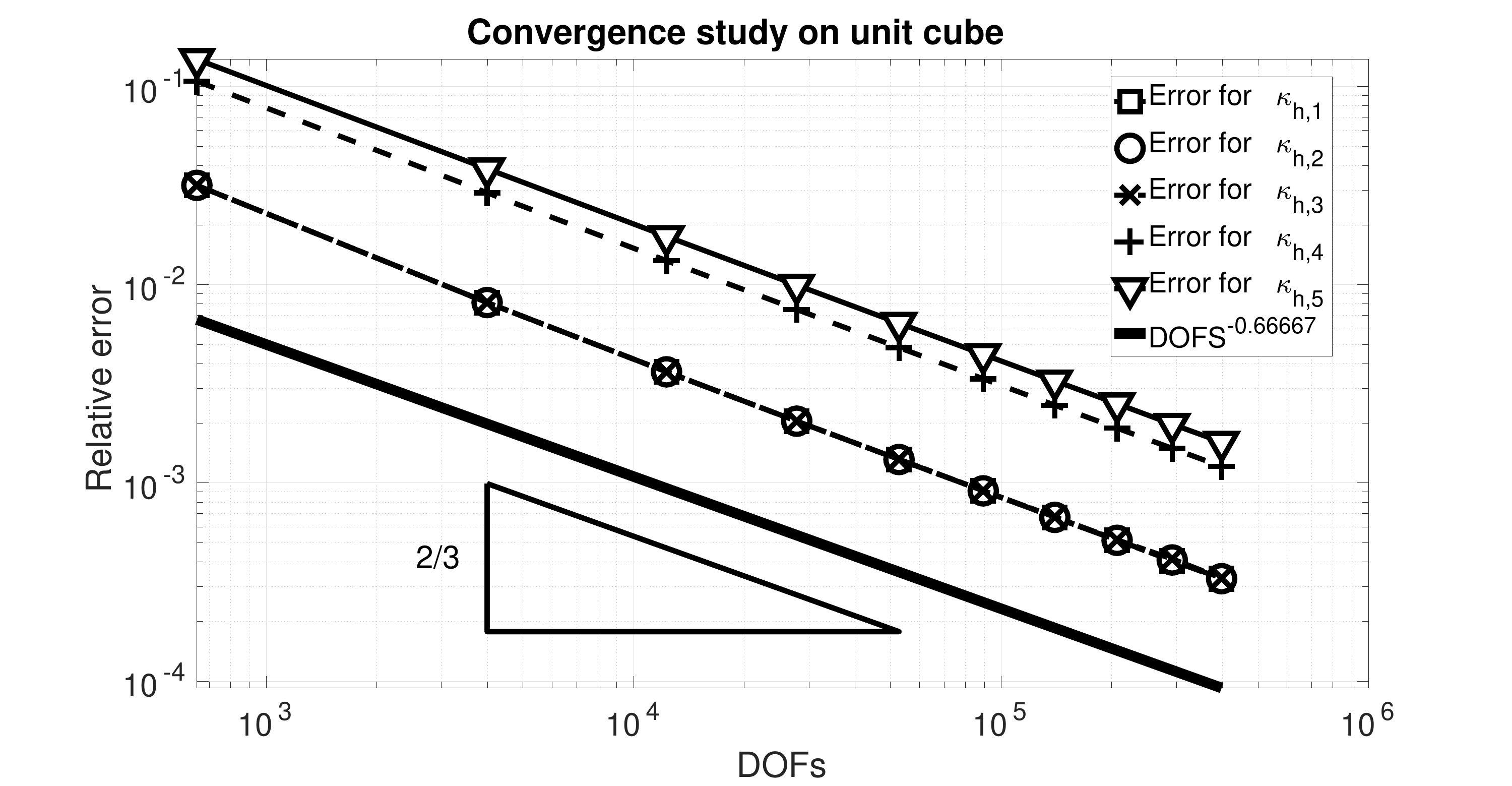}
\caption{Convergence study for the first 5 eigenvalues of 
\autoref{eq:jones} on $\Omega_1$ (top-left), $\Omega_2$ (to-right), and $\Omega_3$ 
(bottom). }\label{fig:cvhistory}
\end{figure}

 The convergence history of the first 5 eigenvalues of \autoref{eq:primal1} on $\Omega_1$ is shown in 
\autoref{fig:cvhistory} top-left for the parameters $\mu = 10,\, \lambda = 1,$ and $\rho = 12$. We observe that the discretization
method exhibits the expected convergence behaviour as we {increase the number of degrees of freedom. We can see that in all 5 cases the error goes down with a similar rate.}

For our second example on the L-shaped domain $\Omega_2$, we set the parameters $\mu = \lambda = \rho = 1$. As expected, the rate of convergence reflects the poor regularity of {some} eigenfunctions. However, as for other elliptic operators, it seems that some eigenfunctions do not capture the corner singularity at the origin. This is the case for  $\kappa_{h,5}$ (\autoref{fig:cvhistory} bottom). We see that the relative error for this eigenvalues decays as DOFs$^{-1}$, suggesting that its corresponding eigenfunctions belong to $\hh^2(\Omega_2)$.  

{In the last example, we consider the unit cube $\Omega_3$ with parameters $\mu = 10$, 
$\lambda = 1$, and $\rho = 12$. \autoref{fig:cvhistory} (top right) shows the convergence 
history of the first five Jones eigenvalues on $\Omega_3$. We notice that computed eigenvalues 
converge at the predicted rate as the triangulation gets refined. }

\subsection{An example of variable density}\label{subsection:squarevardensity}
{Jones modes also can be found when the material density $\rho$ varies. As an example, 
	we consider the unit square $\Omega_1$ with a variable density $\rho(x,y) = 12|x-0.3||y-0.25|$
and Lam\'e parameters $\mu=\lambda=1$.  The rate of convergence of the first five Jones eigenvalues is shown in \autoref{fig:cv_study_dofs_square-variabledensity} top. We see that the error of the computed eigenvalues decays at the same rate as for the case with constant density. The two eigenfunctions shown in \autoref{fig:cv_study_dofs_square-variabledensity} middle and bottom correspond to distinct eigenvalues; the weighted $\ll^2_\rho$-inner product of these eigenfunctions is zero as discussed in \autoref{subsection:vardensity}.}

\graphicspath{{./images/}}
\begin{figure}
\centering
\includegraphics[width=1.0\textwidth, height=.35\textheight]{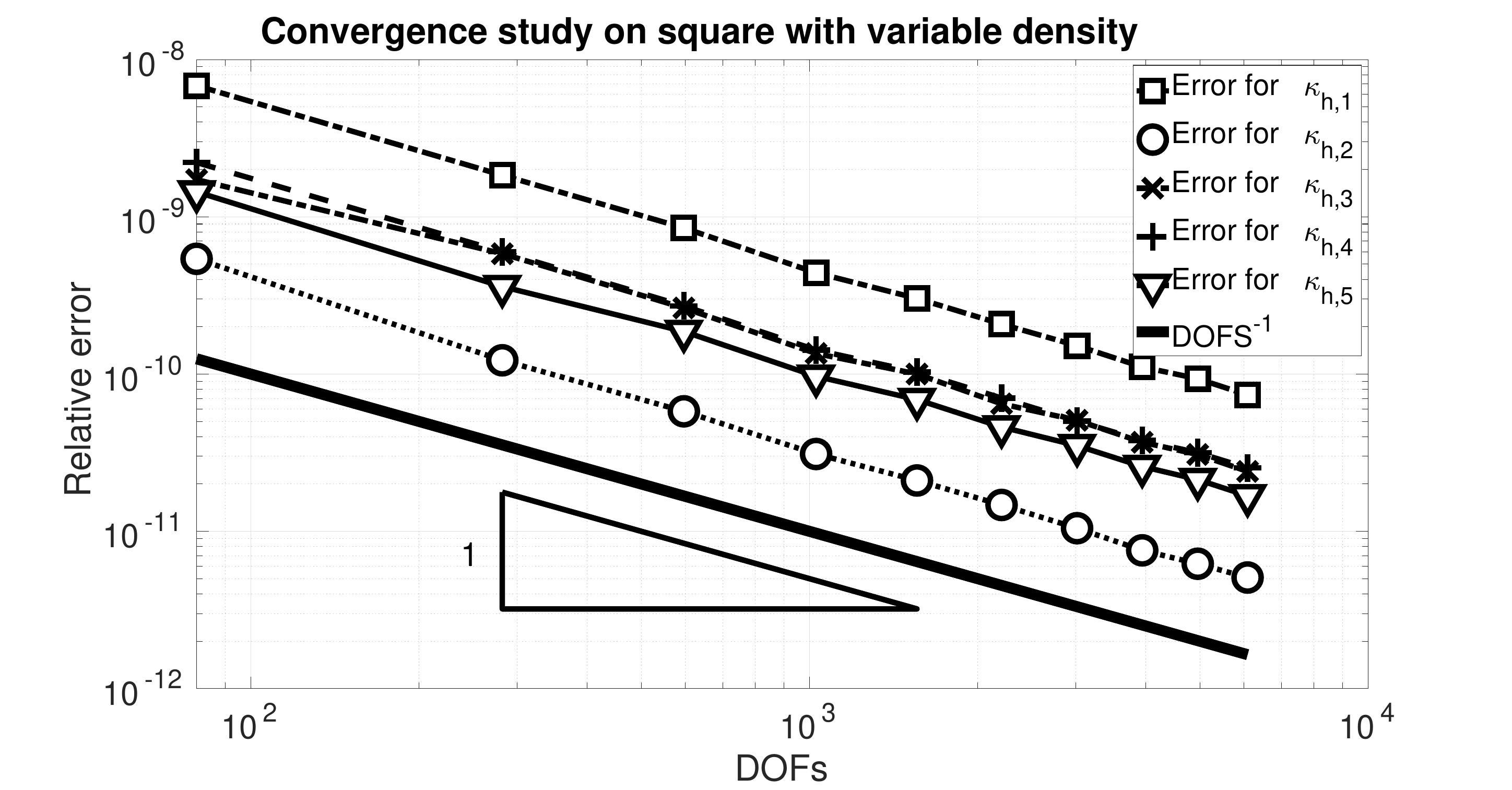}\\
\includegraphics[width=0.5\textwidth, height=.2\textheight]{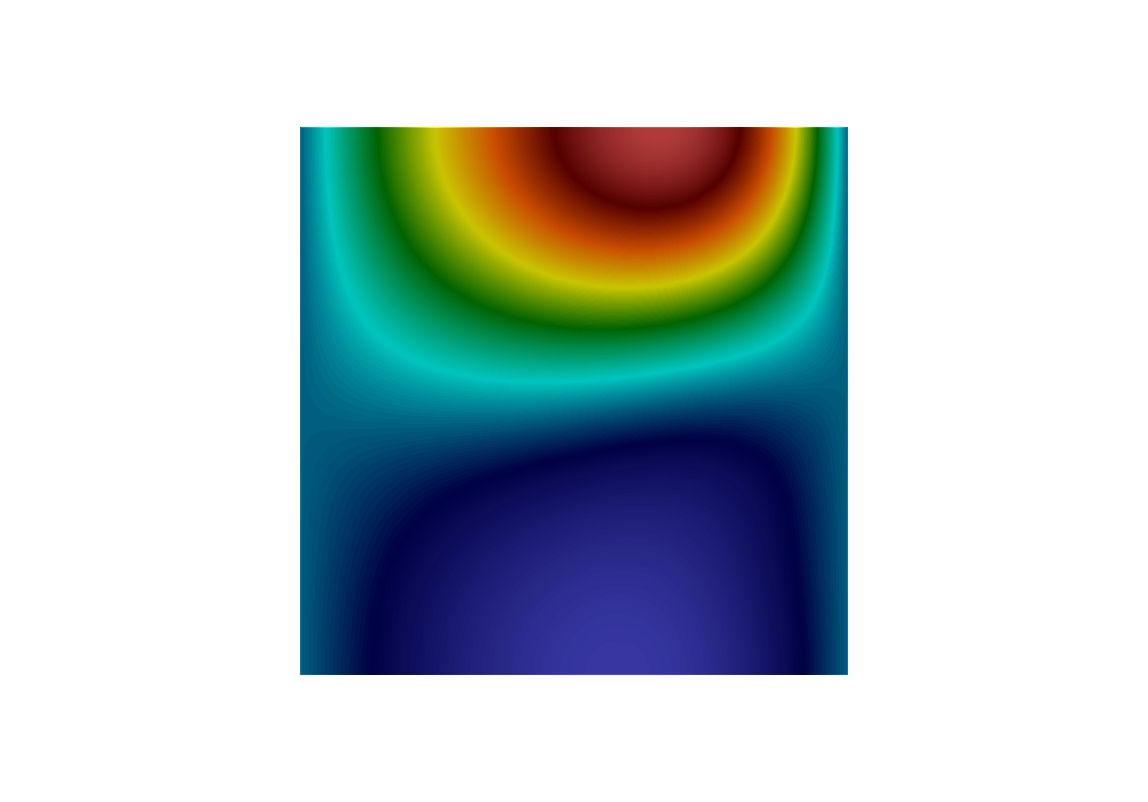}\includegraphics[width=0.5\textwidth, height=.2\textheight]{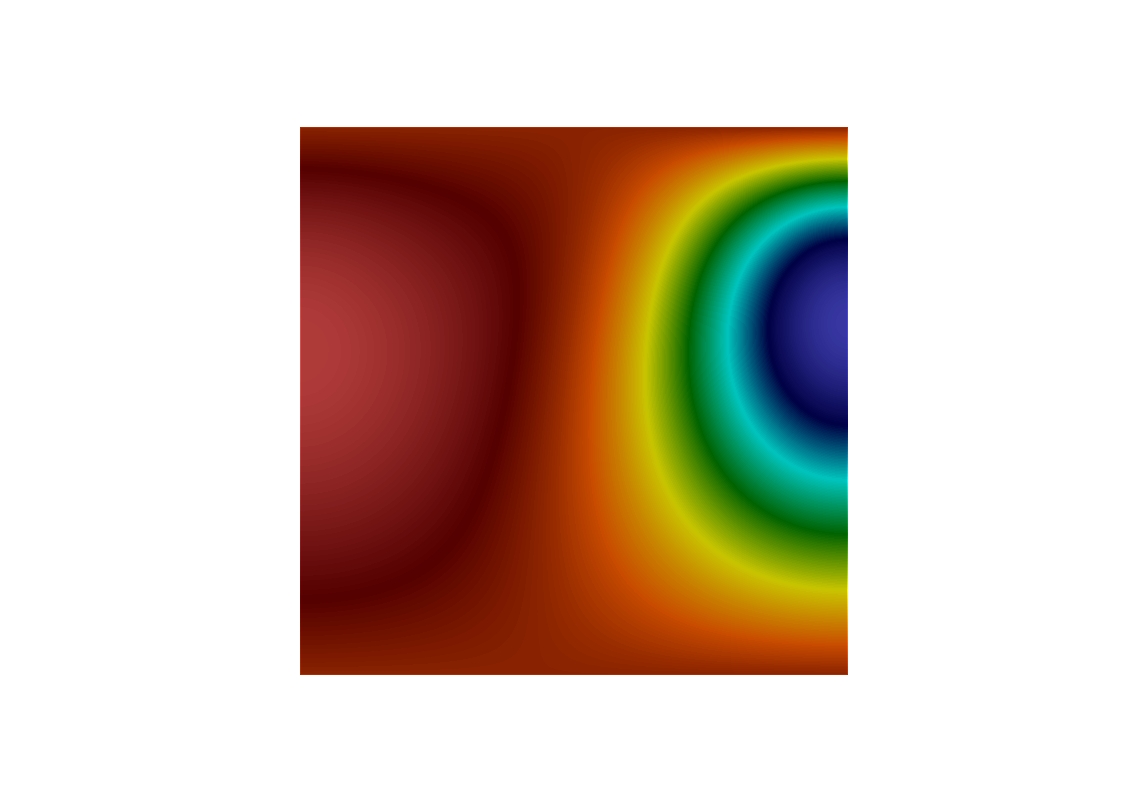}
\includegraphics[width=0.5\textwidth, height=.2\textheight]{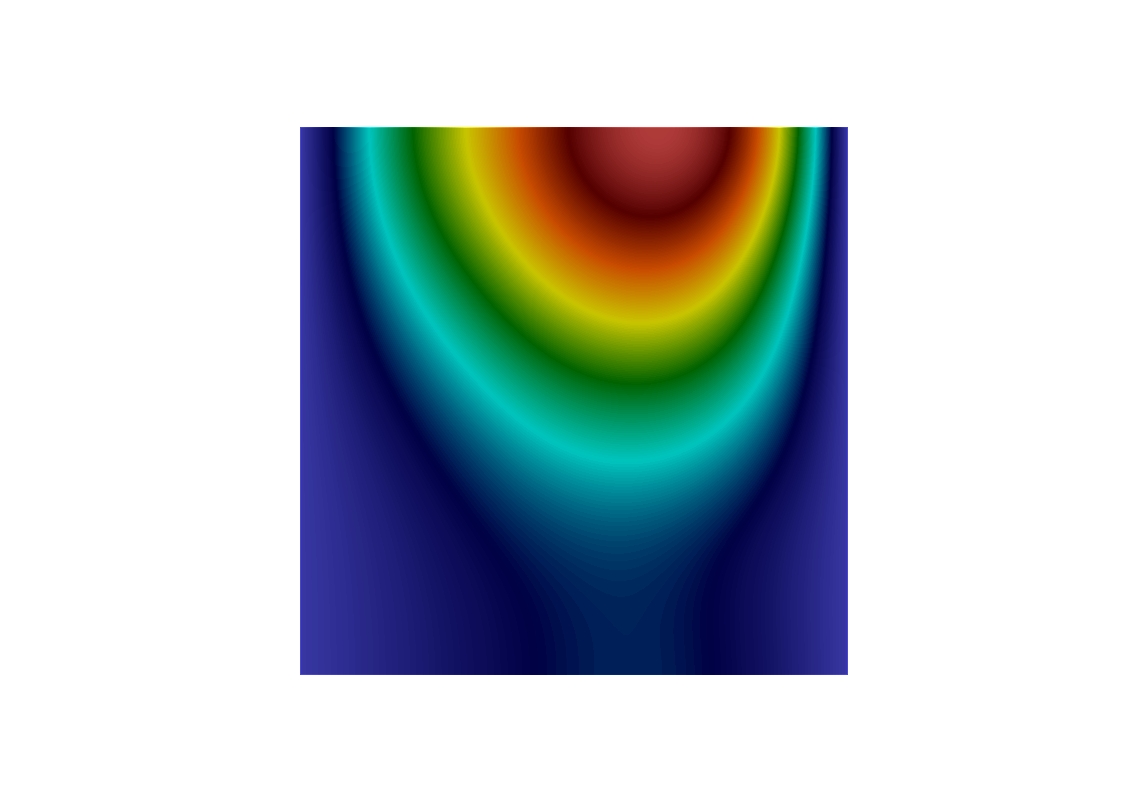}\includegraphics[width=0.5\textwidth, height=.2\textheight]{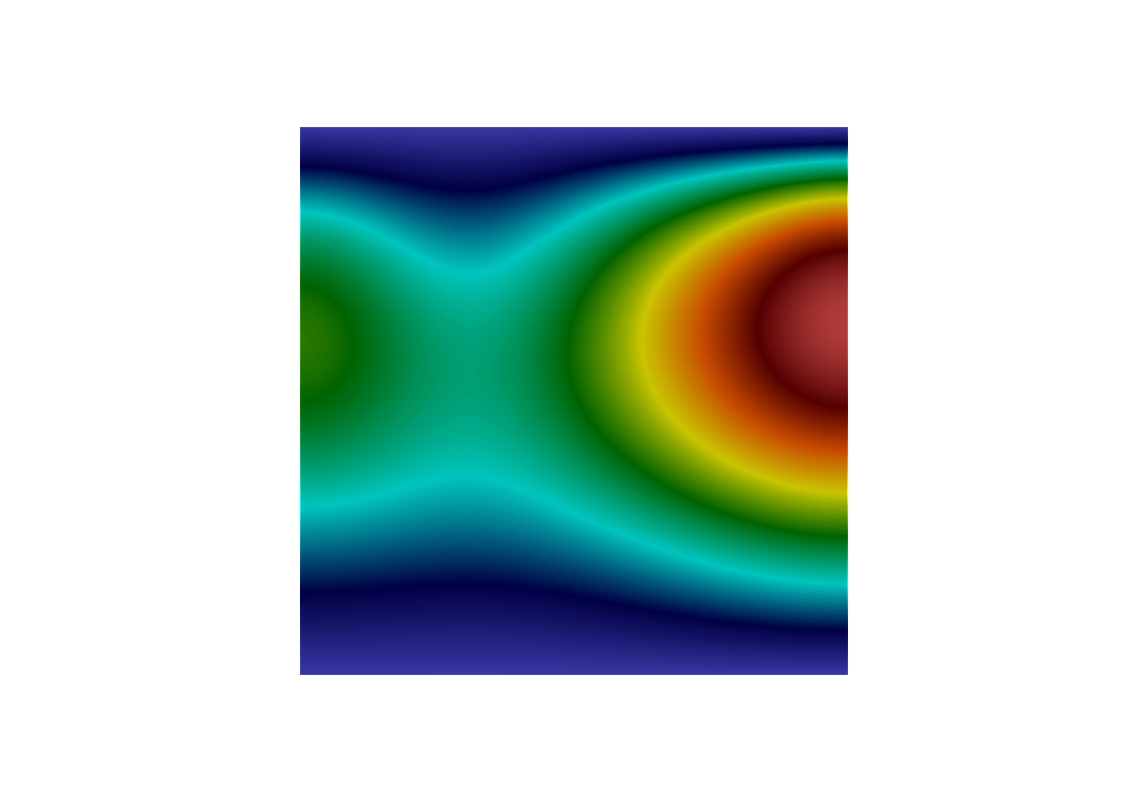}
\caption{Convergence study (top) of the first 5 Jones on the unit square $\Omega_1$ with variable density $\rho(x,y)$. The figures in the middle and bottem represent the x-component (left) and y-component (right) of the Jones eigenmodes associated with $\kappa_{1,h}$ and $\kappa_{2,h}$.}
\label{fig:cv_study_dofs_square-variabledensity}
\end{figure}

\subsection{Jones modes on a disk}\label{subsection:resultsdisk}
We next
present numerical results demonstrating conforming discretizations of both the primal formulation \autoref{eq:primal1} and the mixed formulation \autoref{eq:mixedform} on the disk, where we have used regular triangles.
We consider the unit disk {$\Omega_4 := B(0,1)$ centered at the origin,} with parameters $\mu = 1$, 
$\lambda = 1$, and $\rho = 1$. As discussed in \autoref{sec:rigidmotions}, an eigenmode associated to the 
eigenvalue $\kappa_1 = 0$ is added on the circle (2D case) as a consequence of the symmetry of the domain and the 
condition on the normal trace of the displacement on the boundary. \autoref{fig:spurious} shows the eigenfunction 
$\u_1$ associated to $\kappa_1$. We can see that this displacement is a rigid mode with a pure 
tangential displacement towards the boundary. 

\graphicspath{{./images/EigenvectorPlots/}}
\begin{figure}[!ht]
\centering\includegraphics[width = .5\textwidth, 
height=0.2\textheight]{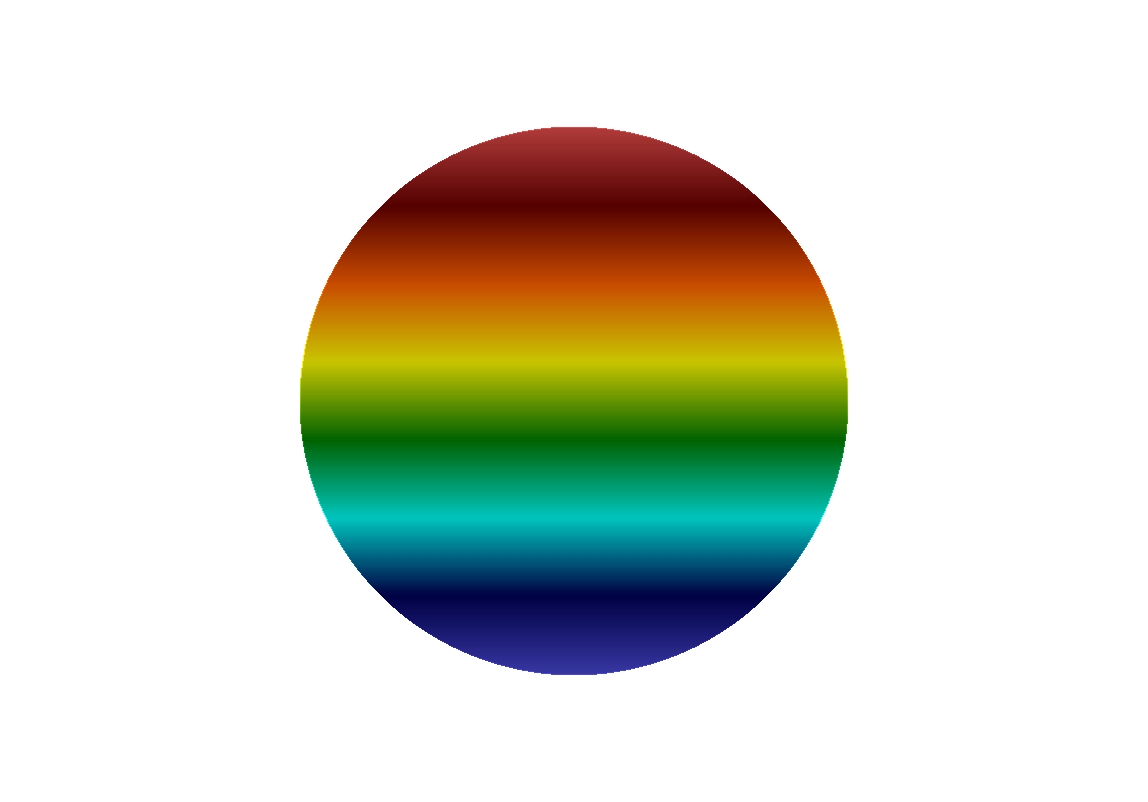}\includegraphics[width = .5\textwidth, 
height=0.2\textheight]{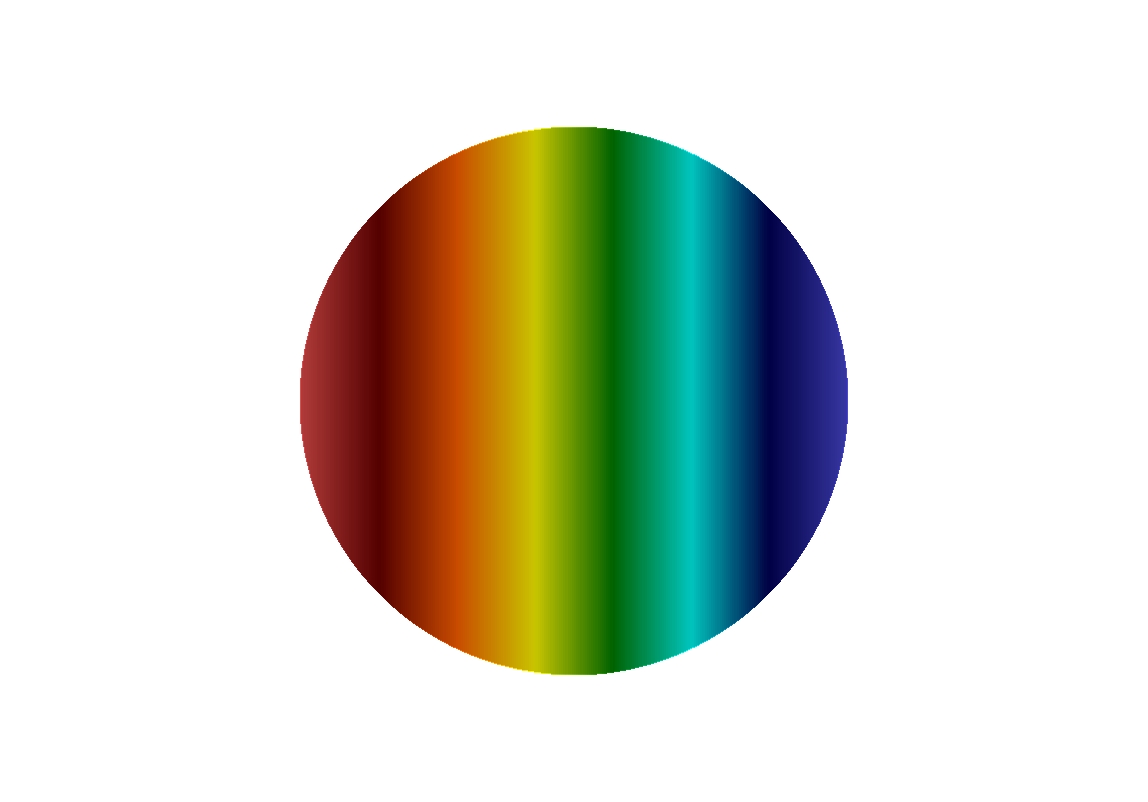}
\caption{$x$-component (left) and $y$-component (right) of the eigenfunction $\u_{h,1}$ 
on the unit disk associated to the eigenvalue $\kappa_1 = 0$. The computations are performed on a grid comprising of 
triangles.}\label{fig:spurious}
\end{figure}

{In \autoref{tab:primalp1}-\autoref{tab:mixedp2}, we present the first 5 Jones eigenmodes, computed on the same meshes (over 5 levels of refinement). We observe that the computed spectrum using the primal formulation \autoref{tab:primalp1}-\autoref{tab:primalp2} misses the first three Jones eigenvalues that the mixed formulation approximates \autoref{tab:mixedp1}-\autoref{tab:mixedp2}. The mixed formulation accurately captures these modes and, in particular, captures the shifted zero eigenvalue that is added to the spectrum when an axisymmetric domain is considered.}

{The convergence history of the first five eigenvalues on the circle is shown for a $\pp_2-P_2$-conforming discrete formulation of \autoref{eq:mixedform}, with predicted rate of convergence of $h^2=$DOFs$^{-1}$. The decay for $\kappa_{h,1}$ is of the order of the tolerance we set in the eigenvalue solver. This comes from the fact that $\kappa_{h,1}$ approximates the zero eigenvalue.}

{
\begin{table}[ht!]
	\centering
	\resizebox{\textwidth}{!}{%
		\begin{tabular}{|c|c|c|c|c|c|}
			\hline
			DOFs  & $\kappa_{1,h}$ & $\kappa_{2,h}$ & $\kappa_{3,h}$ & $\kappa_{4,h}$ & $\kappa_{5,h}$ \\ \hline
						15972 & 12.32475648    & 12.32477023    & 15.6907989     & 28.29273062    & 28.29380655    \\ \hline
			21488 & 12.32412585    & 12.32414851    & 15.68856957    & 28.28731457    & 28.28944094    \\ \hline
			28120 & 12.32366488    & 12.32368938    & 15.68708044    & 28.28443298    & 28.28511814    \\ \hline
			36140 & 12.32329776    & 12.32330826    & 15.68588773    & 28.28155673    & 28.28226112    \\ \hline
			65637 & 12.32312009    & 12.32313548    & 15.68522911    & 28.28039584    & 28.28075982    \\ \hline
		\end{tabular}%
	}
	\caption{First five (shifted) Jones eigenvalues on the unit circle for five levels of refinement, using the P1 conforming scheme (vectorial Lagrange elements) of the primal formulation.}
	\label{tab:primalp1}
\end{table}
\begin{table}[ht!]
	\centering
	\resizebox{\textwidth}{!}{%
		\begin{tabular}{|c|c|c|c|c|c|}
			\hline
			DOFs   & $\kappa_{1,h}$ & $\kappa_{2,h}$ & $\kappa_{3,h}$ & $\kappa_{4,h}$ & $\kappa_{5,h}$ \\ \hline
					63282  & 12.31982312    & 12.3198605     & 15.60885775    & 28.27462732    & 28.27464639    \\ \hline
			85246  & 12.31995092    & 12.32000357    & 15.61697078    & 28.27415568    & 28.27415894    \\ \hline
			111674 & 12.31996175    & 12.3200125     & 15.62039435    & 28.27382317    & 28.27384615    \\ \hline
			143654 & 12.3202651     & 12.32030354    & 15.62961012    & 28.27365126    & 28.27366316    \\ \hline
			261039 & 12.32050245    & 12.32052281    & 15.63657853    & 28.27352834    & 28.27353722    \\ \hline
		\end{tabular}%
	}
	\caption{First five (shifted) Jones eigenvalues on the unit circle for five levels of refinement, using the P2 conforming scheme (vectorial Lagrange elements) of the primal formulation.}
	\label{tab:primalp2}
\end{table}
\begin{table}[ht!]
	\centering
	\resizebox{\textwidth}{!}{%
		\begin{tabular}{|c|c|c|c|c|c|}
			\hline
			DOFs  & $\kappa_{1,h}$ & $\kappa_{2,h}$ & $\kappa_{3,h}$ & $\kappa_{4,h}$ & $\kappa_{5,h}$ \\ \hline
					23958 & 1.000000000    & 6.189626972    & 6.189643368    & 13.15499546    & 13.15505685    \\ \hline
			32232 & 1.000000000    & 6.189248857    & 6.189257759    & 13.15390925    & 13.15400431    \\ \hline
			42180 & 1.000000000    & 6.188970148    & 6.188978719    & 13.15310588    & 13.1531842     \\ \hline
			54210 & 1.000000000    & 6.188735065    & 6.188747593    & 13.15246565    & 13.15249264    \\ \hline
			65637 & 1.000000000    & 6.188634352    & 6.188638195    & 13.15215436    & 13.15220753    \\ \hline
		\end{tabular}%
	}
	\caption{First five (shifted) Jones eigenvalues on the unit circle for ten different levels of refinement, using a P1-P1 scheme of the mixed formulation.}
	\label{tab:mixedp1}
\end{table}
\begin{table}[ht!]
	\centering
	\resizebox{\textwidth}{!}{%
		\begin{tabular}{|c|c|c|c|c|c|}
			\hline
			DOFs   & $\kappa_{1,h}$ & $\kappa_{2,h}$ & $\kappa_{3,h}$ & $\kappa_{4,h}$ & $\kappa_{5,h}$ \\ \hline
			94923  & 1.000000000    & 6.188425852    & 6.188425903    & 13.15133728    & 13.15133729    \\ \hline
			127869 & 1.000000000    & 6.18832008     & 6.188320227    & 13.15110141    & 13.15110143    \\ \hline
			167511 & 1.000000000    & 6.188251842    & 6.188251917    & 13.15094836    & 13.15094836    \\ \hline
			215481 & 1.000000000    & 6.188205638    & 6.188205663    & 13.15084343    & 13.15084343    \\ \hline
			261039 & 1.000000000    & 6.188172645    & 6.188172683    & 13.15076839    & 13.15076839    \\ \hline
		\end{tabular}%
	}
	\caption{First five (shifted) Jones eigenvalues on the unit circle for ten different levels of refinement, using a P2-P2 scheme of the mixed formulation.}
	\label{tab:mixedp2}
\end{table}
}

\graphicspath{{./images/cv-study/}}
\begin{figure}[!ht]\centering
\includegraphics[width = 1.0\textwidth, 
height=0.35\textheight]{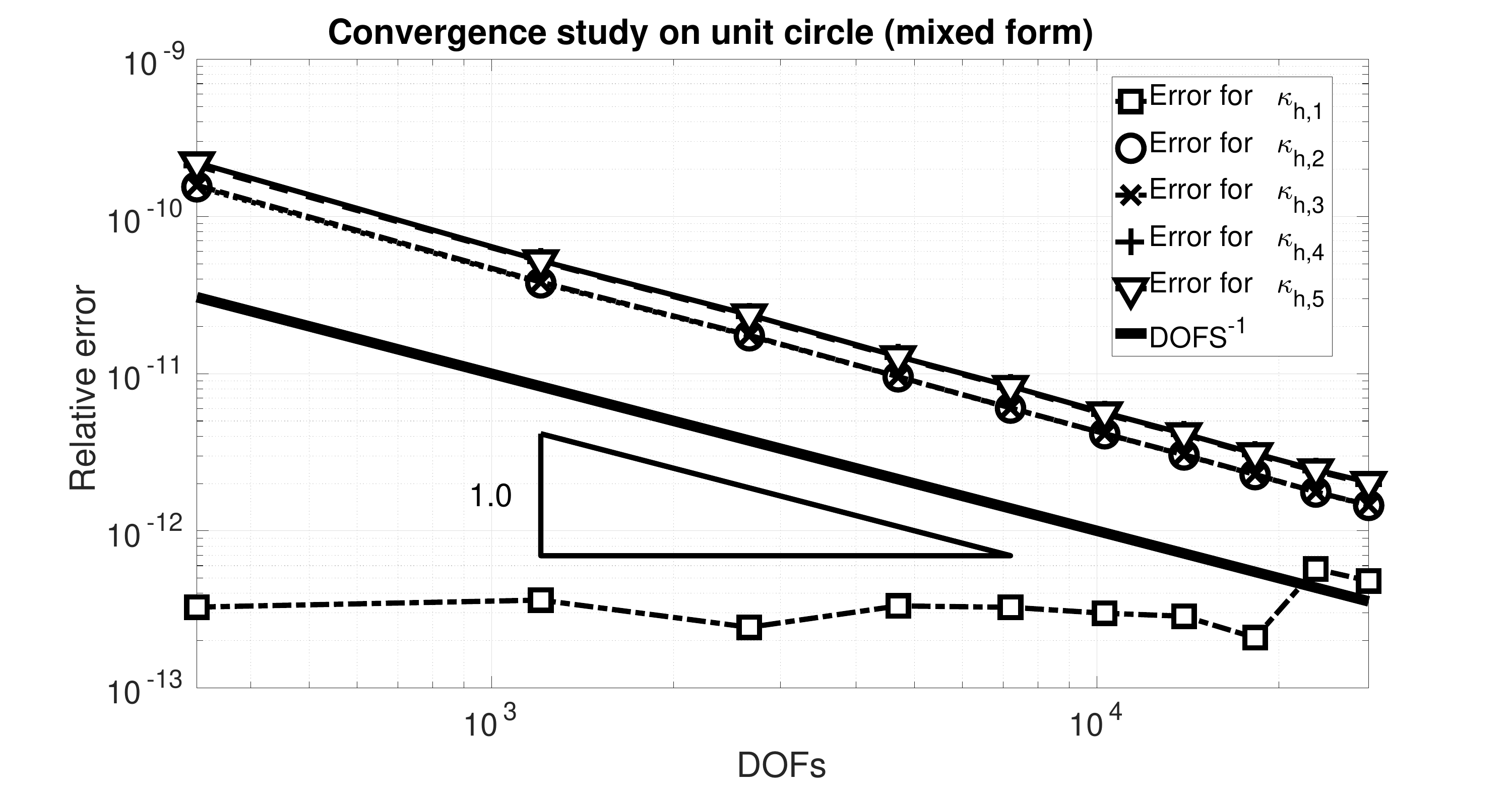}
\caption{Convergence study for the first 5 eigenvalues of 
\autoref{eq:jones} on the unit circle using the mixed formulation with quadratic elements. The approximated eigenvalue $\kappa_{h,1}$ corresponds to the eigenvalue 1.0 in the shifted formulation of \autoref{eq:mixedform}.}\label{fig:cvhistorycircle}
\end{figure}

\subsection{Shear and compression modes of Jones eigenpairs}
We end this section by using {the discrete formulation in \autoref{eq:discrete1}} to compute Jones modes on some 
geometries, to explore the 
dependence of the spectrum on domain shape and the Lam\'e parameters. We also report the $L^2$-norm of the divergence 
and rotational of the computed fields. In simple shapes{, as the rectangle,} the Jones modes can be readily 
identified as pure $s$- or 
$p$- modes. The eigenvalues can also have multiplicity, and the eigenspaces may include eigenmodes of both types.  
This points to the need for care with resolving eigenmodes, as with any problem involving multiple eigenvalues or clusters of these.

As found in \autoref{sec:formulation}, the Jones eigenvalue problem on Lipschitz domains possesses a countable set of 
eigenvalues $w_{m\ell}^2$, $m,\ell = 0,1,2,\ldots$. For the sake of presentation, we set $\nu_j = 
w^2_{m\ell}$, with $j = j(m,\ell)\in\nnn_0$, such that $\nu_j \leq \nu_{j+1}$, for all $j\in\nnn_0$.

{In \autoref{subsec:harge}, we showed that eigenpairs given by \autoref{eq:smode} and 
\autoref{eq:pmode} are part of the spectrum of the Jones eigenproblem on a rectangle. By changing the Lam\'e parameters, one expects to change not only the eigenvalues, but potentially also the geometric multiplicity of eigenspaces. An example of this behaviour can be seen in \autoref{table:lambda1square} and \autoref{table:lambda2square} where we show the first seven approximated Jones eigenpairs on the unit square with parameters $\mu = \rho = 1$, and $\lambda\in\{1,2\}$. }

\graphicspath{{./images/EigenvectorPlots/}}
\small{\begin{table}[!ht]
\centering
\begin{tabular}{|c|c|c|c|c|c|c|}
\hline
 $j$& $\nu_j$ & $\nu_j/\pi^2$ & $\| \div\, \u\|_0^2$ & $\|\curl\,\u\|_0^2$ & $x-$component & $y-$component 
\\\hline
 1& 19.74 & 2.000 & 5.048e-08 & 19.72 & \raisebox{-.5\height}{\includegraphics[width = .15\textwidth, 
height=0.04\textheight]{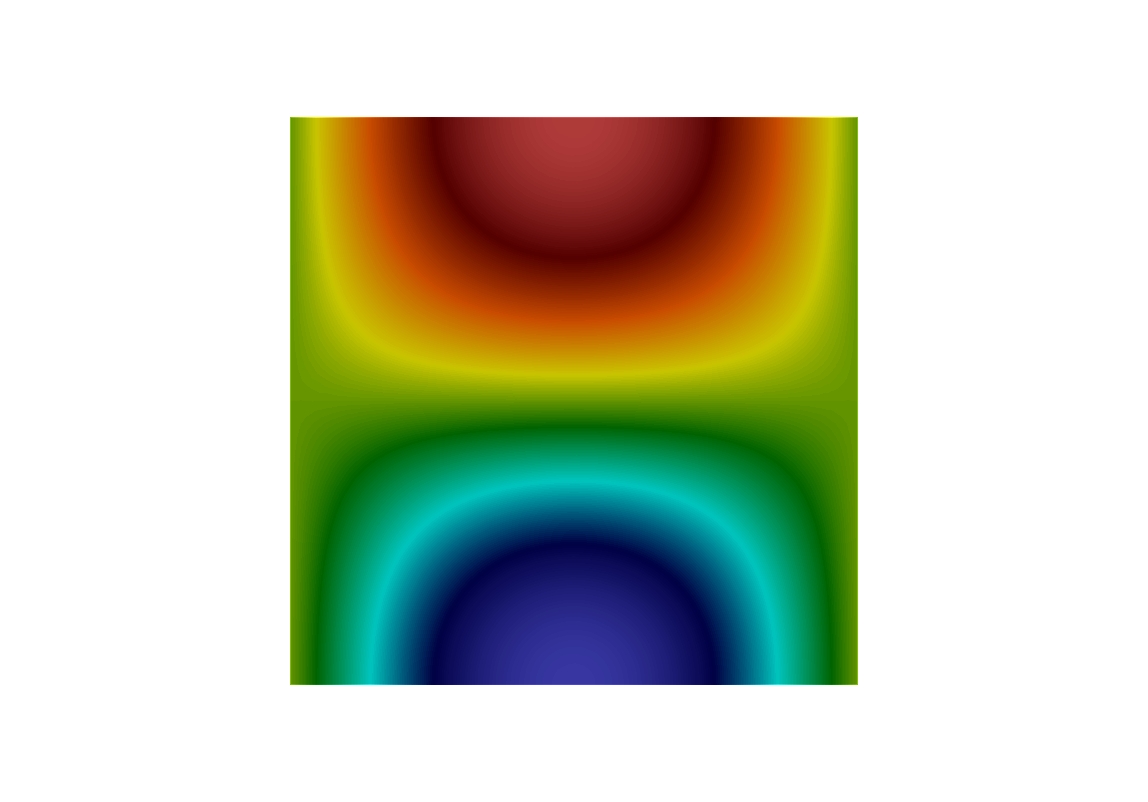}} & 
\raisebox{-.5\height}{\includegraphics[width = .15\textwidth, 
height=0.04\textheight]{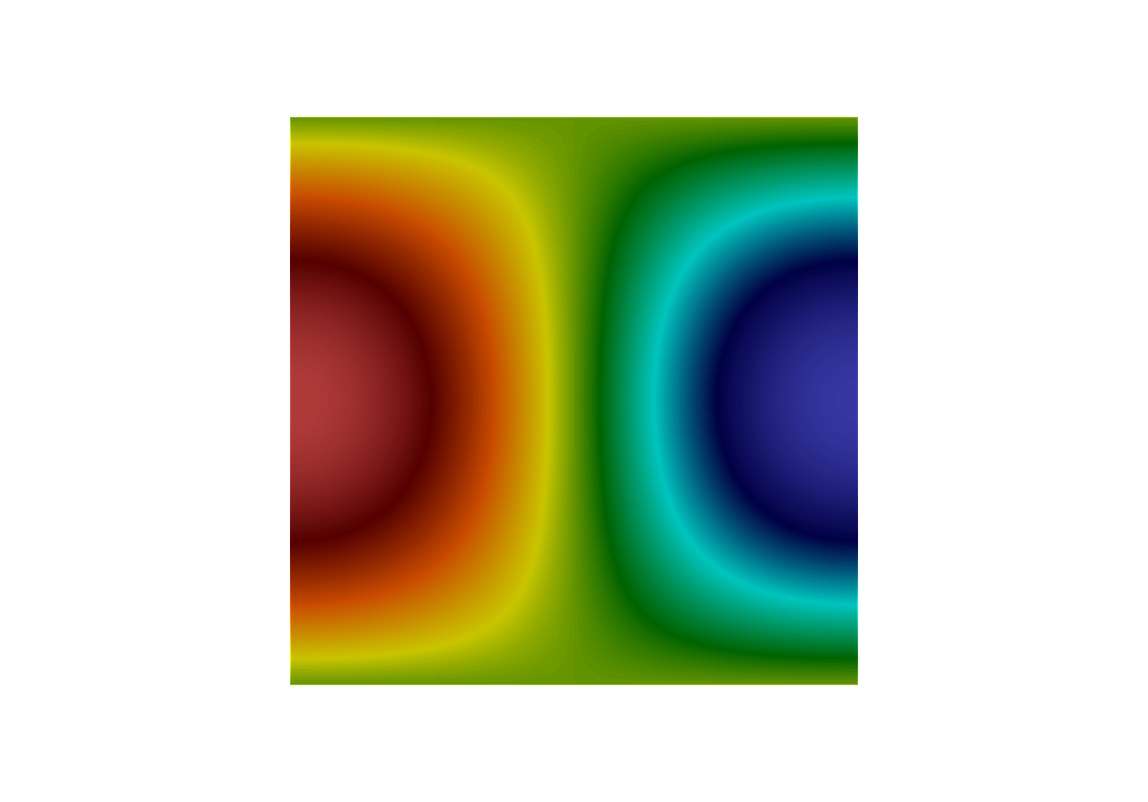}}\\
2 & 29.61 & 3.000 & 9.870 & 0.000188 & \raisebox{-.5\height}{\includegraphics[width = .15\textwidth, 
height=0.04\textheight]{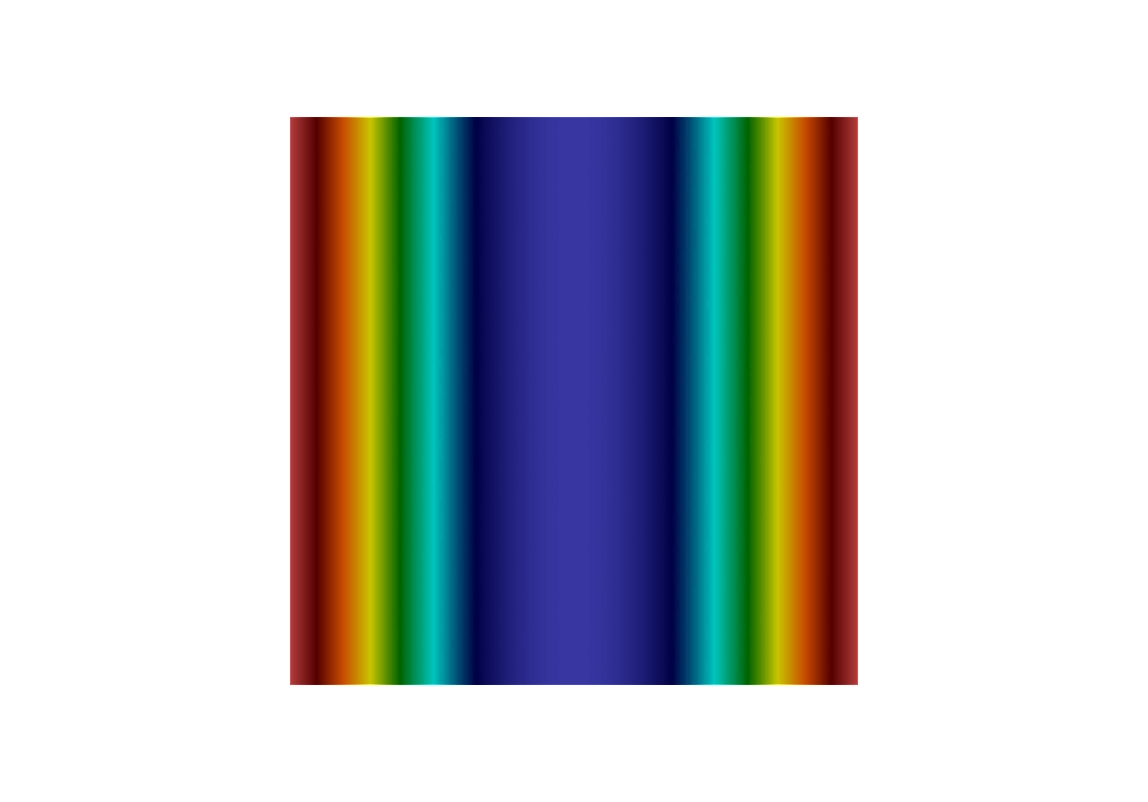}} & 
\raisebox{-.5\height}{\includegraphics[width = .15\textwidth, 
height=0.04\textheight]{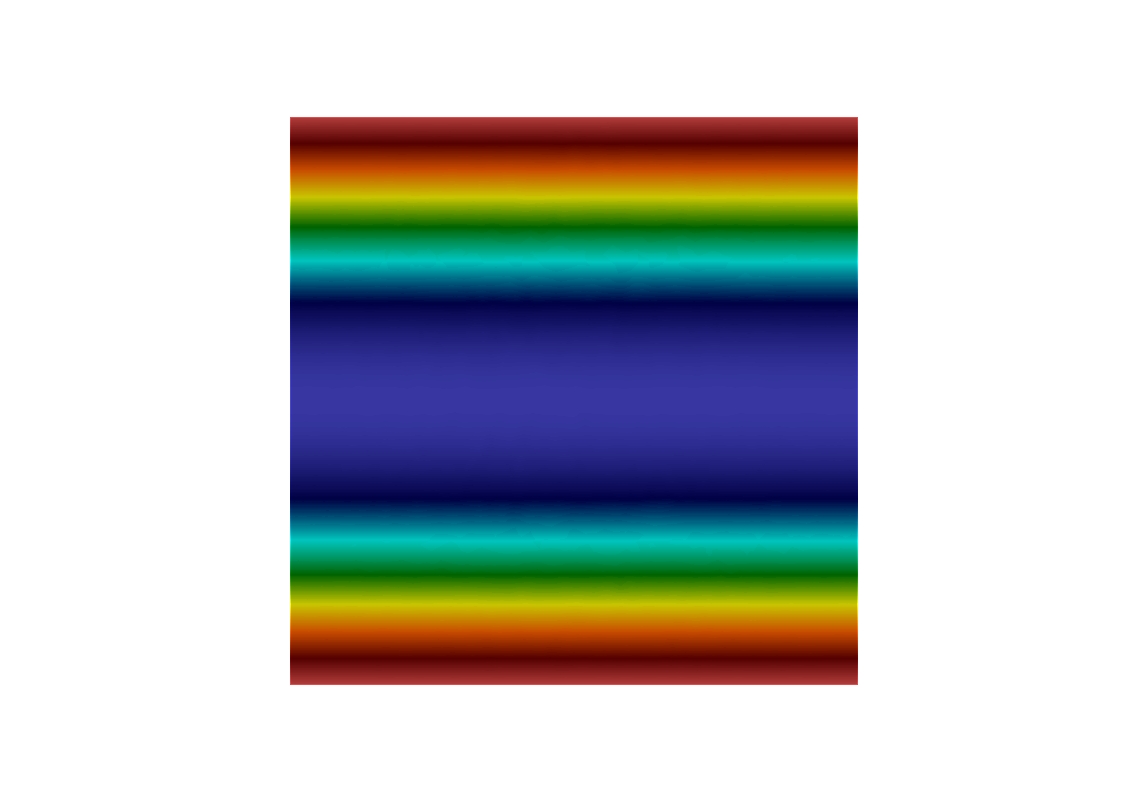}}\\
3 & 29.61 & 3.000 & 9.870 & 0.0001293 & \raisebox{-.5\height}{\includegraphics[width = .15\textwidth, 
height=0.04\textheight]{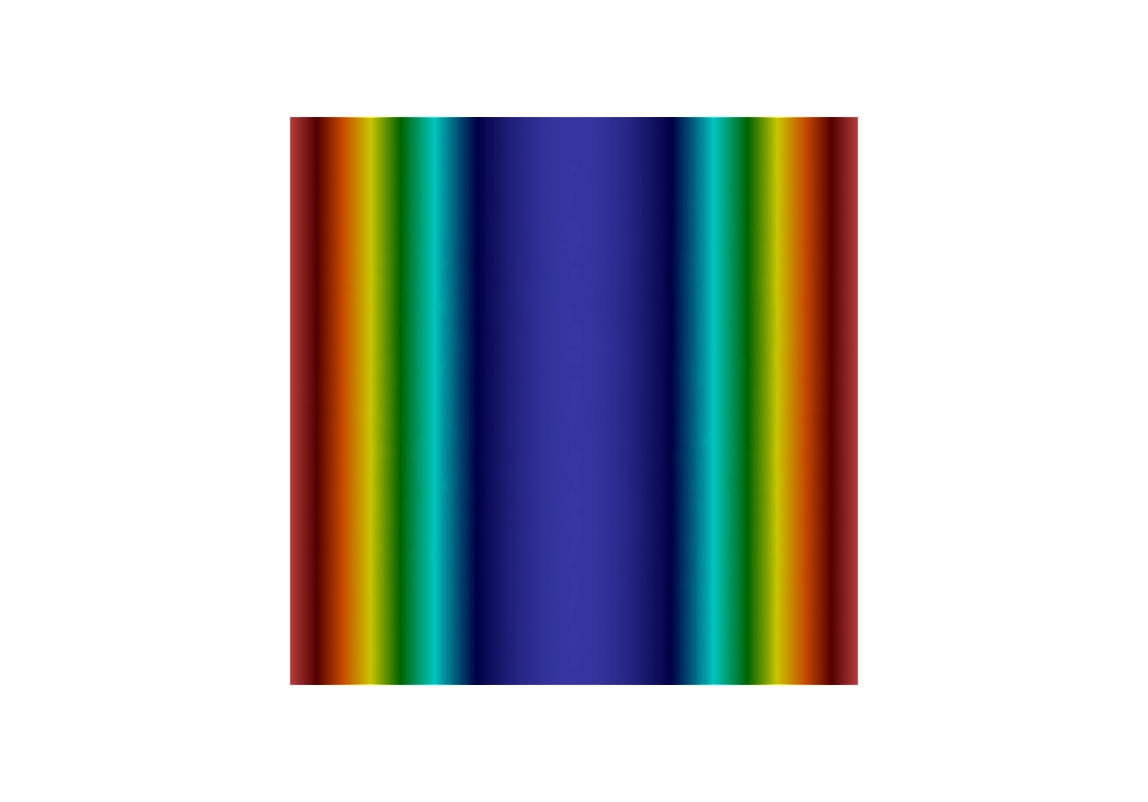}} & 
\raisebox{-.5\height}{\includegraphics[width = .15\textwidth, 
height=0.04\textheight]{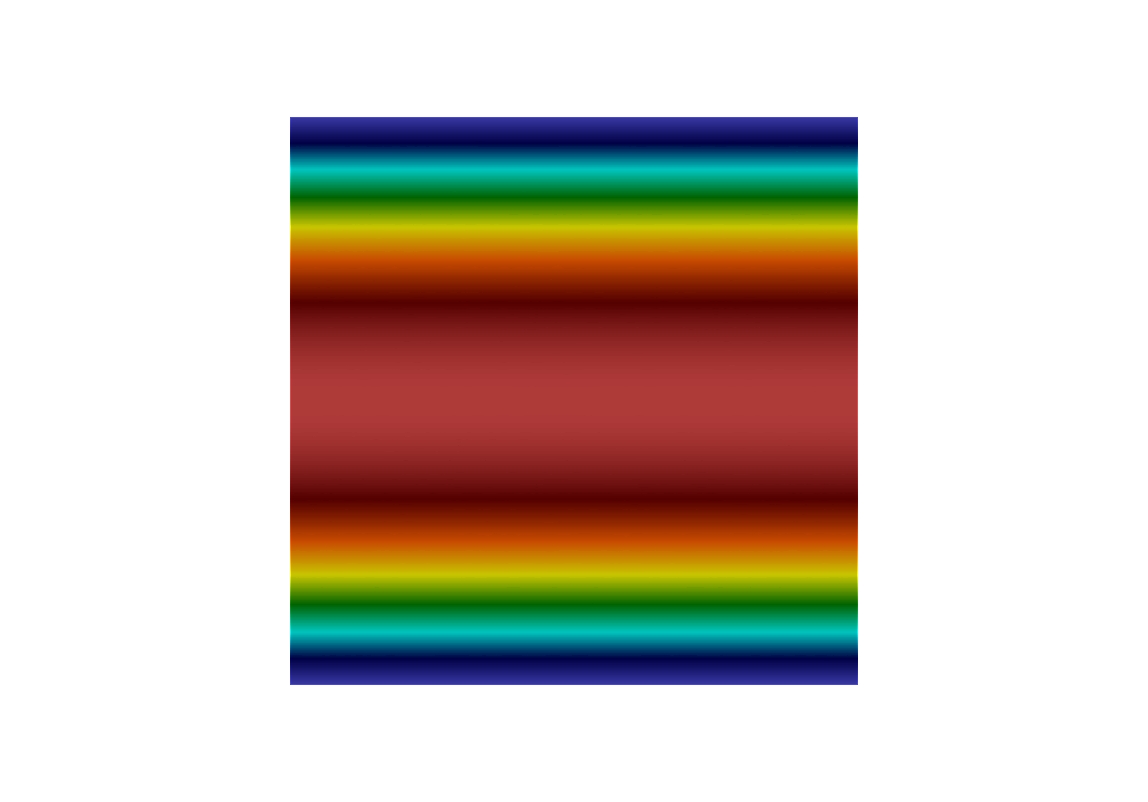}}\\
4 & 49.35 & 5.000 & 7.168e-07 & 49.22 & \raisebox{-.5\height}{\includegraphics[width = .15\textwidth, 
height=0.04\textheight]{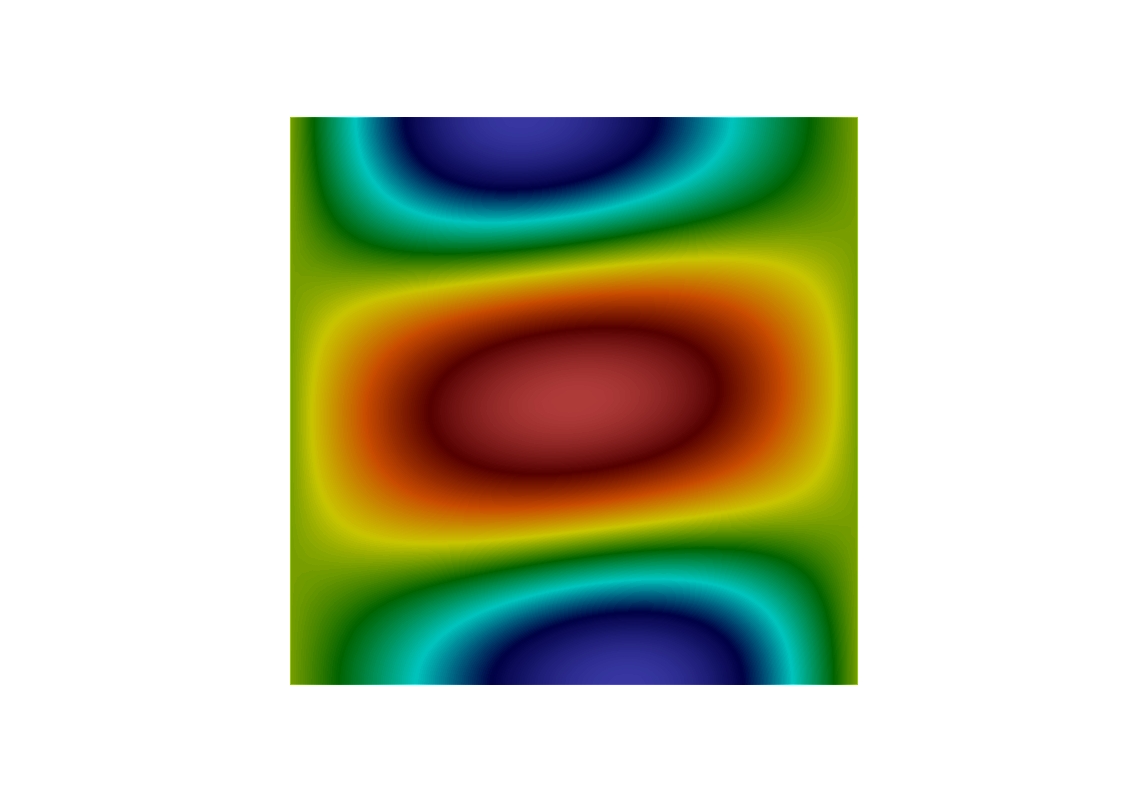}} & 
\raisebox{-.5\height}{\includegraphics[width = .15\textwidth, 
height=0.04\textheight]{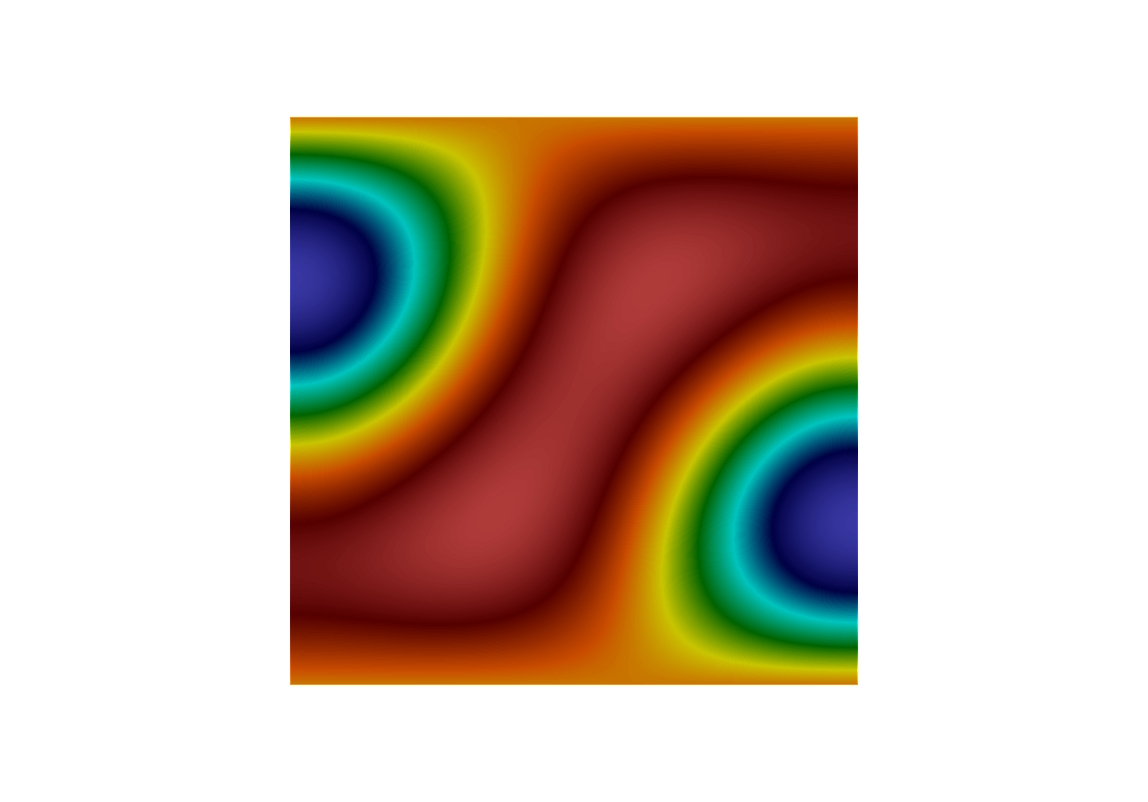}}\\
5 & 49.35 & 5.000 & 8.863e-07 & 49.25 & \raisebox{-.5\height}{\includegraphics[width = .15\textwidth, 
height=0.04\textheight]{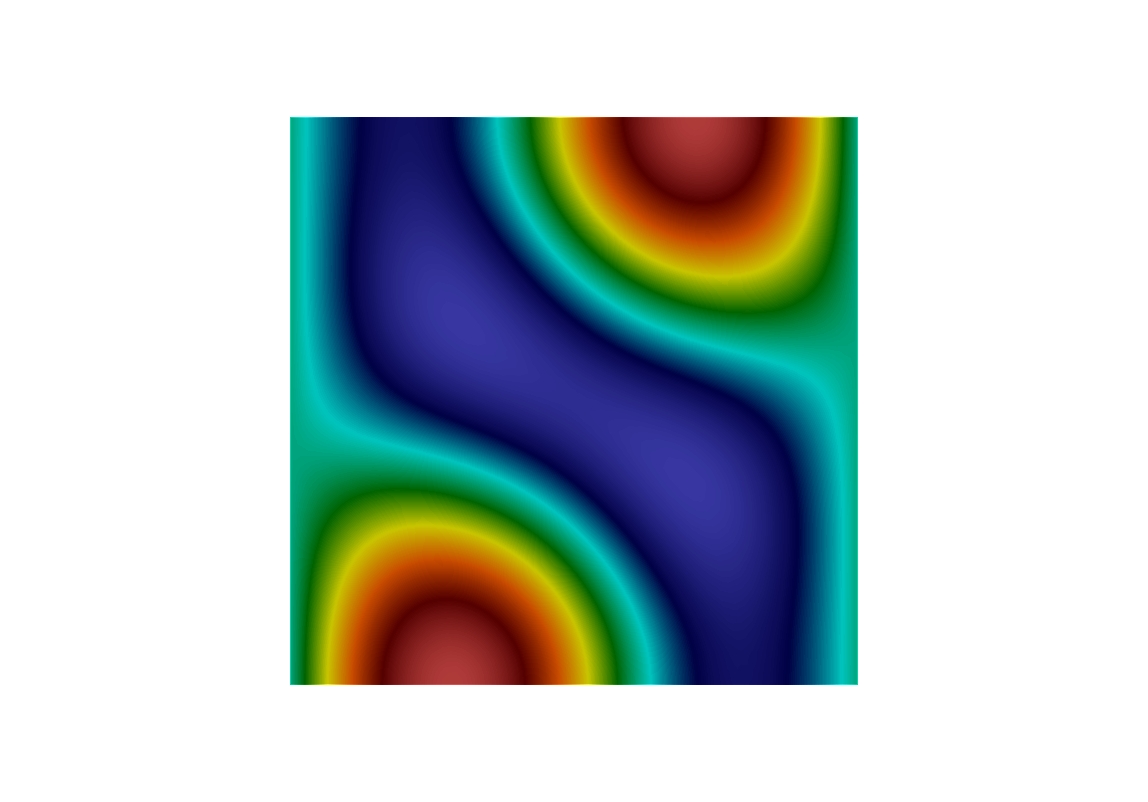}} & 
\raisebox{-.5\height}{\includegraphics[width = .15\textwidth, 
height=0.04\textheight]{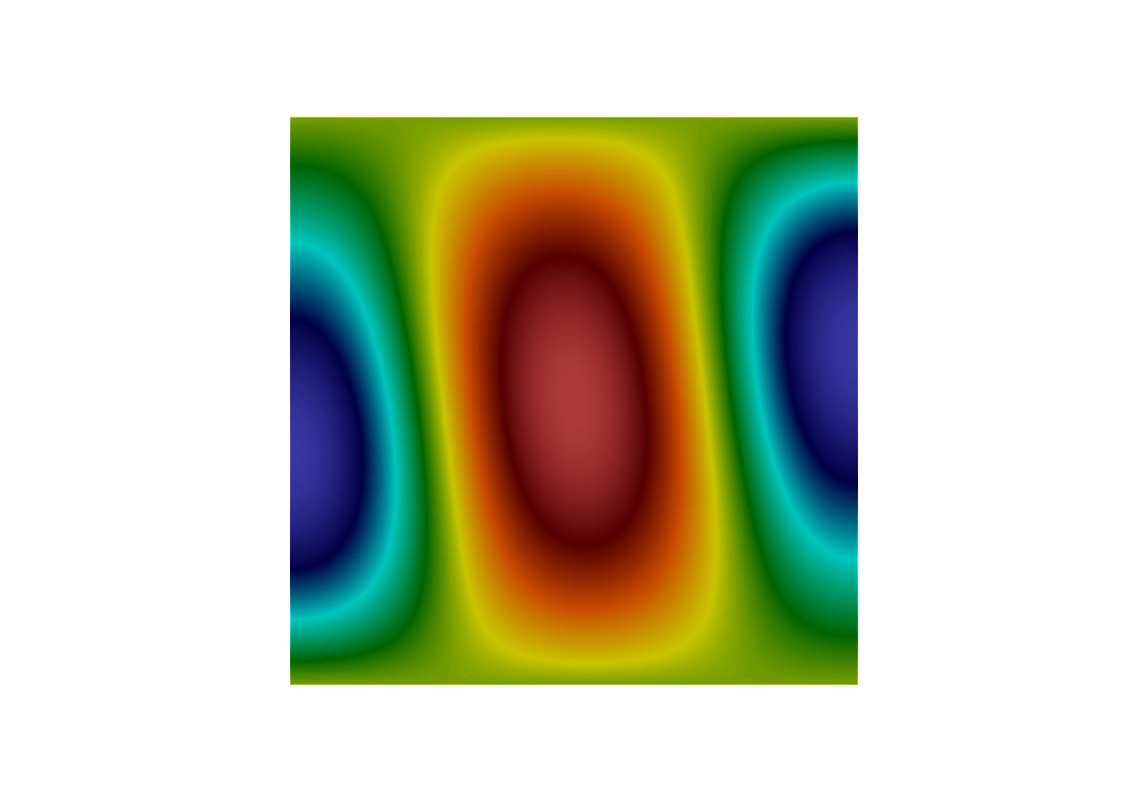}}\\
6 & 59.22 & 6.000 & 19.74 & 0.0005061 & \raisebox{-.5\height}{\includegraphics[width = .15\textwidth, 
height=0.04\textheight]{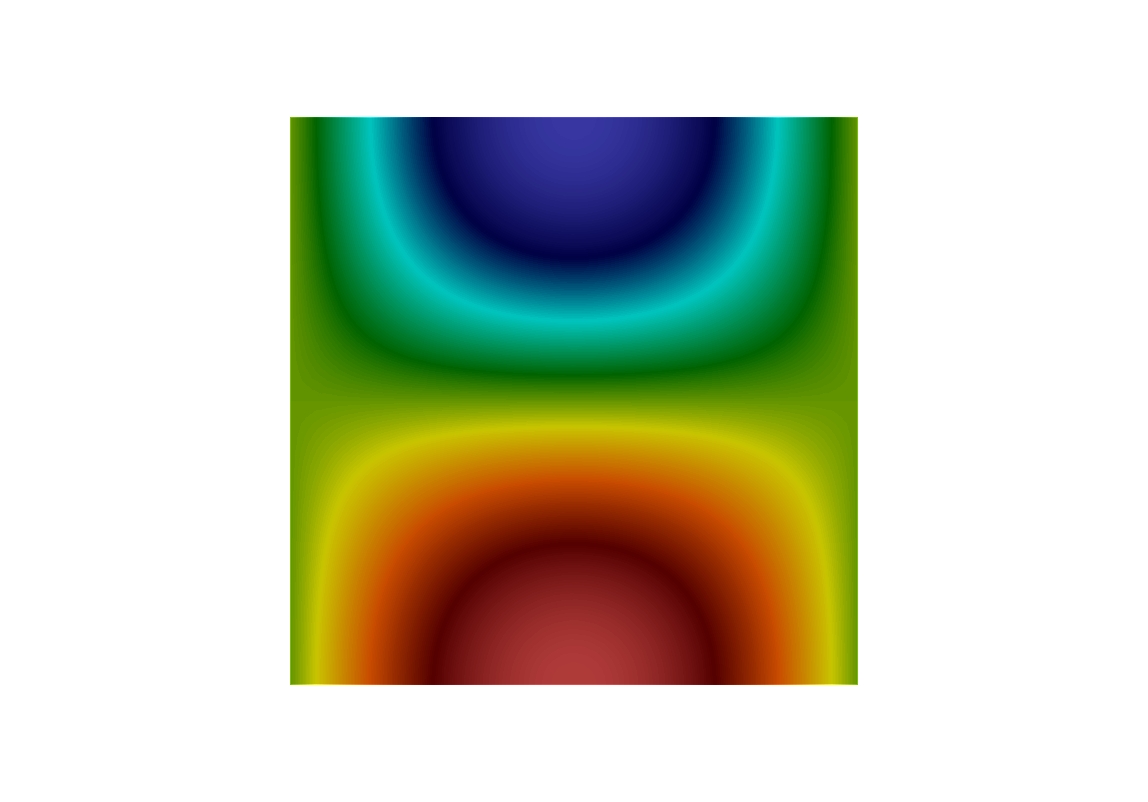}} & 
\raisebox{-.5\height}{\includegraphics[width = .15\textwidth, 
height=0.04\textheight]{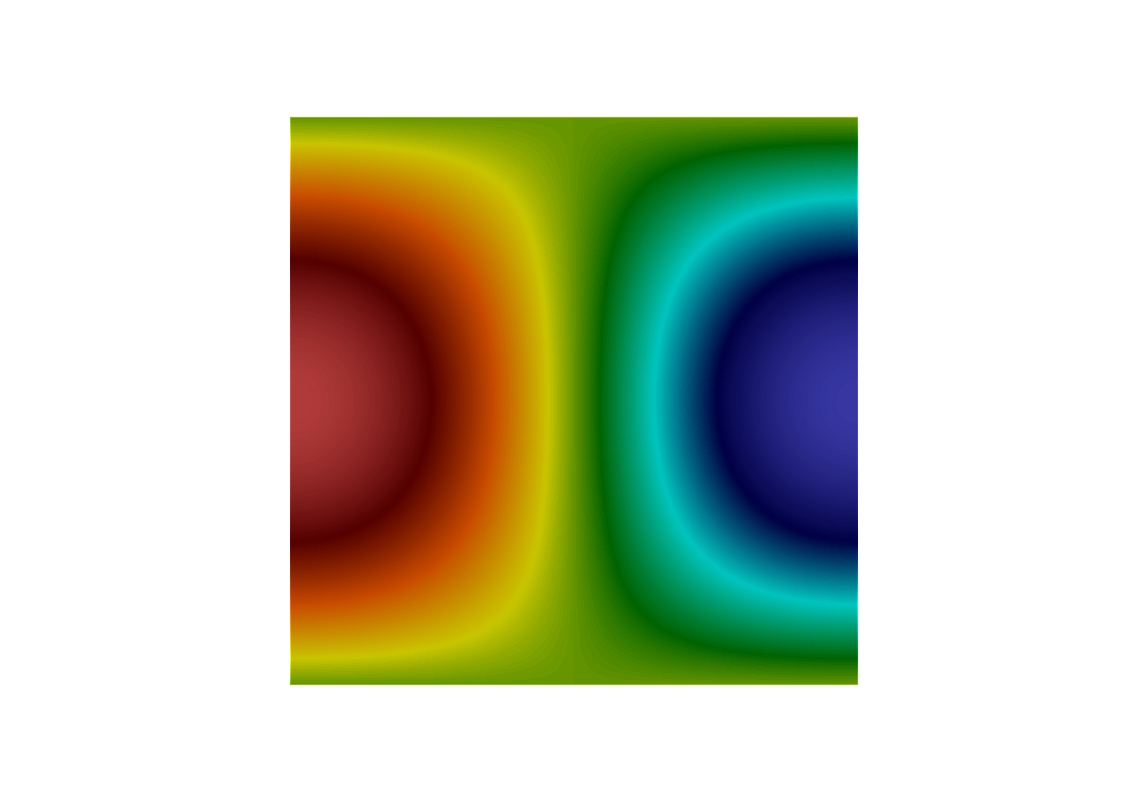}}\\
7 & 78.96 & 8.000 & 2.979e-06 & 78.78 & \raisebox{-.5\height}{\includegraphics[width = .15\textwidth, 
height=0.04\textheight]{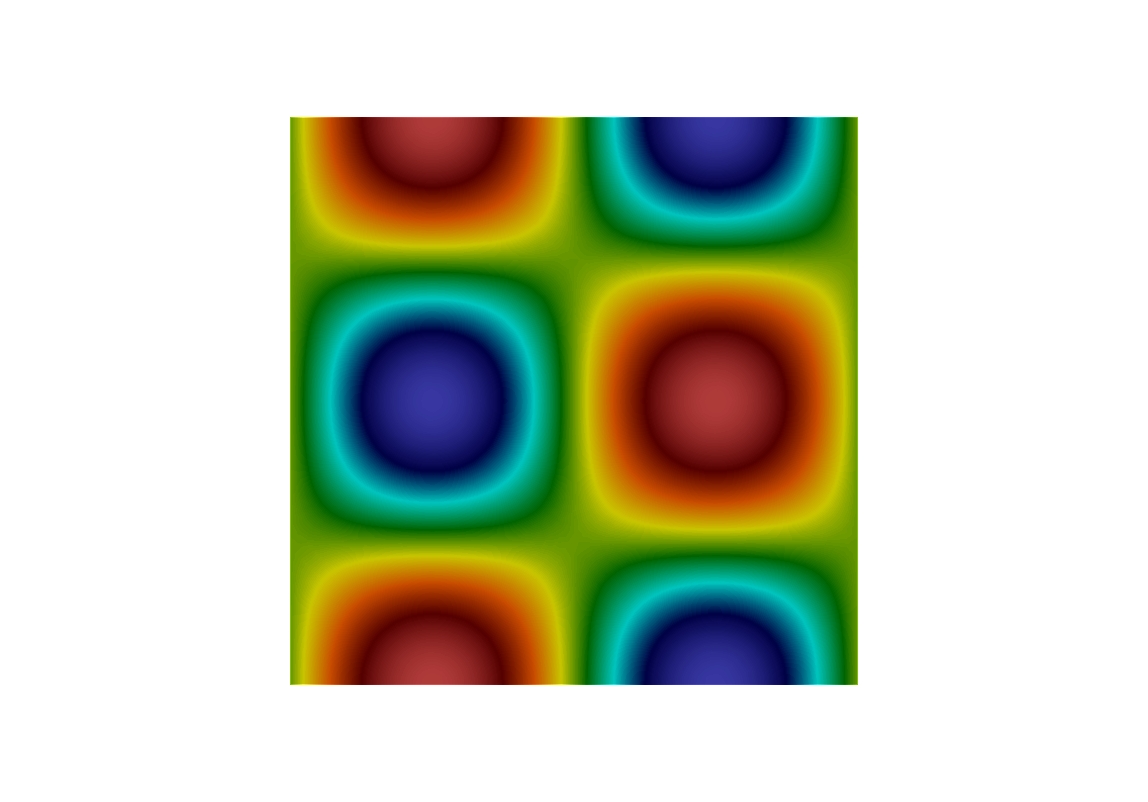}} & 
\raisebox{-.5\height}{\includegraphics[width = .15\textwidth, 
height=0.04\textheight]{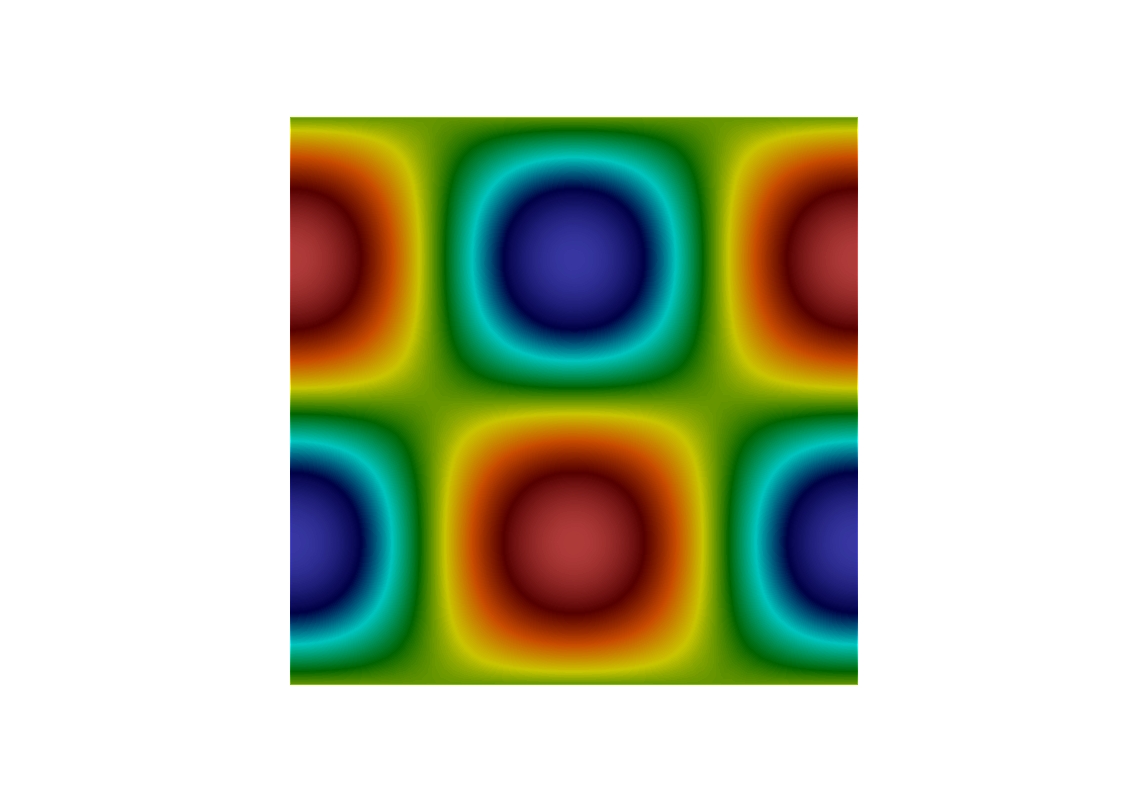}}\\\hline
\end{tabular}\label{table:lambda1square}\caption{First 7 eigemodes on unit square with parameters $\mu=\lambda=\rho= 1$. $\nu_2=\nu_3$, and the corresponding eigenspace has pure $p$ modes. Also, $\nu_4=\nu_5$, and the corresponding eigenspace has pure $s$ modes.  }
\end{table}}

\graphicspath{{./images/EigenvectorPlots/}}
\small{\begin{table}[!ht]
\centering
\begin{tabular}{|c|c|c|c|c|c|c|}
\hline
 $j$& $\nu_j$ & $\nu_j/\pi^2$ & $\| \div\, \u\|_0^2$ & $\|\curl\,\u\|_0^2$ & $x-$component & $y-$component 
\\\hline
 1& 19.74 & 2.000 & 5.048e-08 & 19.72 & \raisebox{-.5\height}{\includegraphics[width = .15\textwidth, 
height=0.04\textheight]{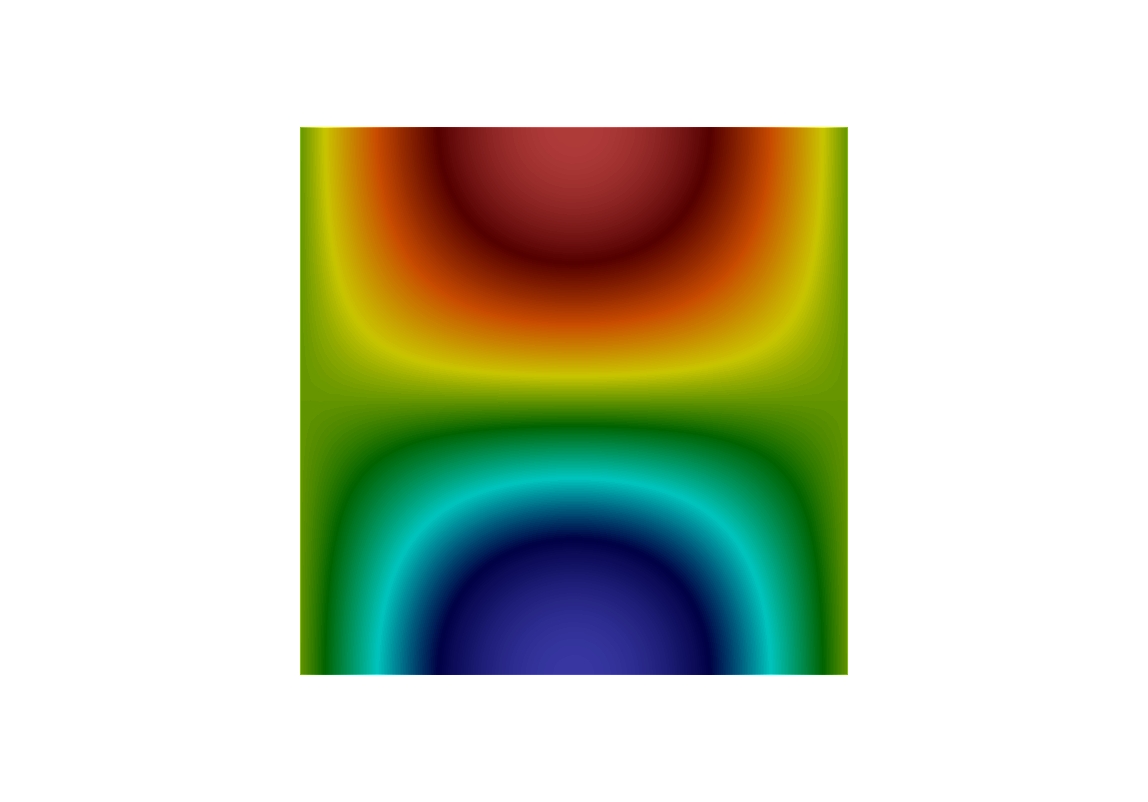}} & 
\raisebox{-.5\height}{\includegraphics[width = .15\textwidth, 
height=0.04\textheight]{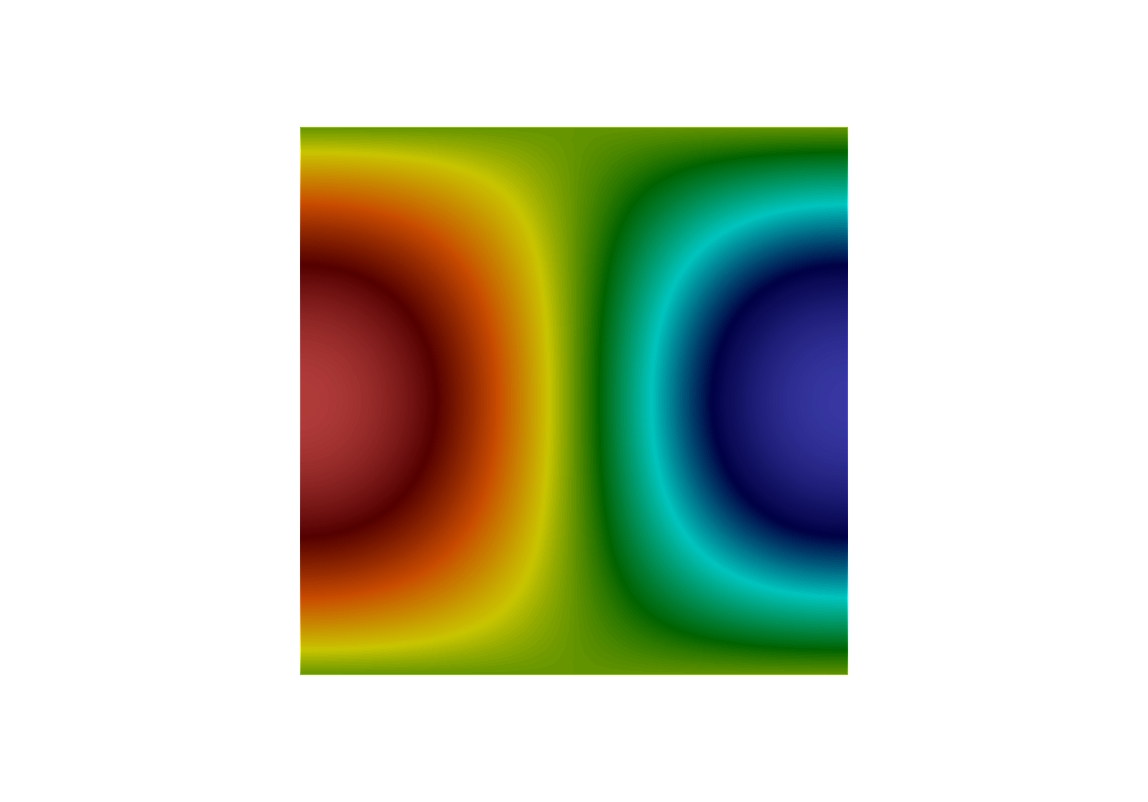}}\\
2 & 39.48 & 4.000 & 9.870 & 0.000188 & \raisebox{-.5\height}{\includegraphics[width = .15\textwidth, 
height=0.04\textheight]{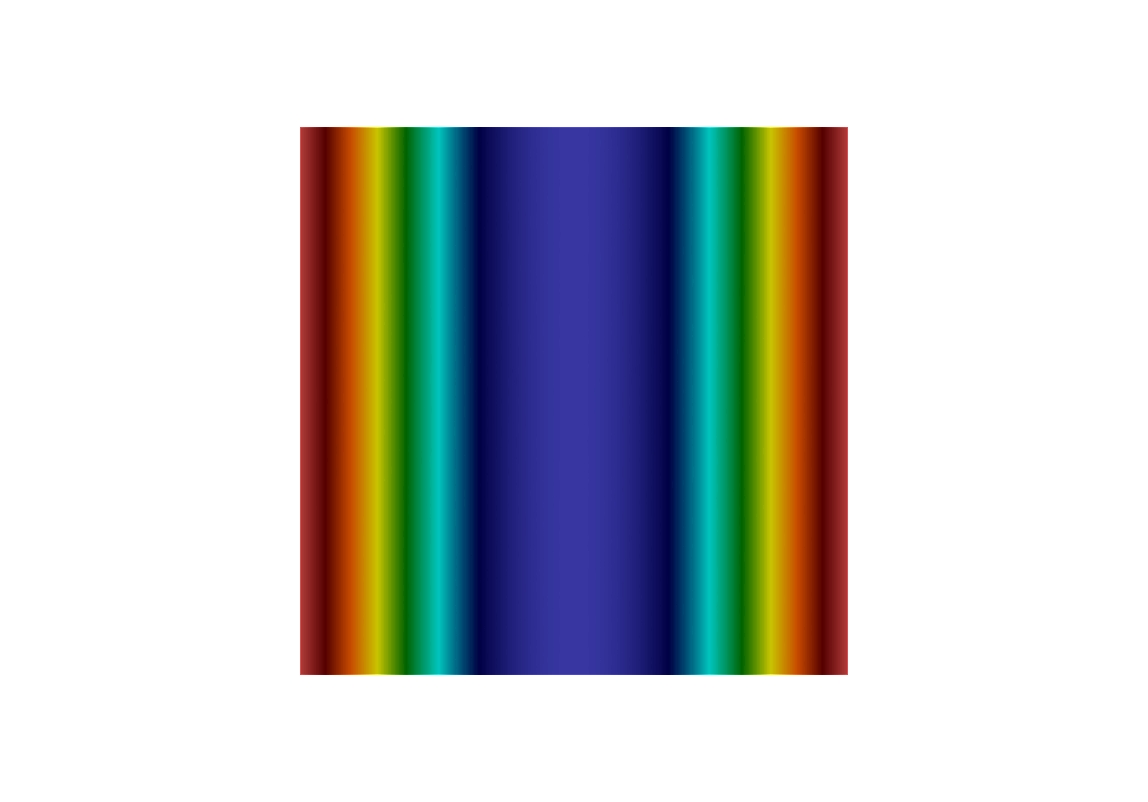}} & 
\raisebox{-.5\height}{\includegraphics[width = .15\textwidth, 
height=0.04\textheight]{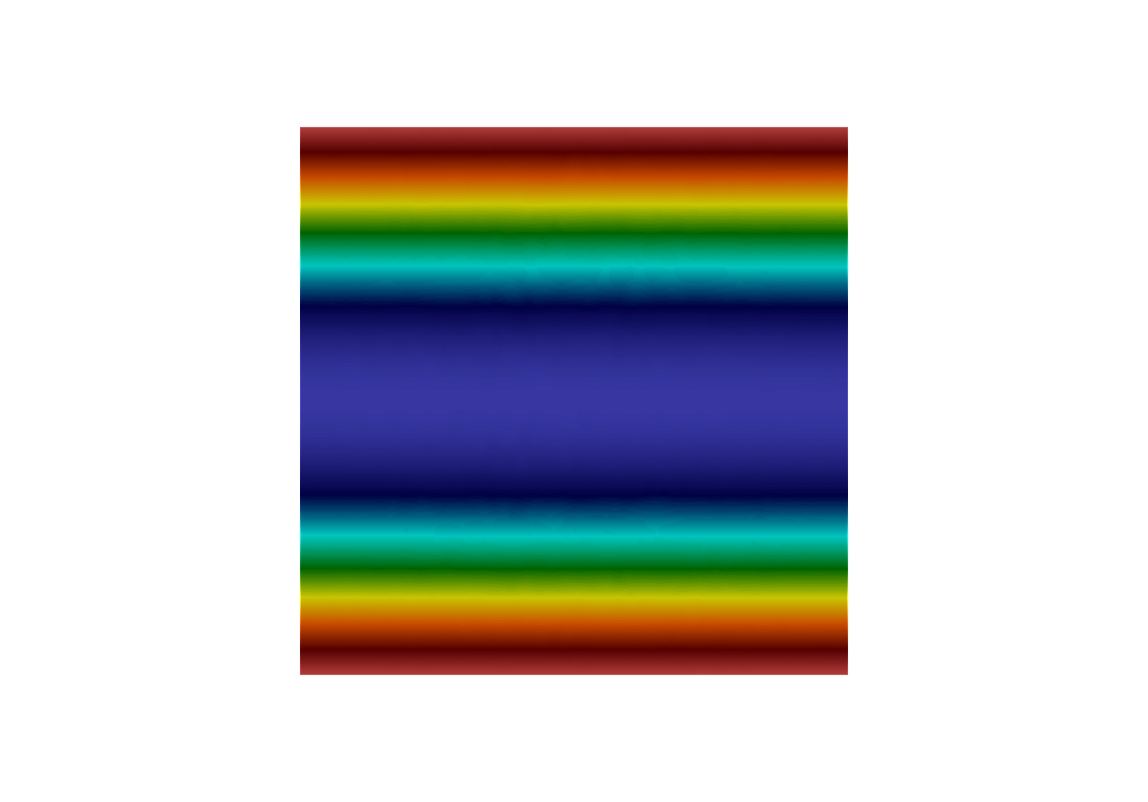}}\\
3 & 39.48 & 4.000 & 9.870 & 0.0001294 & \raisebox{-.5\height}{\includegraphics[width = .15\textwidth, 
height=0.04\textheight]{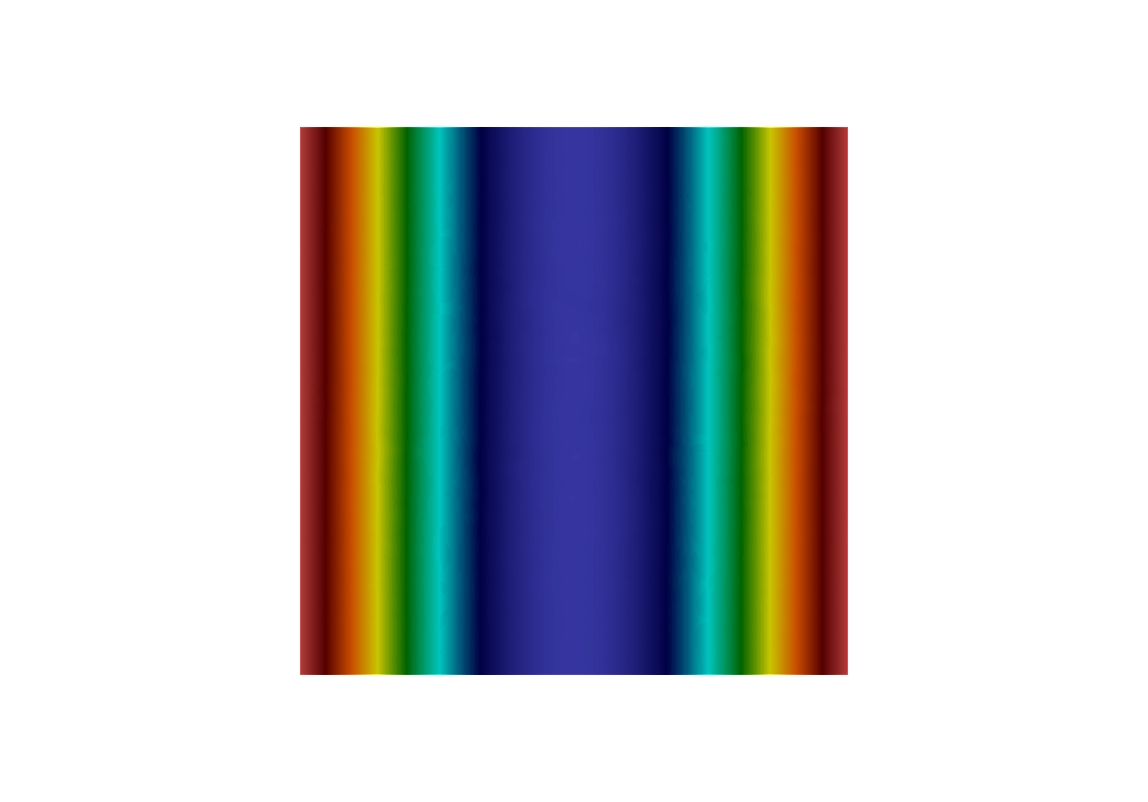}} & 
\raisebox{-.5\height}{\includegraphics[width = .15\textwidth, 
height=0.04\textheight]{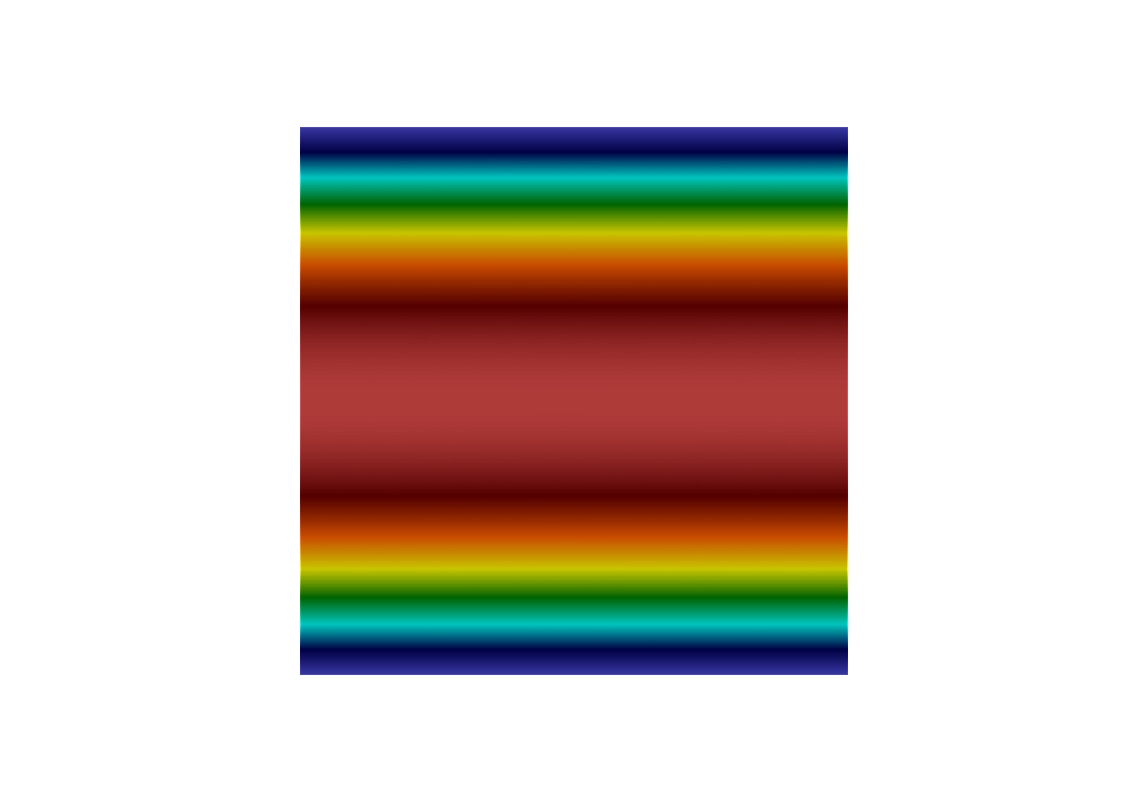}}\\
4 & 49.35 & 5.000 & 6.414e-07 & 49.22 & \raisebox{-.5\height}{\includegraphics[width = .15\textwidth, 
height=0.04\textheight]{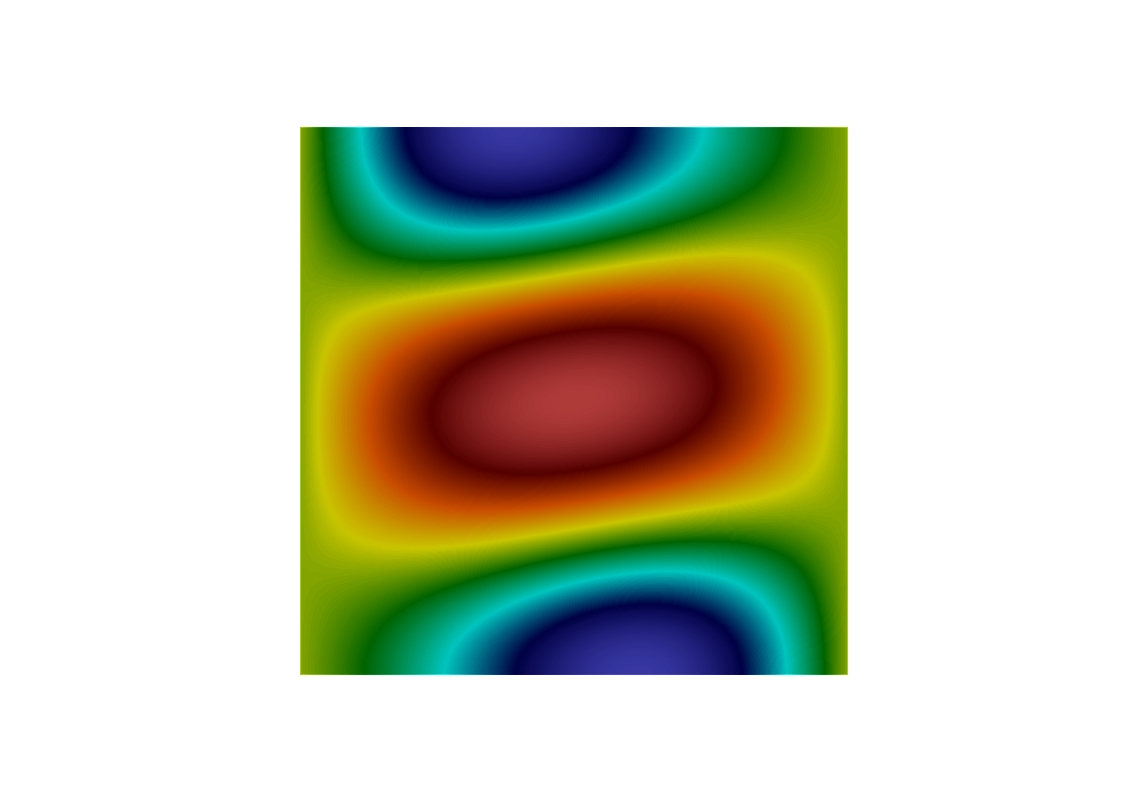}} & 
\raisebox{-.5\height}{\includegraphics[width = .15\textwidth, 
height=0.04\textheight]{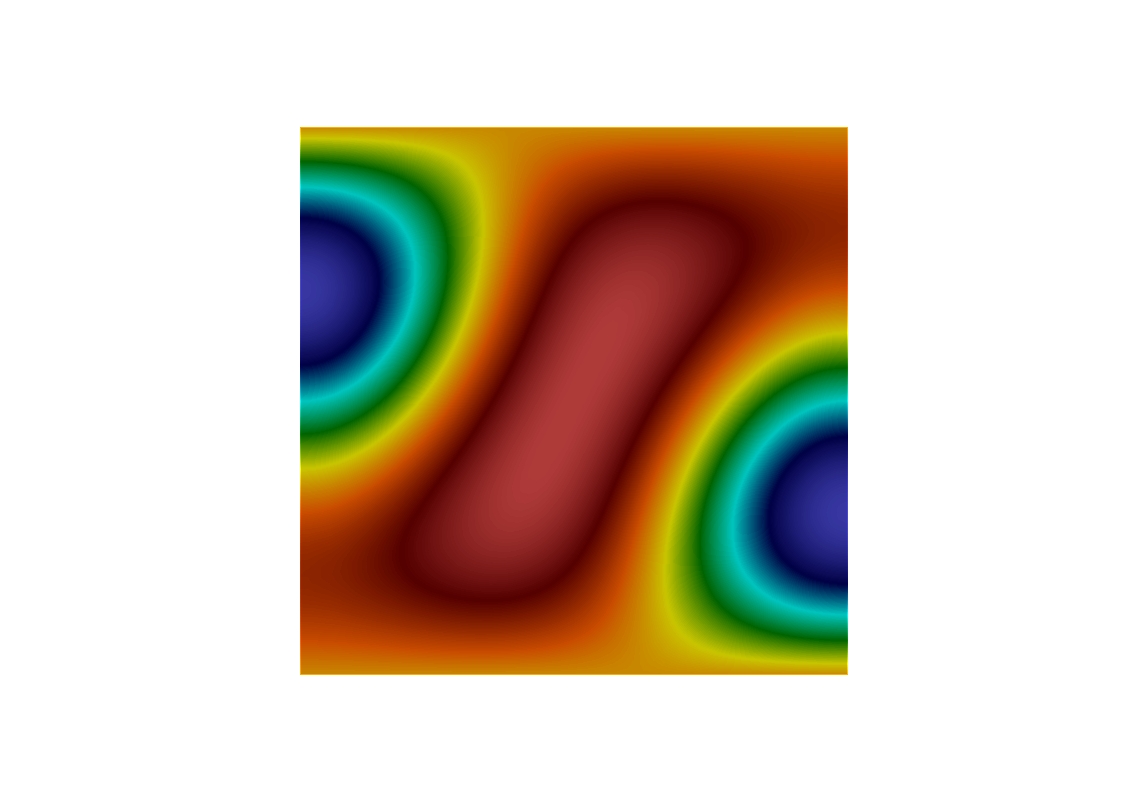}}\\
5 & 49.35 & 5.000 & 8.263e-07 & 49.25 & \raisebox{-.5\height}{\includegraphics[width = .15\textwidth, 
height=0.04\textheight]{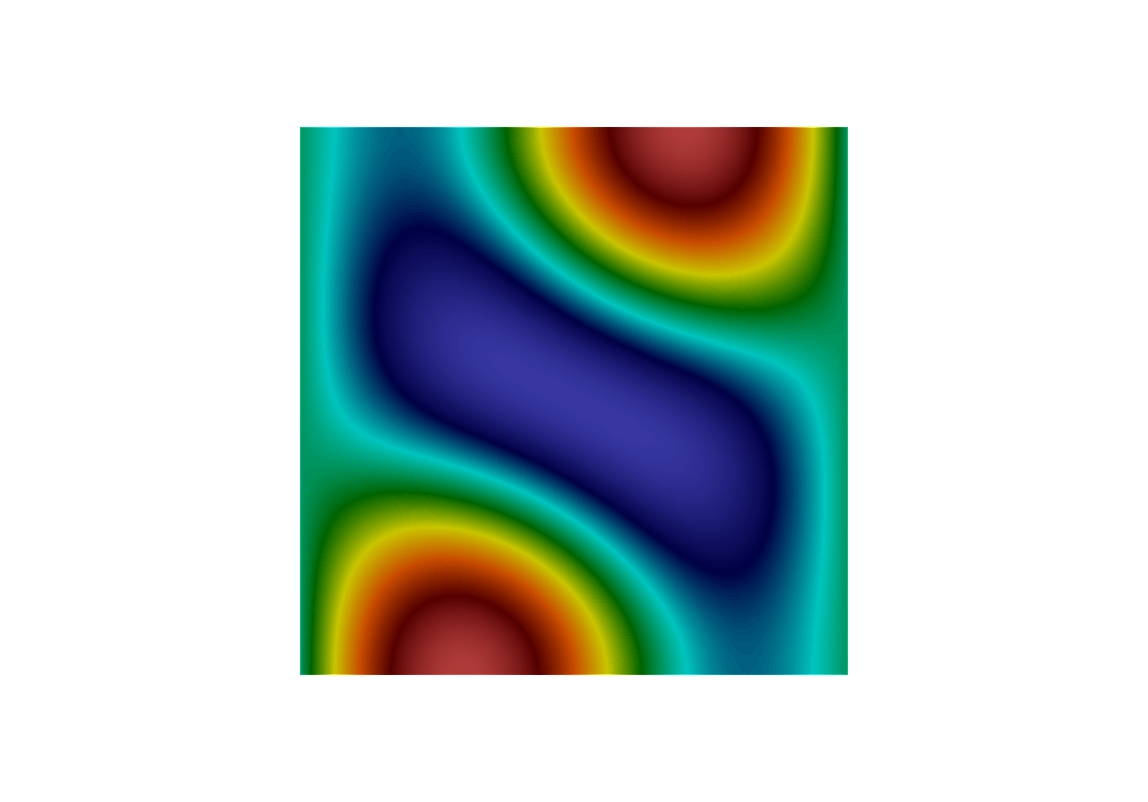}} & 
\raisebox{-.5\height}{\includegraphics[width = .15\textwidth, 
height=0.04\textheight]{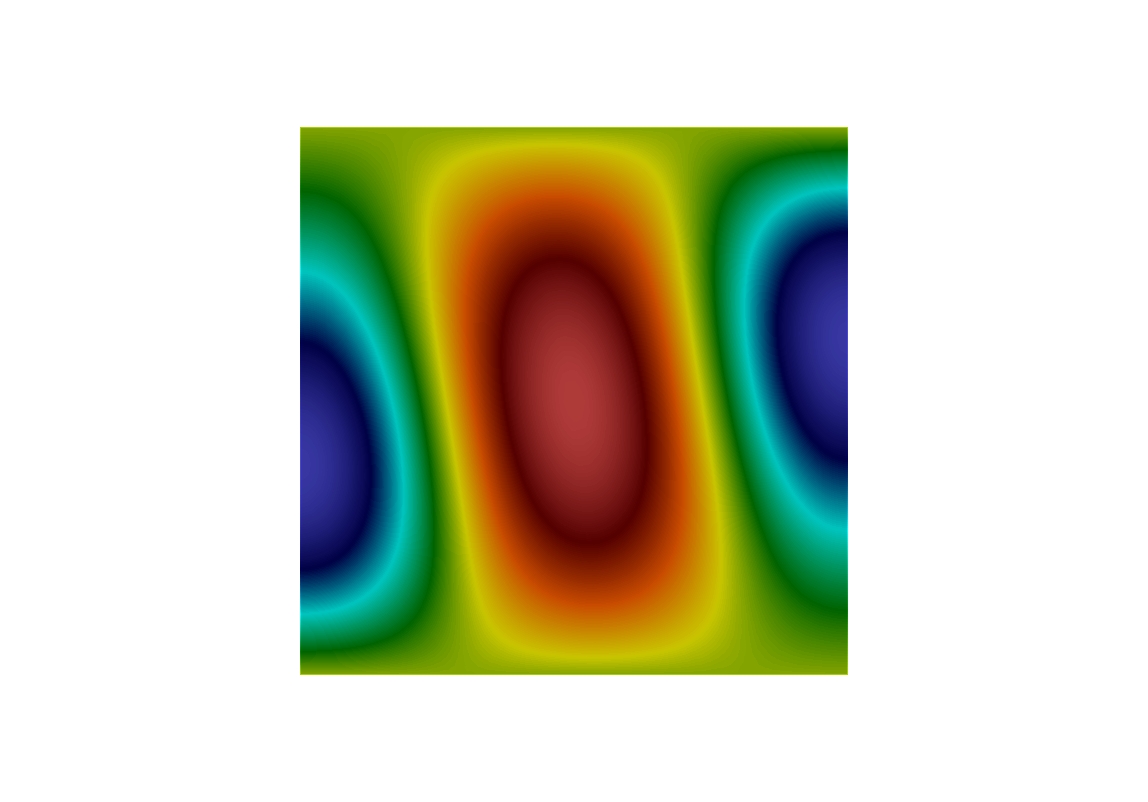}}\\
6 & 78.96 & 8.000 & 19.74 & 0.0005061 & \raisebox{-.5\height}{\includegraphics[width = .15\textwidth, 
height=0.04\textheight]{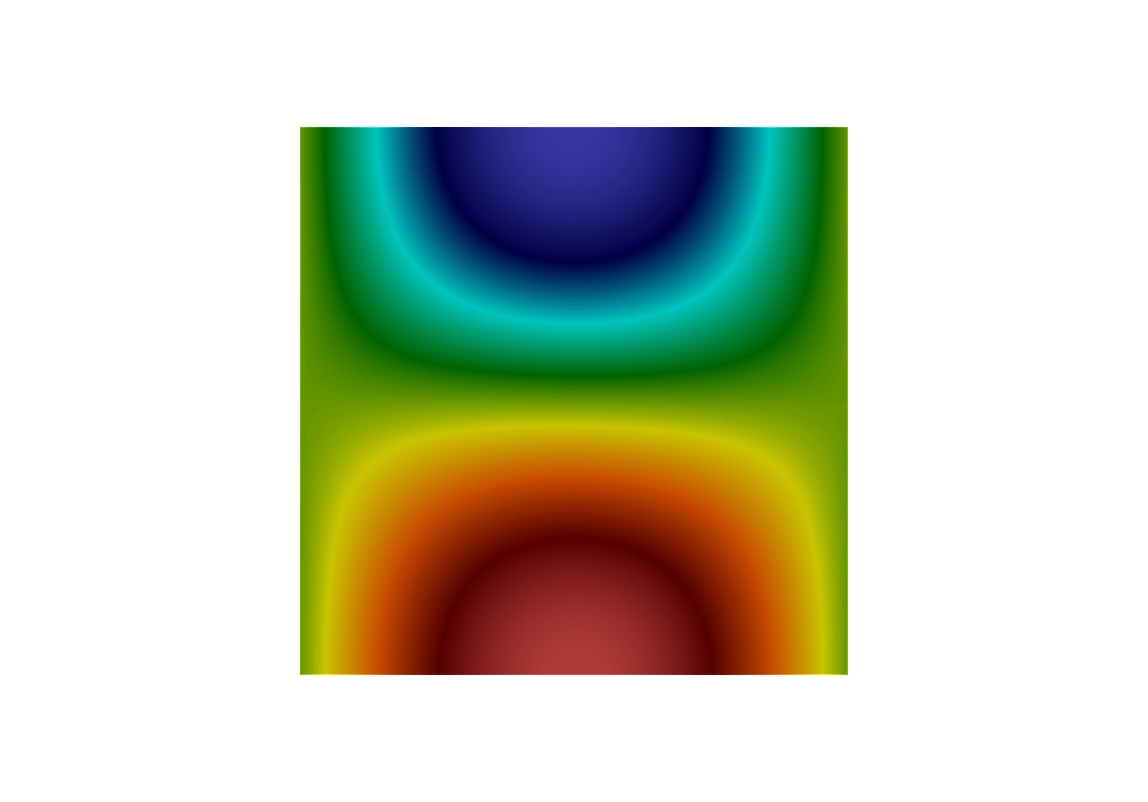}} & 
\raisebox{-.5\height}{\includegraphics[width = .15\textwidth, 
height=0.04\textheight]{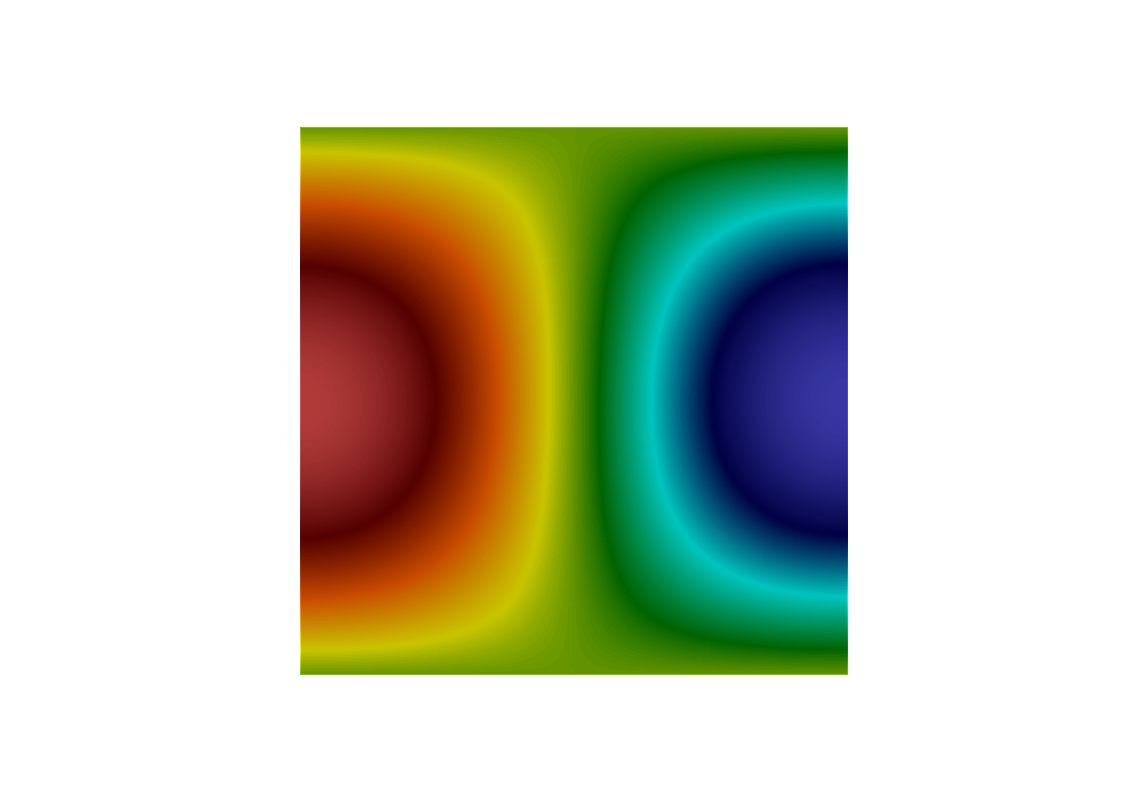}}\\
7 & 78.96 & 8.000 & 2.742e-06 & 78.78 & \raisebox{-.5\height}{\includegraphics[width = .15\textwidth, 
height=0.04\textheight]{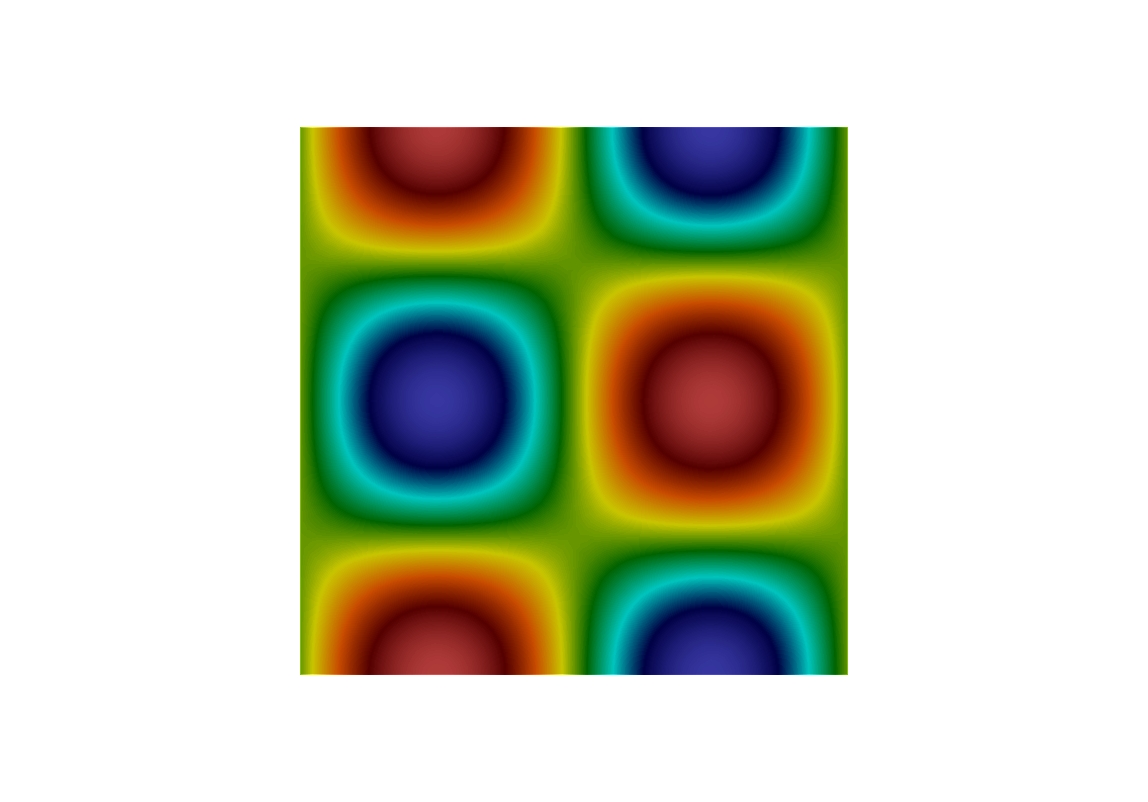}} & 
\raisebox{-.5\height}{\includegraphics[width = .15\textwidth, 
height=0.04\textheight]{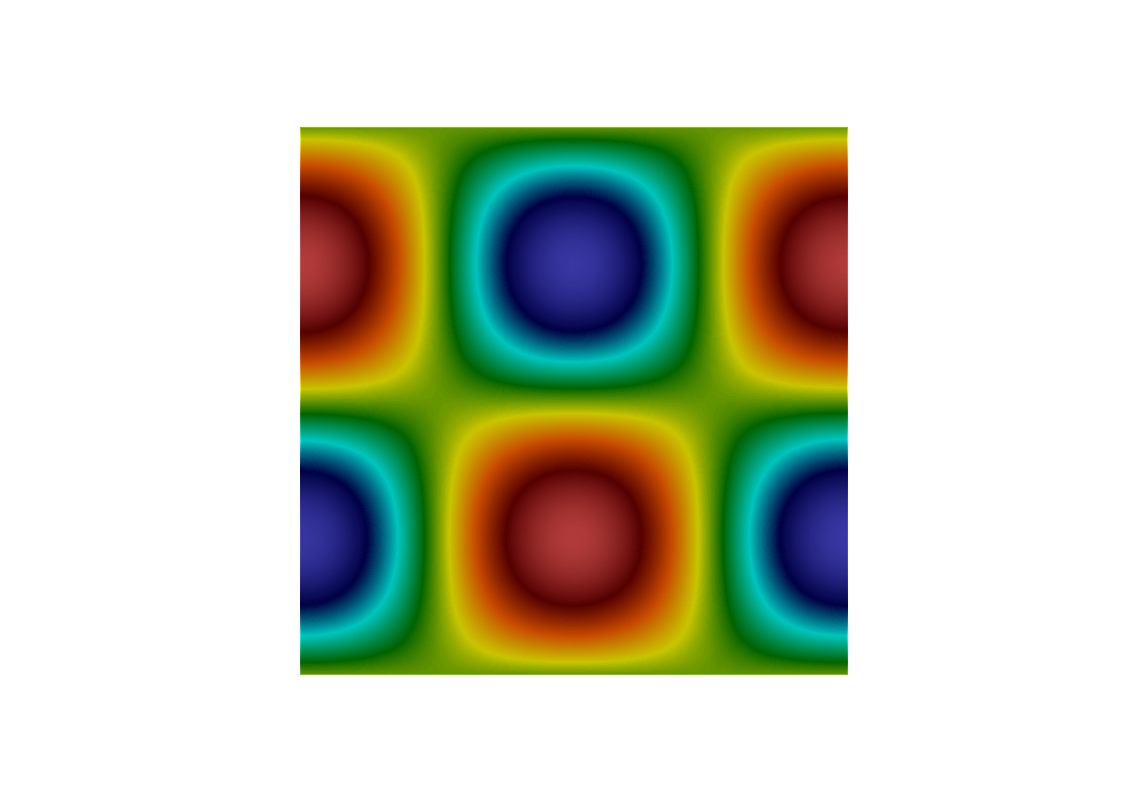}}\\\hline
\end{tabular}\label{table:lambda2square}\caption{First 7 eigenmodes on unit square with parameters $\mu=\rho=1,\lambda=2$. Again, $\nu_2=\nu_3$, and the corresponding eigenspace has pure $p$ modes;  $\nu_4=\nu_5$, and the corresponding eigenspace has pure $s-$ modes. Note that the eigenspace corresponding to $\nu_6$ has an $s$ mode as well as a p-mode, in contrast to \autoref{table:lambda1square}.}
\end{table}}

From the analytic expressions \autoref{eq:pmode} and \autoref{eq:smode} for the rectangle, it is clear that by changing the Lam\'e parameters one can change the multiplicity as well as the composition of eigenspaces. For example, the first eigenmode on a rectangle $[0,1]\times[0,2]$ with $\mu=\lambda=\rho=1$ is a pure $p$ mode; changing $\lambda$ to 10  will make the first eigenmode be a pure $s$ mode. 

 For shapes other than rectangles, the eigenmodes do not need to be pure $s$ or pure $p$ waves. This points out the fact that the shape of the boundary is an important factor in this problem. The L-shaped domain discussed in the convergence studies is an example of such a domain:  due to  the re-entrant corner (see \autoref{table:lshaped}), the eigenmodes are neither pure shear nor pure compression modes. The same can be observed on the isosceles triangle of vertices $(0,0),\, (2,0)$ and $(0,1)$, where none of the eigenfunctions seem to be divergence or curl free (see \autoref{table:triangle}).

\graphicspath{{./images/EigenvectorPlots/}}
\small{\begin{table}[!ht]
\centering
\begin{tabular}{|c|c|c|c|c|c|c|}
\hline
 $j$& $\nu_j$ & $\nu_j/\pi^2$ & $\| \div\, \u\|_0^2$ & $\|\curl\,\u\|_0^2$ & $x-$component & $y-$component 
\\\hline
 1& 0.08848& 0.008965 & 0.01531 & 5.013 & \raisebox{-.5\height}{\includegraphics[width = .15\textwidth, 
height=0.04\textheight]{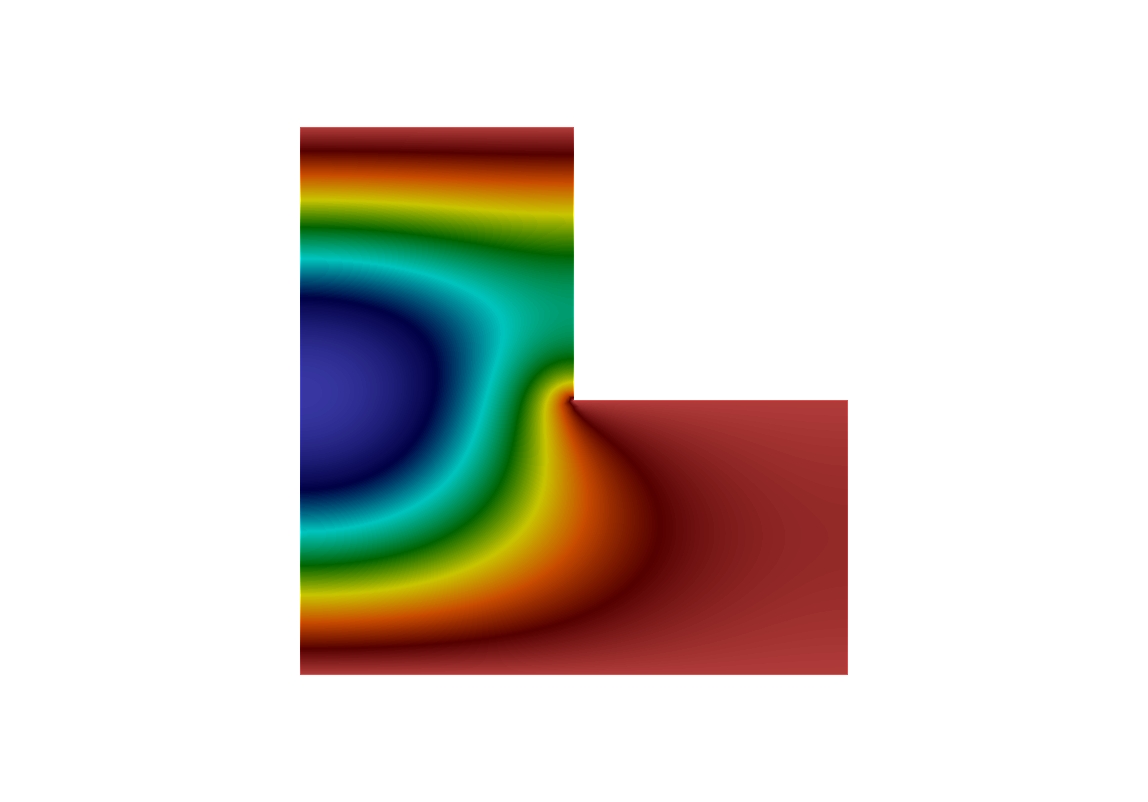}} & 
\raisebox{-.5\height}{\includegraphics[width = .15\textwidth, 
height=0.04\textheight]{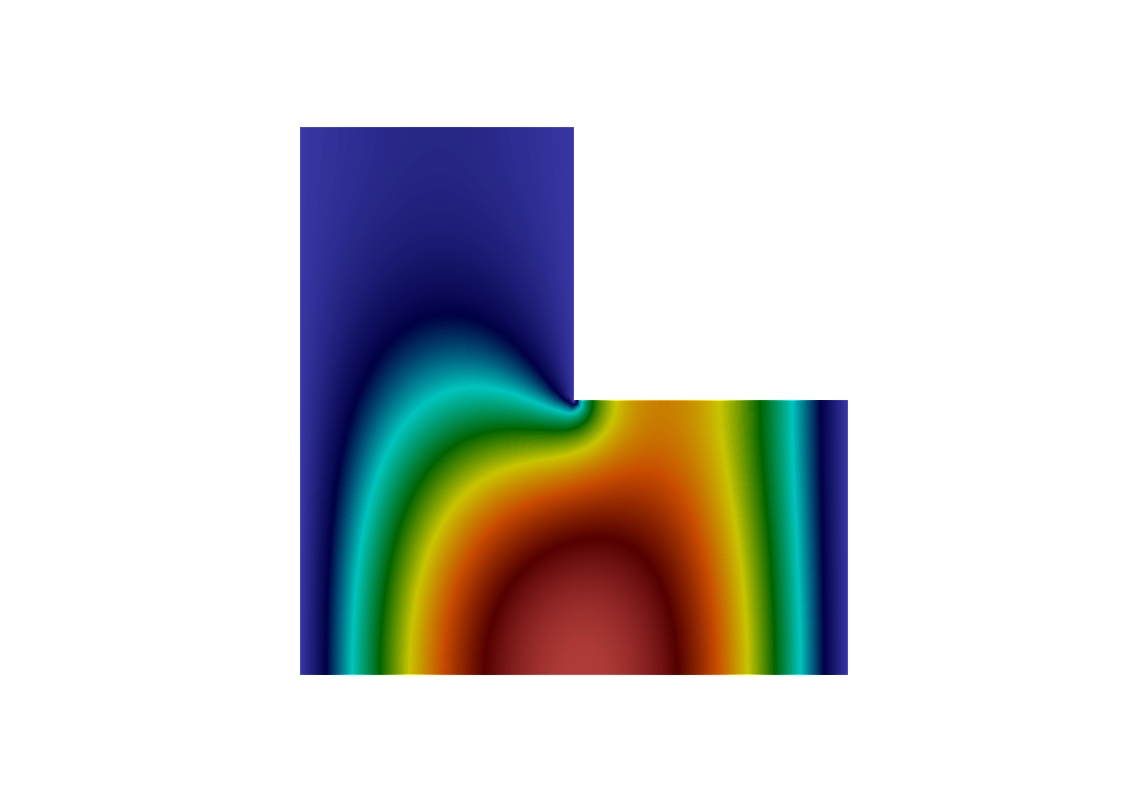}}\\
 2& 1.285& 0.1302 & 0.1411 & 5.273 & \raisebox{-.5\height}{\includegraphics[width = .15\textwidth, 
height=0.04\textheight]{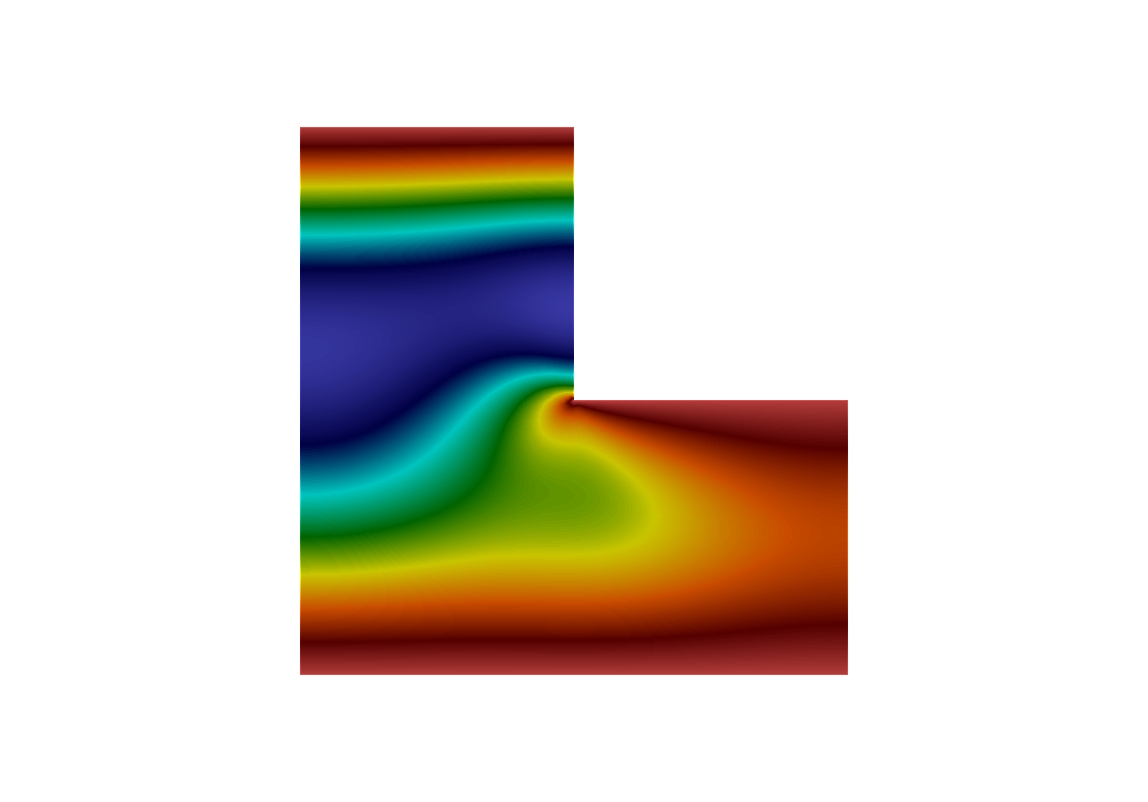}} & 
\raisebox{-.5\height}{\includegraphics[width = .15\textwidth, 
height=0.04\textheight]{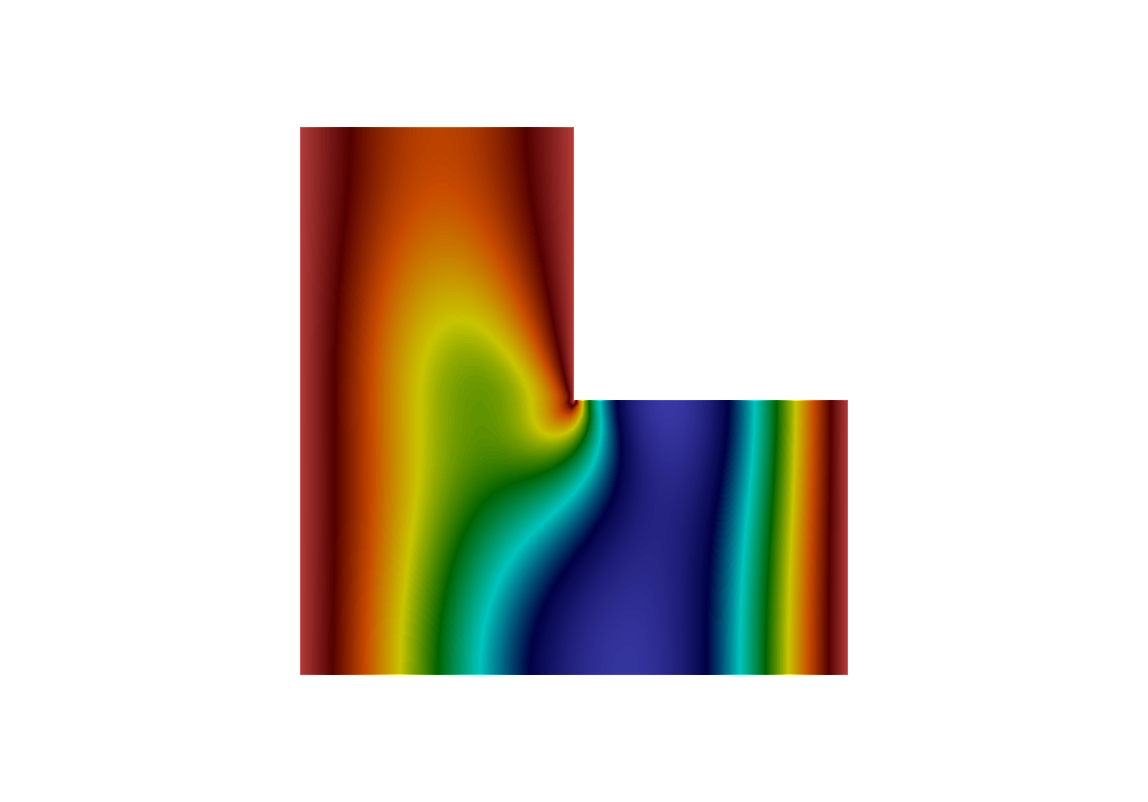}}\\
 3& 3.8100& 0.3861 & 0.2903 & 15.43 & \raisebox{-.5\height}{\includegraphics[width = .15\textwidth, 
height=0.05\textheight]{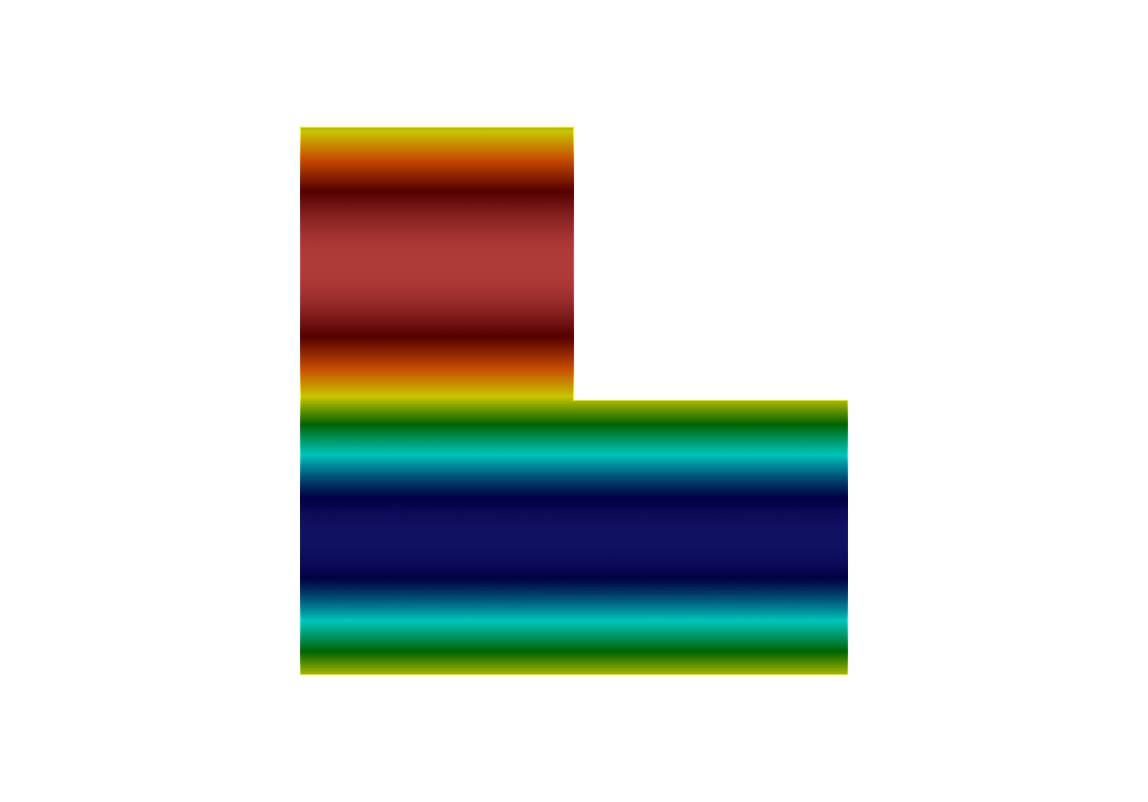}} & 
\raisebox{-.5\height}{\includegraphics[width = .15\textwidth, 
height=0.04\textheight]{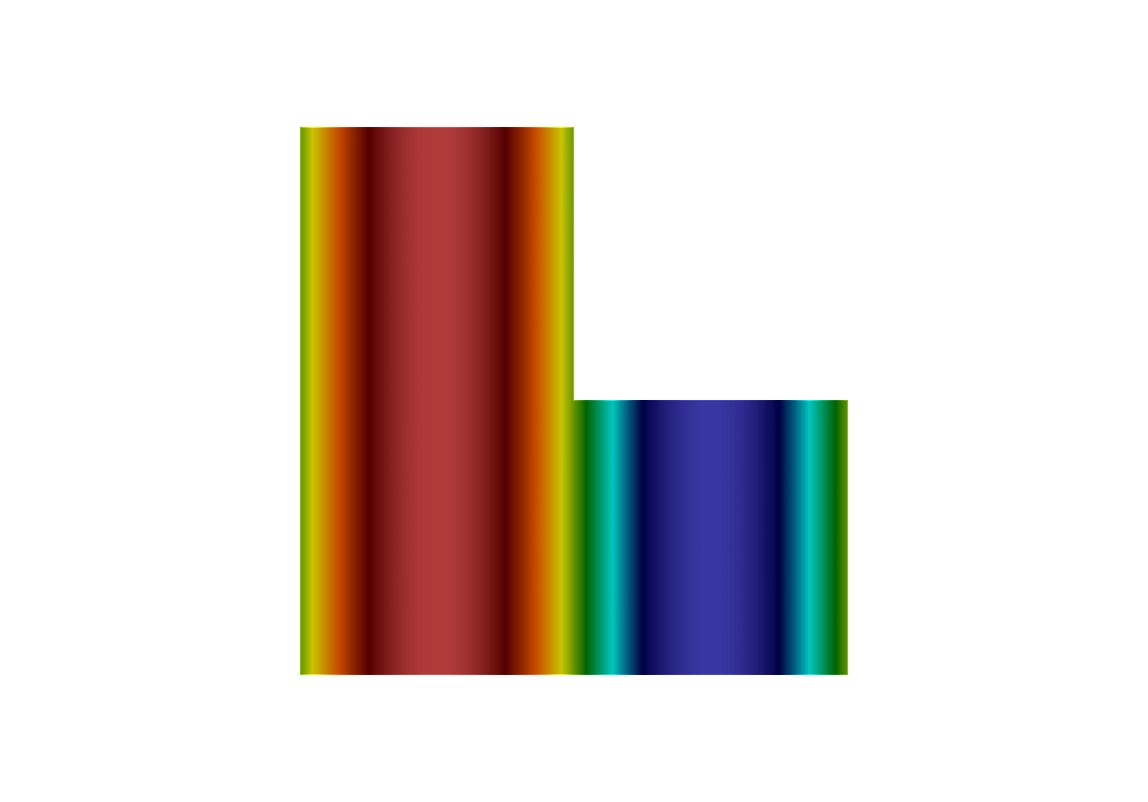}}\\
 4& 5.1300& 0.5197 & 0.4458 & 8.277 & \raisebox{-.5\height}{\includegraphics[width = .15\textwidth, 
height=0.05\textheight]{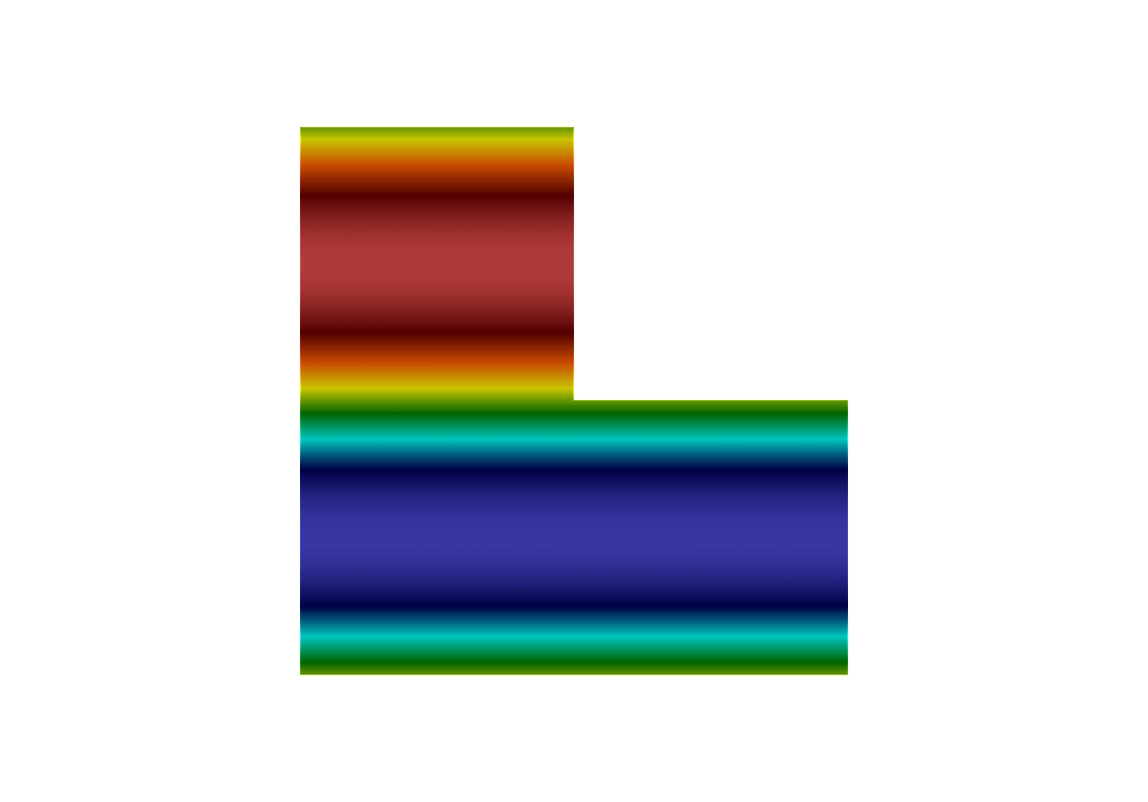}} & 
\raisebox{-.5\height}{\includegraphics[width = .15\textwidth, 
height=0.04\textheight]{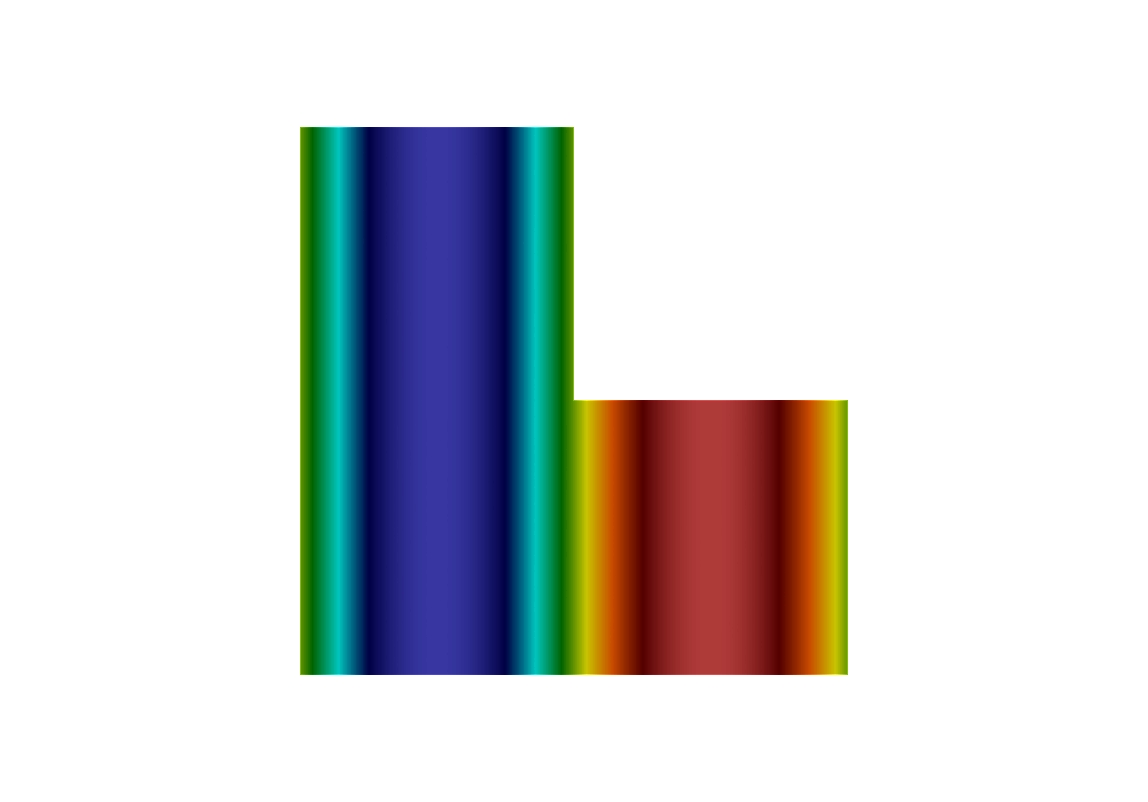}}\\\hline
\end{tabular}\label{table:lshaped}
\caption{L-shaped domain with parameters $\mu=\lambda=\rho=1$.}
\end{table}}

\graphicspath{{./images/EigenvectorPlots/}}
\small{\begin{table}[!ht]
\centering
\begin{tabular}{|c|c|c|c|c|c|c|}
\hline
 $j$& $\nu_j$ & $\nu_j/\pi^2$ & $\| \div\, \u\|_0^2$ & $\|\curl\,\u\|_0^2$ & $x-$component & $y-$component 
\\\hline
 1& 4.6563& 0.4718 & 0.7007 & 24.36 & \raisebox{-.5\height}{\includegraphics[width = .15\textwidth, 
height=0.05\textheight]{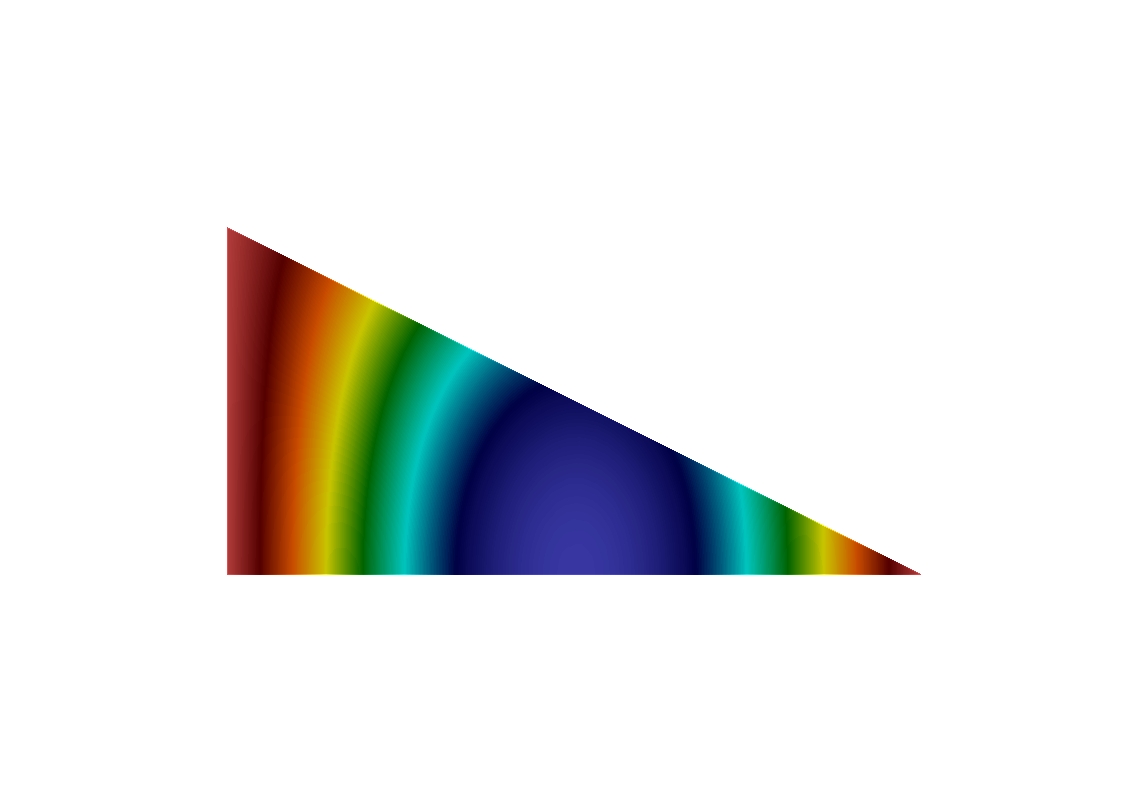}} & 
\raisebox{-.5\height}{\includegraphics[width = .15\textwidth, 
height=0.05\textheight]{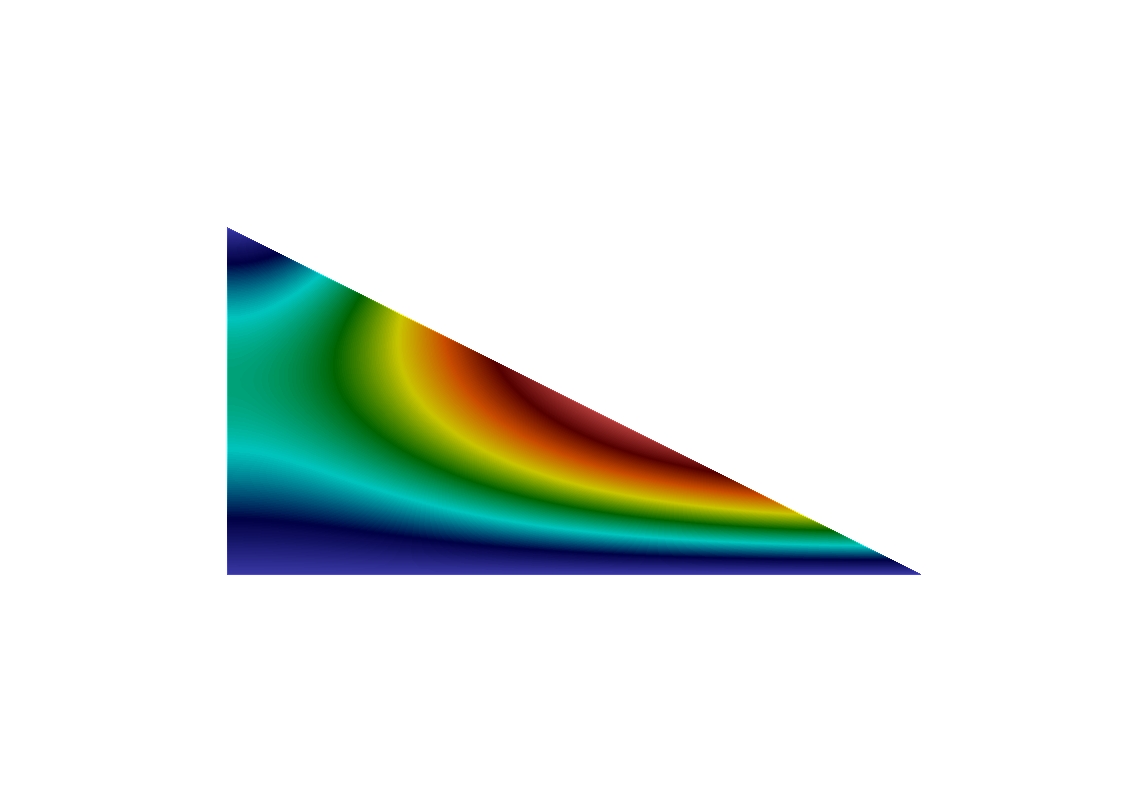}}\\
 2& 8.3125& 0.8422 & 0.4333 & 14.42 & \raisebox{-.5\height}{\includegraphics[width = .15\textwidth, 
height=0.05\textheight]{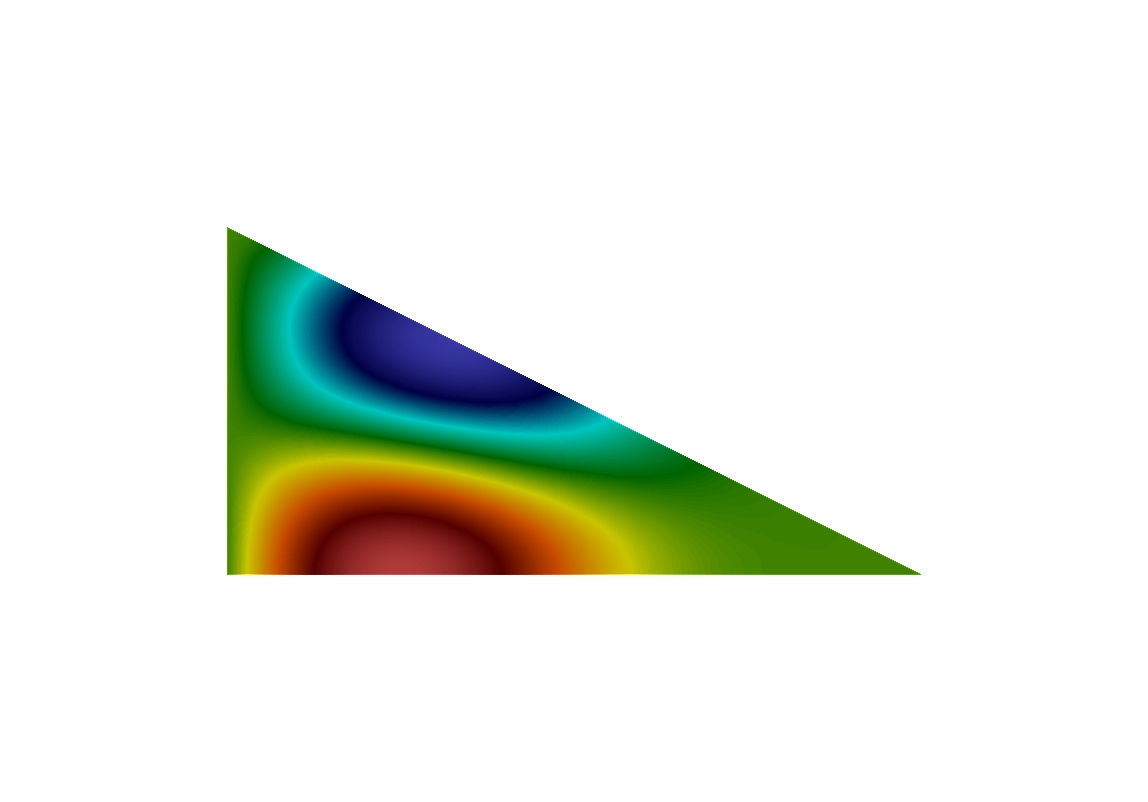}} & 
\raisebox{-.5\height}{\includegraphics[width = .15\textwidth, 
height=0.05\textheight]{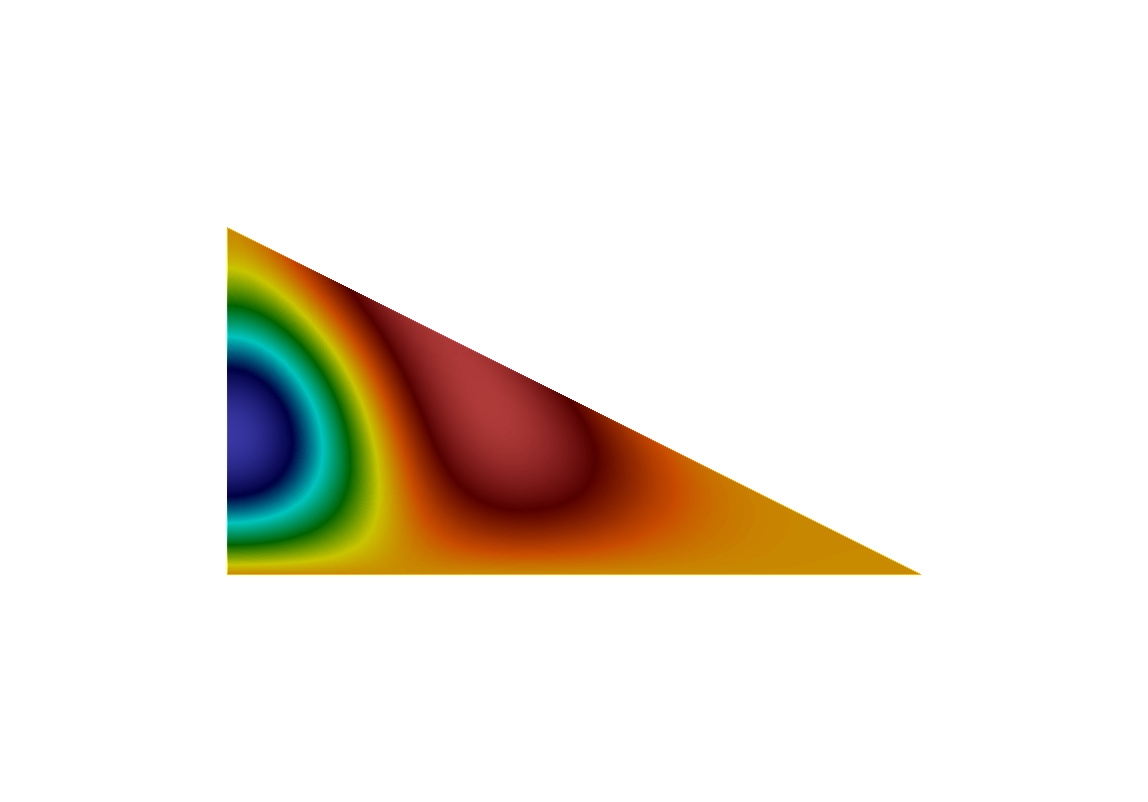}}\\
 3& 11.84674& 1.200 & 2.527 & 4.15 & \raisebox{-.5\height}{\includegraphics[width = .15\textwidth, 
height=0.05\textheight]{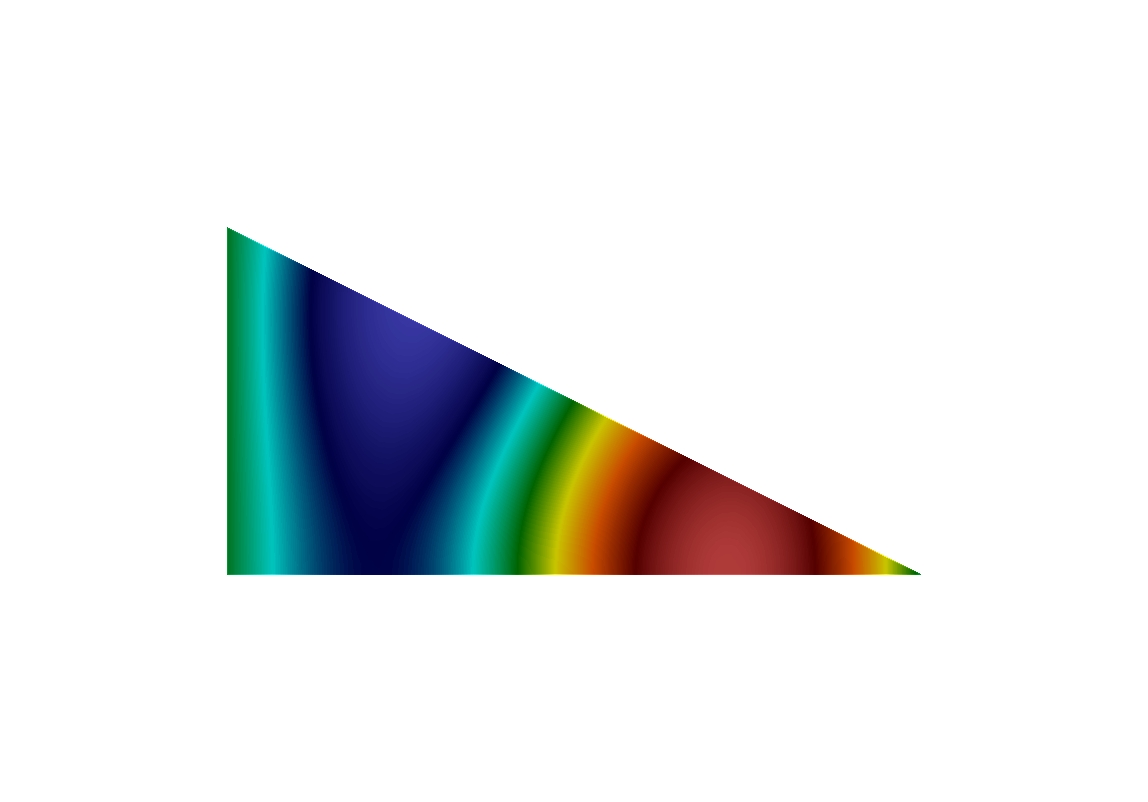}} & 
\raisebox{-.5\height}{\includegraphics[width = .15\textwidth, 
height=0.05\textheight]{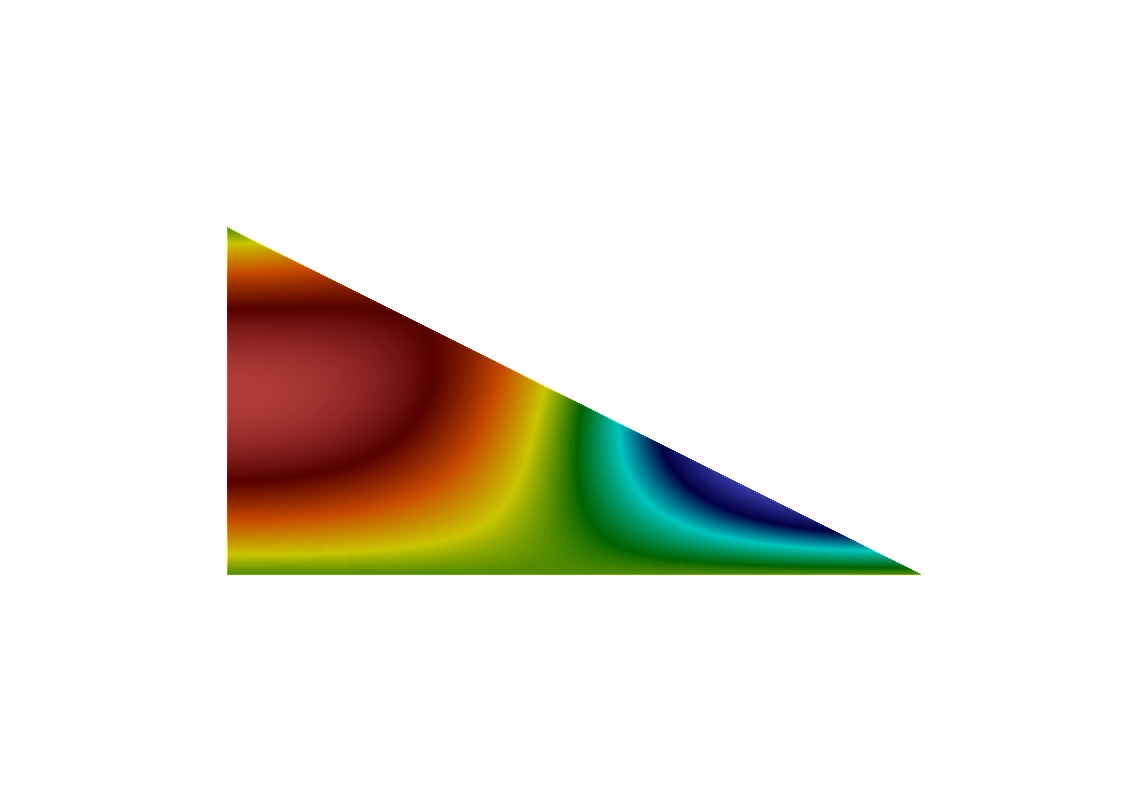}}\\
 4& 21.0647& 2.134 & 1.640 & 75.96 & \raisebox{-.5\height}{\includegraphics[width = .15\textwidth, 
height=0.05\textheight]{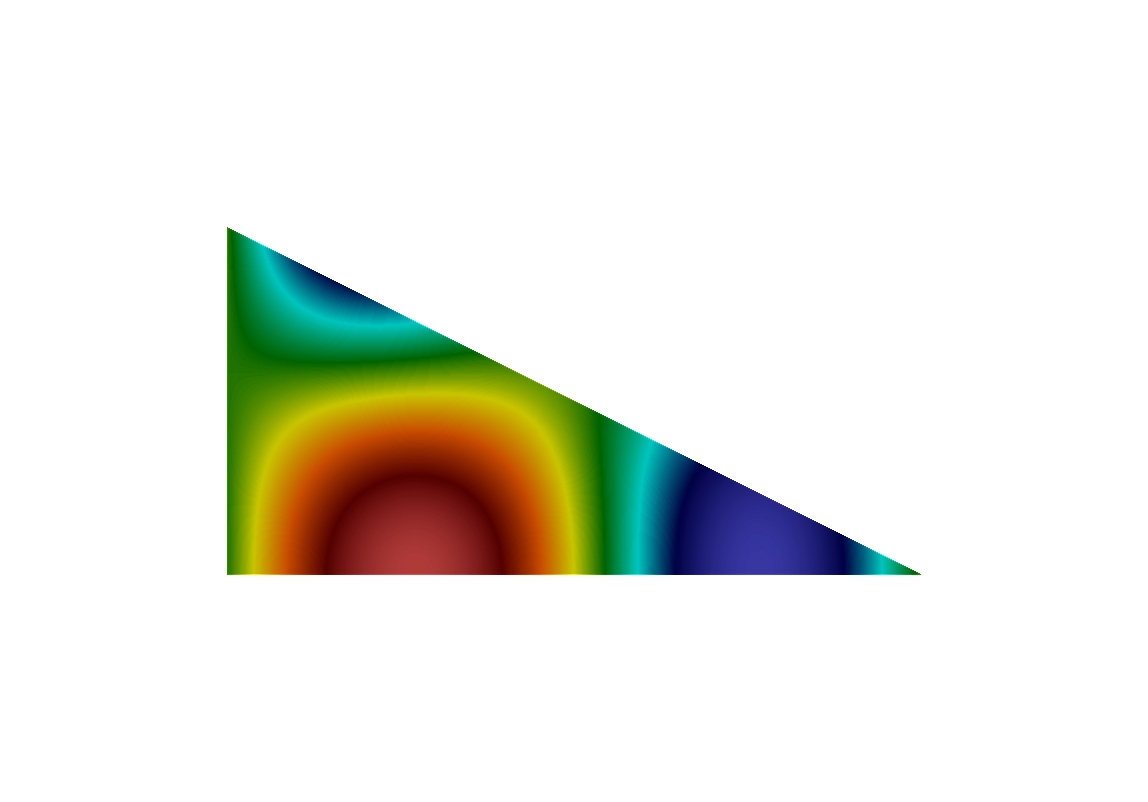}} & 
\raisebox{-.5\height}{\includegraphics[width = .15\textwidth, 
height=0.05\textheight]{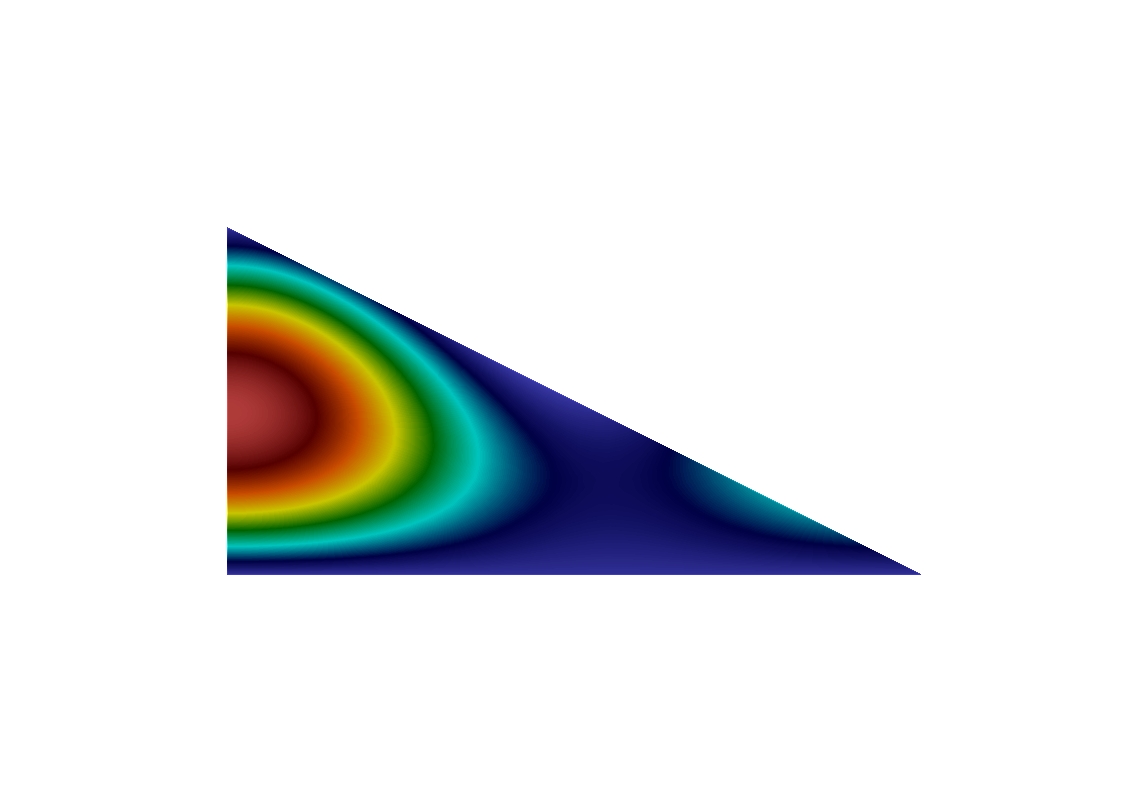}}\\\hline
\end{tabular}\label{table:triangle}
\caption{Isosceles triangle of vertices $(0,0)$, $(2,0)$ and $(0,1)$ with parameters $\lambda=\mu=\rho=1$.}
\end{table}}

\paragraph{Conclusions}\label{sec:conclusions}
In this work we demonstrated the {existence} of Jones modes on Lipschitz domains. The spectrum of 
{the Jones eigenproblem (cf. \autoref{eq:jones-modes})}  depends 
heavily on the shape of the domain under consideration (as shown in \autoref{sec:formulation}). 

{Axisymmetric domains such as bodies of revolution, spheres or disks present more challenges for computation compared to polygonal domains. Concretely, there are two issues in these cases: the presence of a zero eigenvalue, and the discretization of the boundary curve. The former issue is handled by using a shift. The effect of the discretization of the smooth boundary by a polygonal one is more profound. As we saw in the computations for the disk, the primal formulation did not capture the same part of the spectrum. This phenomenon was also found to hold for other smooth domains, such as ellipses.}
		
{We believe the reason for this phenomenon is the interplay between the imposition of the constraint on $\mathbf{u\cdot n}$, and the fact that the boundary is being approximated by piecewise linear polynomials. 
We expect curvilinear elements which are boundary-conforming should ameliorate this problem, and a careful investigation of this is ongoing work.}

\paragraph{Acknowledgements}
Sebasti\'an Dom\'inguez thanks the support of CONICYT-Chile, through Becas Chile. Nilima Nigam gratefully thanks the 
financial support of the National Sciences and Engineering Research Council of Canada (NSERC) Discovery Grants and the hospitality of the Institute for Mathematics and its Applications, Minneapolis. The 
research of Jiguang Sun was partially supported by NSF grant DMS-1521555. {The authors are grateful to the anonymous referees for their helpful comments.}

{This paper is dedicated to George C. Hsiao, Gabriel Gatica and Francisco-Javier Sayas Gonz\'alez (Pancho) on the occassion of their 85th, 60th and 50th birthdays respectively. Initial ideas for this work were discussed at length with Pancho. His friendship, intellectual courage and mathematical talent will be deeply missed.}

\bibliographystyle{plain}
\bibliography{COMPLETE.bib}

\end{document}